\documentclass[12pt, a4]{article}
\usepackage[left=1in,right=1in,bottom=1in,top=1.25in]{geometry}
\usepackage{amsfonts}
\usepackage{amsmath, amssymb}
\usepackage{graphicx}
\usepackage[font=small,labelfont=bf]{caption}
\usepackage{epstopdf}
\usepackage[pdfpagelabels,hyperindex]{hyperref}
\usepackage{xcolor}
\usepackage{amsthm}
\usepackage{float}
\usepackage{pgfplots}
\usepackage{listings}
\usepackage{longtable}
\usepackage{mathrsfs}
\usepackage{hyperref}
\hypersetup{
    colorlinks=true,
    citecolor = red,
    linkcolor=blue,
    filecolor=magenta,
    urlcolor=red,
    pdftitle={Overleaf Example},
    pdfpagemode=FullScreen,
    }
\providecommand{\norm}[1]{\lVert#1\rVert}
\hypersetup{
pdftitle={THE FRACTIONAL STEFAN PROBLEM: GLOBAL REGULARITY OF THE SELFSIMILAR SOLUTION},
pdfsubject={Mathematics, Number Theory, Combinatorics},
pdfauthor={MARCOS LLORCA AND JUAN LUIS VÁZQUEZ},
pdfkeywords={Finite differences}
}

\sffamily

\newtheorem{thm}{Theorem}[section]
\newtheorem{cor}[thm]{Corollary}
\newtheorem{prop}[thm]{Proposition}
\newtheorem{lem}{Lemma}[section]
\newtheorem{rem}{Remark}[section]
\newtheorem{defn}{Definition}[section]
\numberwithin{equation}{section} %Para numerar ecuaciones
\newcommand{\ve}{\varepsilon}

\definecolor{energy}{RGB}{114,0,172}
\definecolor{freq}{RGB}{45,177,93}
\definecolor{spin}{RGB}{251,0,29}
\definecolor{signal}{RGB}{203,23,206}
\definecolor{circle}{RGB}{217,86,16}
\definecolor{average}{RGB}{203,23,206}

\colorlet{shadecolor}{gray!20}
\pgfplotsset{compat=1.9}

\usepgflibrary{fpu}

\definecolor{darkblue}{rgb}{0.05, .05, .65}
\definecolor{darkgreen}{rgb}{0.1, .65, .1}
\definecolor{darkred}{rgb}{0.8,0,0}
\definecolor{ao}{rgb}{0.0, 0.0, 1.0}

\newcommand{\nc}{\normalcolor}

\begin{document}

\title{The Fractional Stefan Problem: Global Regularity \\ of the Bounded \nc Selfsimilar Solution}

%--------Meta Data: Fill in your info------
\author{Marcos Llorca and Juan Luis V\'azquez}
%\email{marcosllorca@gmail.com}
%\keywords{Stefan problem, phase transition, free boundaries, nonlinear
%and nonlocal equation, regularity, fractional diffusion}
%\urladdr{}
%\subjclass[]{80A22, 35B65, 35D30, 35K15, 35K65, 35R09, 35R11}

\date{\today}
\maketitle

\begin{abstract}
We study the regularity of the bounded self-similar solution to the one-phase Stefan problem with fractional diffusion posed on the whole line. In terms of the enthalpy $h(x,t)$, the evolution problem reads
\[
\begin{cases}
\partial_t h + (-\Delta)^s \Phi(h) = 0 & \text{in } \mathbb{R}^n \times (0,T),\\[2mm]
h(\cdot,0) = h_0 & \text{in } \mathbb{R}^n ,
\end{cases}
\]
where $u = \Phi(h) := (h-L)_+ = \max\{h-L,0\}$ denotes the temperature, $L>0$ is the latent heat, and $s \in (0,1)$. We prove that the regularity of the self-similar solution depends on $s$, with a critical threshold at $s = 1/2$. More precisely, in the subcritical case $0 < s < 1/2$, the self-similar solution exhibits at least $C^{1,\alpha}$ regularity, with Hölder exponent $\alpha >0$. In contrast, we show that the enthalpy of the self-similar solution is not Lipschitz continuous at the free boundary in the critical case $s=1/2$, as well as in the supercritical case $1/2 < s < 1$. Additional results  are also established concerning the lateral regularity at the free boundary and the asymptotic behavior of the solution profile as $x \to \pm\infty$\nc.
\end{abstract}

\

\tableofcontents

\vskip 1cm


\

  	%35R11,   	%Fractional partial differential equations
%    35K55,  	%Nonlinear parabolic equations
%   	35K65,   	%Degenerate parabolic equations
%    35C06.      %Self-similar solutions to PDEs
    %35A08,   	%Fundamental solutions to PDEs
    %35B40.   	%Asymptotic behavior of solutions to PDEs.

\medskip

\parskip 4pt
\parindent 10pt

\newpage

\section{Introduction and main results}

\subsection{The Problem}
In this paper we study the regularity of the  bounded selfsimilar solution  of the one-phase Stefan problem with fractional diffusion posed in  the line $\mathbb{R}$. The Fractional Stefan Problem we consider has the form:
\begin{equation}\tag{FSP}
\label{stefan}
\left\{
\begin{aligned}
\partial_t h + (-\Delta)^s \Phi (h) &= 0 \hspace{7mm}\textup{in} \hspace{7mm} \mathbb{R} \times (0,T),\\
h( \cdot,0) &= h_0 \hspace{5.3mm} \textup{on} \hspace{6.2mm} \mathbb{R}.
\end{aligned}
\right.
\end{equation}
In terms of the physical application we have in mind, the nonlinearity is given by the formula
$\Phi(h) = (h -L)_{+} = \textup{max}\{ h-L,0\}$, and the function $h = h(x,t)$ denotes the enthalpy while
$$
u(x,t):= \Phi(h(x,y)) = (h(x,t) -L)_{+}
$$
denotes  the temperature. The constant $L> 0$  represents the \textit{latent heat} at the phase change front and $s\in(0,1)$ is the fractional exponent. We take $0<T \le \infty$ and $h_0 \in L^\infty(\mathbb{R})$. The operator $(-\Delta)^s$ denotes the \textit{fractional Laplacian} acting on the space variables, with parameter $s\in(0,1)$, defined by
\[
    (-\Delta )^s u(x) = c_{n,s} \,\textup{p.v.}\int_{\mathbb{R}} \frac{u(x) - u(y)}{|x-y|^{n+2s}} \hspace{1mm}dy,
 \]
 for some normalizing constant $c_{n,s} > 0$, and p.v. denotes the principal value,  cf. \cite{ADV}\nc.
Without loss of generality in this paper we will set $c_{n,s} =1$ to avoid dealing with unnecessary constants\nc.
We want to study the properties of the special solution that takes on initial data of the form of a step function
 \begin{equation}\label{IC}\tag{IC}
h_0(x):=\left\{
\begin{aligned}
 &L + P_1 \hspace{10mm}\textup{if} \hspace{7mm} x \leq 0, \\
&0 \hspace{19.7mm}\textup{if} \hspace{7mm}x > 0,
\end{aligned}
\right.
\end{equation}
with $P_1 > 0$ a constant. Our study uses as a starting point the results of paper \cite{DEV1}, which establishes the basic theory. Specifically, it proves the existence of a unique self-similar solution of \eqref{stefan}, of the form
 \begin{equation}\label{h.init}
h(x,t) = H(xt^{-\frac{1}{2s}}),
 \end{equation}
where $H$ is bounded, $0<H<L+P_1$, and satisfies the equation
\begin{equation}\label{eq.ssh}
     -\frac{1}{2s}\xi \cdot \nabla H(\xi) + (-\Delta)^s U(\xi) = 0 \qquad \textup{in} \qquad \mathcal{D}'(\mathbb{R}).
 \end{equation}

Let us note that the equation is invariant under the linear scaling $h=A\,\widehat h$ , $u=A\,\widehat u$, $\widehat L=L/A$, for any $A>0$. It follows  that, without loss of generality, the latent heat can be taken as $\widehat L=1$, and then
the equation holds for $\widehat h$ with initial data $\widehat h_0(x) = 1+ (P_1/L)$ for $x<0$. Thus, $\widehat P_1=P_1/L$  is the only essential parameter that matters in the results that follow besides the  fractional exponent $s$.

\smallskip

\noindent{\bf Historical perspective.} We recall that when $s = 1$ we recover the classical Stefan problem, where diffusion is governed by the standard Laplace operator, $\partial_t h =\Delta \Phi (h)$. A rich theory has been developed for this problem over the years, starting with the pioneering work of Lam\'e and Clapeyron (\cite{LC}, 1831) and later Stefan (\cite{Stefan}, 1889). Since then, enormous progress has been achieved by many distinguished specialists. The bounded self-similar solution is then explicit,  due to Neumann \cite{Ne}, 1883. It uses the error function, has a free boundary of the form $X(t)= C\,t^{1/2}$, and has played a role in the theory and the applications of the Stefan model.

Together with the obstacle problem, the classical Stefan problem has served as a fundamental benchmark for studying the mathematics of phase transitions and free boundary problems, lying at the interface of analysis, geometry, and computation. The ensuing literature is vast, both in theory and applications, and references to the main papers and monographs are readily available. Let us simply mention \cite{Ka, Me, Ru}, the last of which contains an extensive historical account. Also noteworthy is the comprehensive bibliography on the Stefan problem and related topics compiled by Tarzia \cite{Ta}, which includes more than 5{,}000 entries.  Over the past half-century, Caffarelli has established fundamental results on the regularity of solutions and free boundaries for the Stefan problem and related models \cite{Ca1, Ca2}, and these techniques continue to be refined today; see, for instance, the work of Figalli et al. \ \cite{FRS}.

\smallskip

\noindent{\bf Preliminaries of our problem.} Here we consider the fractional model  \eqref{stefan} proposed in  \cite{3, 10}. It describes a substance that exists in two phases (typically called ice and water) and undergoes a phase transition at a prescribed temperature, here normalized to $u = 0$. We focus on the \textit{one-phase} Stefan problem, which represents a phase change in which the ice is maintained at zero temperature.
The level $h = L$ represents the so-called \textit{phase-change level} and separates the water region  (namely, the set of points where $u > 0$ or $h > L$)  from the ice region, where $u = 0$ or $h < L$.
The basic theory for this problem was developed by the authors of \cite{DEV1}. Earlier work on the formulation of the model as a boundary heat control problem was carried out in \cite{AtCa}, where the authors established continuity of the temperature. Continuity of the temperature in the classical Stefan problem was proved in the 1980s in \cite{CaEv}.

\smallskip

Since the precise regularity of solutions and free boundaries for this problem is mostly unknown, our goal is to provide both positive and negative results by analysing a relevant class of self-similar solutions with bounded data, which correspond to the Neumann solutions of the classical Stefan Problem. The motivation for this restricted study lies in the prominent role that self-similar solutions usually play in the theory of diffusion equations (and more generally, in many nonlinear elliptic and parabolic equations), where they often serve as indicators of key qualitative features that extend to  broader classes of solutions.

Our study uses as a starting point  paper \cite{DEV1}, which establishes the basic theory. Specifically, it proves the existence of a unique self-similar solution of \eqref{stefan} with initial data \eqref{IC}. It has the selfsimilar form given by formula \eqref{h.init}
with $P_1 > 0$ a constant.  We also put $U=\Phi(H)$ as selfsimilar profile for the temperature. It was  proved  that this selfsimilar $h$ exhibits a phase-change front located at  $ x(t) = \xi_0\,t^{\frac{1}{2s}}$, corresponding to $H(\xi)=L$, $U(\xi)=0$\nc.  It is usually called the \textit{free boundary} and separates the ice region, $x>x(t)$, from the water region, $x<x(t)$. Moreover, the regularity of $H(\xi)$ in the water region $\xi<\xi_0$ is shown to be  $C^{1,\alpha}$ for some $\alpha \in (0,1)$ while the regularity in the ice region $\xi>\xi_0$ is $C^\infty$.
The value of the free boundary constant is always positive, $\xi_0(s)>0$ for all $0<s<1$. See a list of the proven  properties in Theorem  \ref{teo 4} below. The authors do not prove any regularity of the solution near the free boundary point, beyond continuity. This is the main open problem that we want to address  here.

Figure \ref{fig:Comparison.graph} shows the evolution of $h$ for data $L=P_1=1$ and different values of  $s$. This numerical evidence is taken from \cite{DEV1} and has been a motivation our paper.

\begin{figure}[ht!]\center
\includegraphics[width=\textwidth]{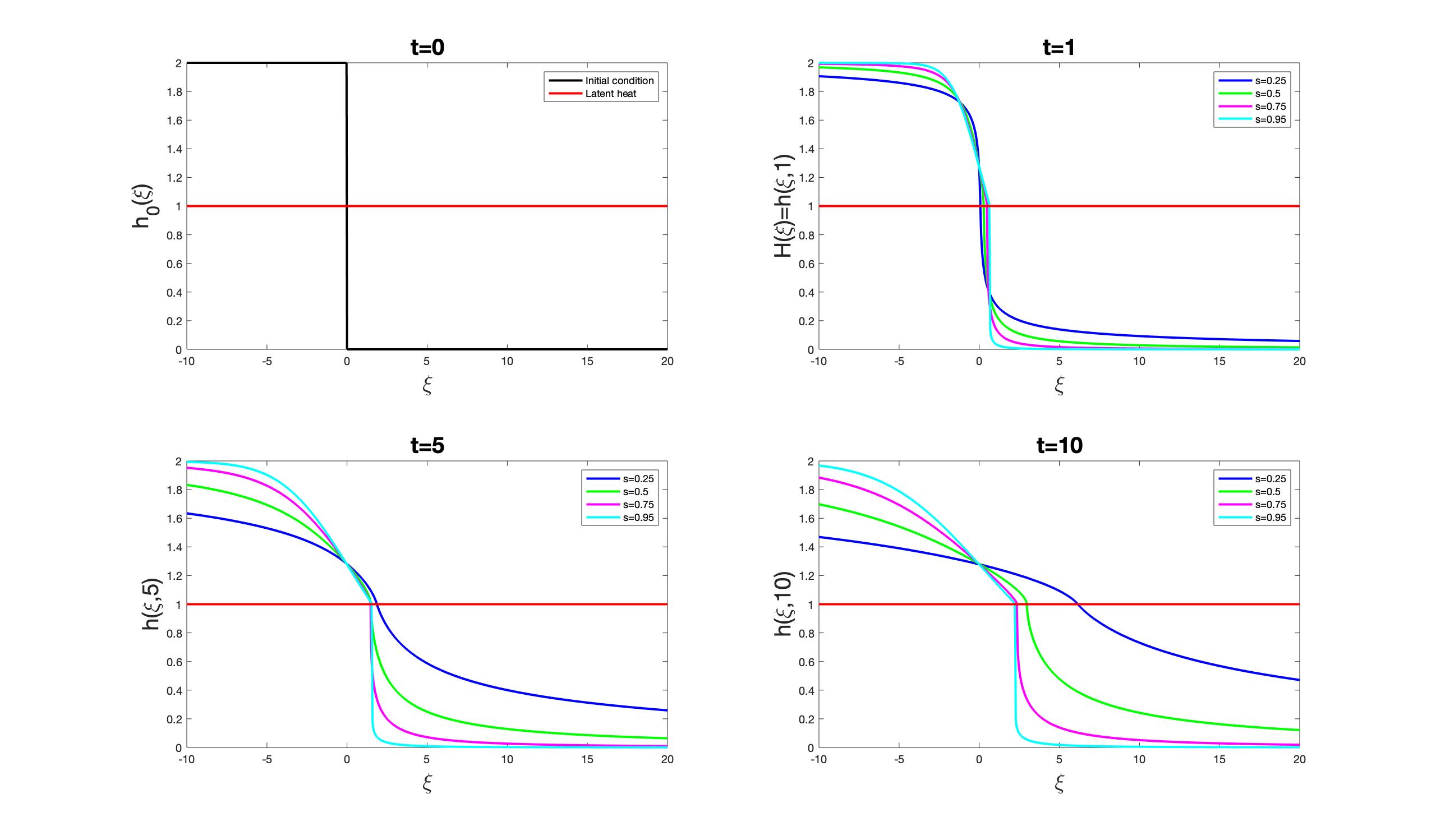}
\caption{Initial  data and profile $h(\xi,t)$ at various times, $L=P_1=1$}
\label{fig:Comparison.graph}
\end{figure}

%%\newpage
%%%%%%%%%%%%%%%%%%%%%%%%%%%%%%%%%%%%%%%%%%%%%%%%%%%%%%%%%%%%%%%%%%%
\subsection{\bf Main Results}
We review the main topics of the paper and state some of the main results for ease of reference.

\noindent {\bf  Regularity in the subcritical case.}  Our main goal is to study the regularity of $H$ near the free boundary point $\xi_0>0$. We find that the regularity of the selfsimilar solution depends on $s$ and actually three regimes appear. The first result we obtain  covers the so-called subcritical regime \nc $s<1/2$ and reads as follows:

\begin{thm}\label{principal}
    Let $s \in (0,\frac{1}{2})$ and let $H$ be the selfsimilar solution of \eqref{stefan} with initial data given by \eqref{IC}. Then,
\begin{itemize}

 \item[a)] Global regularity: \nc $H \in {C}^{1,\alpha}(\mathbb{R})$ for some  $\alpha>0$. Moreover, $H$ satisfies the estimate
\begin{equation}
\label{prin prin}
     \norm{H}_{C^{1,\alpha}(\mathbb{R})}\leq C_s\left(\norm{H}_{L^\infty(\mathbb{R})} + \norm{H}_{C^{0,1}(\mathbb{R})}\right),
\end{equation}
 where $C_s > 0$ is a constant that depends on $s$ and $\alpha$.

\noindent Local regularity: The H\"older exponent $\alpha = 1-2s$ is optimal at  $\xi_0$. We also find local regularity with an exponent  $\alpha^* > 1-2s$  in  all intervals disjoint with  $\xi=0$ and $\xi=\xi_0$\nc.

\item[b)] At the free boundary point we have the Fractional Stefan Formula
\begin{equation}\label{eq1.2}
-H'(\xi_0) = \frac{2s}{\xi_0}\int_{-\infty}^{\xi_0} \frac{U(y)}{|\xi_0 - y|^{1+2s}} \hspace{1mm}dy > 0.
\end{equation}

 \item[c)] Moreover,  $H''(\xi)>0 $ everywhere in the ice interval $(\xi_0,\infty)$ and there is a constant $C_1>0$ such that
\begin{equation}\label{eq1.H''}
    H''(\xi)\,|\xi-\xi_0|^{2s}\ge C_1 \quad \mbox{ for all $\xi>\xi_0$ with $|\xi-\xi_0|$ small.}
\end{equation}
\noindent Consequently,  $H''( \xi_0^+\rm)=+\infty$. \nc
\end{itemize}
\end{thm}
Therefore, for $0<s<1/2$ the solution $H$ cuts the zero temperature level $U=0$ in a regular and transversal way.
The notation  $C^{1,\alpha}(\mathbb{R})=C_b^{1,\alpha}(\mathbb{R})$ denotes the H\"older space consisting of the functions that are bounded with bounded derivatives having a  bounded  $C^{\alpha}$ seminorm in all of $\mathbb{R}$.   The space $C_b^{1,\alpha}(\mathbb{R})$ has a standard norm. Likewise $C^{0,1}(\mathbb{R})$ are the bounded functions with bounded Lipschitz constant. We refer to this kind of uniform bounds as global bounds. The notation $f(\xi_0^+)$ means the limit of $f$ as $\xi\to\xi_0$ with $\xi>\xi_0$. Similarly for $f(\xi_0^-)$. As a complement to item (a) of the Theorem, we recall that $H\in C^{\infty}$ in the open interval $({\xi_0,\infty})$. Finally, the result \ref{eq1.2} is an important control on the singular behaviour at $\xi_0$; it also holds for $1/2\le s <1$\nc.

Proving the  regularity of the solution stated in the theorem implies a  delicate analysis near the free boundary. We want to emphasize that the global regularity is strongly influenced by the regularity of the temperature on the left-hand side of the free boundary, i.e. at $\xi_0^-$, that determines the regularity of the enthalpy on the right-hand side, i.e., at $\xi_0^+$. In the case $s\in (0,1/2)$ we obtain  matching $C^1$ regularity  on both sides.

\normalcolor

In order to prove Theorem \ref{principal} we proceed as follows. First, we define the regularized Fractional Stefan Problem, changing in \eqref{stefan} the function $\Phi$ into a strictly increasing $C^\infty$ function like the  $\Phi_\varepsilon$ described in \eqref{60}. The advantage of the regularized problem is that the solutions are $C^\infty$ and $C^{0,1}$ in $\mathbb{R}$, though the norms of the derivatives may depend on $\varepsilon$. These two properties then allow us to prove the crucial $C^{1,\alpha}$ estimate \eqref{sec 4 ine}. After that we  combining estimate \eqref{sec 4 ine} and the interpolation inequality \eqref{87} to prove estimates for the $C^{0,1}$ and $C^{1,\alpha}$ norms independent of $\varepsilon$. This allows to prove the Lipschitz regularity of the limit solution $H$. Finally, taking the limit in  estimate \eqref{sec 4 ine} we will conclude the $C^{1,\alpha}$ regularity of $H$, i.e., the $C^{\alpha}$ regularity of $H'$ with finite global bounds in all of $\mathbb{R}$.\nc

\medskip

\noindent {\bf Critical and supercritical cases.} The regularity  results for $H(\xi)$ are less positive when $s\ge 1/2$. By the results of \cite{DEV1} we already know that  $H$ is continuous, bounded and monotone in $\mathbb{R}$. Moreover, $H\in C^{1,\alpha}(-\infty,\xi_0)$ and $H\in C^{\infty}(\xi_0,\infty)$. But the $C^1$ regularity at $\xi_0$ breaks down for the two cases $s=1/2$ and $1/2<s<1$.
This is due to the fact that the derivative of the enthalpy is shown to blow up at the free boundary when approaching it from the ice region. This result depends in turn on the existence of an interesting linear subsolution on the left-hand side of the free boundary, a kind of boundary Hopf estimate that we call the ``wedge lemma''. Summing up, we prove that even if the cut at $\xi_0$  is not Lipschitz,  it is  still transversal. We treat the two cases separately for the sake of clarity\nc. Here is the negative result in the critical case $s=1/2$.

\begin{thm}\label{critical}      Let $s = 1/2$ and let $H$ be the selfsimilar solution of \eqref{stefan} with initial data given by \eqref{IC}. Then,
\begin{equation}
\liminf_{\xi\to\xi_0^+}\frac{|H'(\xi)|}{|\log(\xi-\xi_0)|} \geq C>0.
\end{equation}
In particular, $\lim_{\xi\to\xi_0^+} H'(\xi)=-\infty$. Consequently, $H \notin C^{0,1}(\mathbb{R})$. Actually, \eqref{eq1.H''} holds.
    \normalcolor
\end{thm}

 The precise formula is proved in Section \ref{eleven}. \nc Finally, for the supercritical case we will prove the next result  about limited regularity:

\begin{thm}\label{supercritical}
    Let $s \in (1/2,1)$ and let $H$ be the selfsimilar solution of \eqref{stefan} with initial data given by \eqref{IC}. Then,
\begin{equation}
\liminf_{\xi\to\xi_0^+} |H'(\xi)|\,(\xi-\xi_0)^{2s-1} \geq C>0,
\end{equation}
for some constant $C$ depending on $s$. In particular $\lim_{\xi\to\xi_0^+} H'(\xi)=-\infty$.  Actually, \eqref{eq1.H''} holds. It follows that $H\not \in C^{0,\alpha}$ for any $\alpha\ge 2-2s$ \nc near the free boundary $\xi_0$.
\end{thm}

The result is proved in Section \ref{section 13}. Note that, as $s\to 1^-$, the above result ensures that the enthalpy does not belong any more to a given H\"older space. This is coherent with the fact that in the classical Stefan problem  involving the standard Laplace operator (i.e., the case $s=1$) the enthalpy is discontinuous at the \nc free boundary with a jump discontinuity of value $H(\xi_0^-)-H(\xi_0^+)=L$, where $L$ denotes the latent heat.

%%%%%%%%%%%%%%%%%%%%%%%%%%%%%%%%%%%%%%%%%%%%%%%%%%%%%%%%%%%%%%%%%%%

\medskip

\noindent\textbf{Structure of the work}

The paper is divided into three blocks corresponding to each of the cases that we study: the subcritical case occupies most of our attention, from Section \ref{section 3} to Section \ref{sec.nine}.
In Section \ref{section 2} we introduce some preliminary results that will be useful. In Section \ref{section 3} we study the regularized Fractional Stefan Problem and we prove existence, uniqueness, selfsimilarity and $C^{1,\alpha}$ estimates of the solutions of the regularized problem. In Section \ref{section 4} we prove an improvement $C^{1,\alpha}$ estimate for the solutions to the regularized problem. Section \ref{section 5} is devoted to prove Theorem \ref{principal}.  The  proof relies on \nc the $C^{1,\alpha}$ estimate of the selfsimilar solutions to the regularized problem, a key point of this paper.  In Subsection \ref{section 7} we prove further regularity in the water region, $(-\infty, \xi_0)$\nc. In Section \ref{section 6}, we study the lateral regularity on both sides of $\xi_0$. We will show that $H' \in C^{1-2s}[\xi_0, \infty)$ but not better, in particular $H \notin C^2(\mathbb{R})$.

Changing topics, Section \ref{sec.nine} is devoted to the asymptotic behaviour of $H$ and $H'$ near $-\infty$.  Suitable notation is introduced. We note in that the difficult estimates are those in the water,  as $\xi\to \infty$, since the limits as $\xi\to \infty$ are easier.

Finally, in Sections \ref{eleven}-\ref{section 13} we prove that the selfsimilar solutions of \eqref{stefan}-\eqref{IC} are not $C^1$ smooth in the critical and supercritical regime. Moreover, we study in Section \ref{twuelve} the decay at $-\infty$ of the selfsimilar solutions when $s = \frac{1}{2}$. The refer the reader to the end of Section \ref{section 13} for more information on the optimal regularity, and Section \ref{sec.op} for comments on the results.

\medskip

%%%%%%%%%%%%%%%%%%%%%%%%%%%%%%%%%%%%%%%%%%%%%%%%%%%%%%%%%%%%%%%%%
\noindent\textbf{Comment on related results}

   There  are a number of works dealing with stationary or evolution equations involving fractional Laplacians
where the regularity of the solutions depends critically on the value of the fractional exponent  $s$, and
$s=1/2$ appears as the critical value. A well-known example is the work \cite{CaVs} where Caffarelli and Vasseur
consider  drift-diffusion scalar equations of the  form
\begin{equation}\label{eq.cvass}
u_t + {\bf v }\cdot \nabla u + (-\Delta)^{s} u=0, \quad \mbox{div\,} {\bf v }=0.
\end{equation}
They look at the case $s=1/2$, motivated by the critical dissipative quasi-geostrophic equation, see below. They prove
that drift-diffusion equations with critical exponent $s=1/2$ and  $L^2$ initial data are locally Holder continuous. As an application they show that solutions of the quasi-geostrophic equation with $L^2$ initial  data and critical diffusion are locally smooth for any space dimension. This in sharp contrast with our results. See also \cite{KiN}.

Another example comes from C\'ordoba and Zoroa \cite{CoZ} who consider the inviscid Surface Quasi-Geostrophic equation, a significant active scalar model with various applications in atmospheric modeling \cite{Pe}. It uses equation \eqref{eq.cvass} plus a special  relation between ${\bf v }$ and $u$. They prove that there exist solutions that lose functional regularity instantly (i.e., from the very beginning) if $s<1/2$ (a case they call supercritical), and this does not happen for $s=1/2$. Curiously, they find that the  growth around the origin is at least logarithmic,  logarithmic growth appears in our result for $s=1/2$.

Regarding nonlinear fractional diffusion problems related to the Stefan problem, let us mention the fractional porous medium equation, \cite{dP11, dP12, dP16}:   the  equation is still $u_t+(-\Delta)^{s}\Phi(u)=0$  but the nonlinearity \nc $\Phi$ is a power or comparable to a power. In that case H\"older regularity is proved for all $0<s<1$. If $\Phi$ is smooth and non-degenerate then the solutions are $C^\infty$  smooth  \cite{6}. However, the theory for this model does not cover the Stefan problem which can be considered as a less regular limit case. Another related model is called porous medium flow with potential pressure \cite{BilerGM, CaffVa2011a} where the equations read $u_t=\nabla \cdot (u\,\nabla p)$, with pressure given by $p=(-\Delta)^{-s} u$. There  bounded solutions have  H\"older regularity for all $0<s<1$, as proved in \cite{CaffSorVaz}. Actually, the case  $s=1/2$ is quite delicate, it is done in \cite{CaffVa2016}.
\nc

%%%%%%%%%%%%%%%%%%%%%%%%%%%%%%%%%%%%%%%%%%%%%%%%%%%%%%%%%%%%%%%%%%%%%%%%%%%%%%%%%%%%%
\section{Notation and preliminaries}\label{section 2}

We start the technical work with a section on notation and preliminary results that will be useful in the next sections.
Throughout the work we will use the capital letter $H$ to denote the enthalpy of the unique selfsimilar solution of \eqref{stefan} with initial data \eqref{IC}, and the letter $U$ for the selfsimilar profile of the temperature $u(x,t) = U(xt^{-\frac{1}{2s}})$. Sometimes, it will be useful to write the equation of  problem \eqref{stefan} in terms of both variables as
\[
\partial_t h + (-\Delta)^s u = 0 \hspace{7mm}\textup{in} \hspace{7mm} \mathbb{R} \times (0,T),
\]
where $u := \Phi(h) = (h -L)_+$.

We next introduce some useful results that will be  used in the next sections. The first one is an existence and uniqueness theorem  taken from \cite{5} that is valid for a wider class of problems.
\begin{thm}
\label{teo 1}
    Let $s \in (0,1)$ and consider the problem
\begin{equation}
    \label{no 1}
\left\{
\begin{aligned}
\partial_t h + (-\Delta)^s \Phi (h) &= 0 \hspace{7mm}\textup{in} \hspace{7mm} \mathbb{R}^n \times (0,T),\\
h( \cdot,0) &= h_0 \hspace{5.3mm} \textup{on} \hspace{6.5mm} \mathbb{R}^n,
\end{aligned}
\right.
\end{equation}
with $h_0 \in L^\infty(\mathbb{R})$ and $\Phi(\cdot)$ a nondecreasing and locally Lipschitz real function. Then, there exists a unique very weak solution  of \eqref{no 1}.
\end{thm}
For the concept of very weak solution see Definition \ref{def 2.2}. We will use the next result for selfsimilar solutions, taken from \cite{DEV1}.

\begin{thm}[Properties of the selfsimilar solution in $\mathbb{R}$]
\label{teo 4}
Let $s \in (0,1)$  and $P_1,  L > 0$. Consider the initial data
\begin{equation*}
h_0(x):=\left\{
\begin{aligned}
 &L + P_1 \hspace{10mm} &\textup{if} \hspace{7mm} x \leq 0, \\
&0 \hspace{10mm}        &\textup{if} \hspace{7mm}x > 0,
\end{aligned}
\right.
\end{equation*}
and let $h(x,t) \in L^\infty(\mathbb{R} \times (0,T))$ be the corresponding very weak solution of \eqref{stefan}. Then
\begin{itemize}
    \item[{$a)$}] \textup{(Profile)}
    $h$ and $u$ are selfsimilar with formulas
    \[
    h(x,t) = H(xt^{-1/(2s)}),\quad u(x,t) = U(xt^{-1/(2s)}) \quad\textup{for all} \quad (x,t) \in \mathbb{R} \times (0,T),
    \]
    where the selfsimilar profiles $H$ and $U= (H-L)_+$ satisfy the 1-D nonlocal equation:
    \begin{equation}\label{eq.ssh2}
     -\frac{1}{2s}\xi \cdot \nabla H(\xi) + (-\Delta)^s U(\xi) = 0 \qquad \textup{in} \qquad \mathcal{D}'(\mathbb{R}).
    \end{equation}
    \item[{$b)$}] \textup{(Boundedness and limits)} $0 \leq H(\xi)\leq L+P_1$ for all $\xi \in \mathbb{R}$ and
    \[
    \lim_{\xi \to -\infty} H_\varepsilon(\xi) = L + P_1 \quad \textup{and} \quad  \lim_{\xi \to \infty} H(\xi) =0.
    \]
    \item[{$c)$}] \textup{(Free boundary)} There exists a unique finite value $\xi_0 > 0$ such that $H(\xi_0) = L$. This means that the free boundary of the space-time solution $h(x,t)$ at the level $L$ is given by the curve
    \[
    x = \xi_0 \, t^\frac{1}{2s} \quad \textup{for all}\quad t\in (0,T).
    \]
    Moreover $\xi_0 > 0$ depends only on $s$ and the ratio $P_1/L$.

    \item[{$d)$}] \textup{(Monotonicity)} $H$ is nonincreasing. Moreover $H$ is strictly decreasing in $[\xi_0, +\infty)$.

    \item[{$e)$}] \textup{(Regularity)} $H \in C_b(\mathbb{R})$. Moreover, $H \in C^\infty(\xi_0, \infty)$, $H \in C^{1,\alpha}(-\infty, \xi_0)$ for some $\alpha \in (0,1)$.
\end{itemize}
\end{thm}

 The regularity of $H$  at $\xi_0$ is under study in this paper. In the next Section  \ref{section 3} we will  introduce  an approximation process by regularization of the nonlinearity. We will \nc then use the following theorem  taken from \cite{6}, that allows us to prove higher regularity of the solution $h(x,t)$ if the nonlinearity  $\phi \in C^{\infty}(\mathbb{R})$ and $\phi$ is non-degenerate.
\begin{thm}
\label{teo 2}
    Let $s \in (0,1)$ and $h$ a  bounded very weak \nc solution to
    \[
    \partial_t h + (-\Delta)^s  \phi\normalcolor(h) = 0 \hspace{7mm}\textup{in} \hspace{7mm} \mathbb{R}^n \times (0,\infty).
    \]
    If $\phi'(x) > 0$ in $\mathbb{R}$ and $ \phi \in C^\infty(\mathbb{R})$ then $h \in C^\infty(\mathbb{R}^n \times (0,\infty))$ and is a solution of the equation in the classical sense.
\end{thm}

\medskip

Using Theorem \eqref{no 1} and Section 6 of \cite{6}  \nc we get

\begin{cor} \label{co 3} Under the above conditions on $\phi$, if $h$ is a very weak solution to
    \[
    \partial_t h + (-\Delta)^s \phi(h) = 0 \hspace{7mm}\textup{in} \hspace{7mm} \mathbb{R}^n \times (0,\infty).
    \]
 with bounded initial data. There is ${\alpha}>0$ such that \ $\partial_t  h, \partial_{x_i }h, (-\Delta)^s h \in C^\alpha (\mathbb{R}^n \times (\tau,\infty)) \cap L^\infty(\mathbb{R}^n \times (\tau,\infty))$ for every $\tau>0$.  The bounds depend on the nonlinearity, on $s$, and the stated bound on the data\normalcolor.
\end{cor}

For the starting continuity see \cite{AtCa, CCV}\nc. On the other hand, we will need a different regularity result, taken from Theorem 5.2 of reference \cite[Section 5]{ChD} that implies $C^{1,\alpha}$ regularity for some $\alpha>0$ with uniform bounds for the solutions  of some general class of equations\nc:

\begin{thm}
\label{davila}
    Let   $\sigma_o \in (0,2)$, $\sigma \in (\sigma_o, 1)$ and $f \in C([-1,0])$. There is $\rho_o$ depending only on $n, \lambda, \Lambda, \sigma_o$ such that if $I$ is an $\textit{inf sup}$ $(\textit{or sup inf})$ type operator, translation invariant in space and elliptic with respect to the class $\mathcal{L}_1(\sigma, \Lambda, \rho_o)$, and if $h \in C(\overline{B}_1\times [-1,0])$ is a viscosity solution of the equation
    \[
    \partial_t h - Ih = f(t) \qquad \textit{in} \qquad B_2 \times (-1,0],
    \]
    then $h$ is $C^{1,\alpha}$ in space for some universal $\alpha \in (0,1)$. More precisely, there is a constant $C > 0$, depending only on $n$, $\lambda$ and $\sigma$ such that for every $(x,t), (y,s) \in B_{1/4} \times [-1/2,0]\nc$
\begin{equation}\label{form.Hchdv}
    \frac{|\partial_{x_i}h(x,t) - \partial_{x_i}h(y,s)|}{(|x-y|+ |t-s|^{1/\sigma})^\alpha} \leq C\left(\norm{h}_{L^\infty (\overline{B}_2 \times [-1,0])\nc} + \norm{h}_{C(-1,0; L^1 (\omega))}\nc + \norm{f}_{L^{\infty}((-1,0])}\right),
\end{equation}
for $i \in \{1,2,\ldots,n\}$.
\end{thm}
The necessary concepts are carefully defined in the paper, in particular the tail space $L^1 (\omega)$ where $\omega(x)=1/(1+|x|)^{N+\sigma} $ is the appropriate fractional weight. Notice that \cite{ChD} uses the  notation $\sigma=2s$ for the fractional parameter. $N$ is the space dimension, here $N=1$. In our application $f=0$.  The wide class of $\textit{inf sup}$ $(\textit{or sup inf})$ type operators is not needed, we will apply the result to a fractional Laplacian $I=-(-\Delta )^{\sigma/2}$.\nc

\medskip

Finally, we present a functional tool that will be useful to prove the convergence of the regularized problems of next section to the unique solution of \eqref{stefan}-\eqref{IC}.

\begin{thm}[Fr\'echet-Kolmogorov]
    \label{Frechet}
    Let $p \in [1,\infty)$ and $\mathcal{A} \subset L^p(\mathbb{R}^n)$ a subset. Then  $\mathcal{A}$ is relatively compact if only if the following two properties hold:
    \begin{itemize}
        \item [$a)$]{(Equicontinuity) }$\lim_{|h| \to 0} \norm{\tau_h f - f}_{L^p(\mathbb{R}^n)} = 0$ uniformly $\forall f \in \mathcal{A}$.
        \item [$b)$]{(Equitight)} $\lim_{r\to \infty} \int_{|x| > r} |f|^p = 0$ uniformly  $\forall f \in \mathcal{A}$,
    \end{itemize}
    where $\tau_h f := f(x-h)$ denotes the translation of $f$ by an amount $h$.
\end{thm}

%\newpage
%%%%%%%%%%%%%%%%%%%%%%%%%%%%%%%%%%%%%%%%%%%%%%%%%%%
\section{Regularized Fractional Stefan Problem}\label{section 3}

In this section we introduce and study the regularized Fractional Stefan Problem. The idea is to change the Lipschitz function $\Phi$ in  problem \eqref{stefan} into a family of functions $\Phi_\varepsilon \in C^\infty(\mathbb{R})$ with $\varepsilon \in (0,1)$. We define the regularized Fractional Stefan problems as

\begin{equation} \tag{$P_\varepsilon$}
\label{59}
\left\{
\begin{aligned}
\partial_t h + (-\Delta)^s \Phi_\varepsilon (h) &= 0 \hspace{7mm}\textup{in} \hspace{7mm} \mathbb{R}^n \times (0,T),\\
h( \cdot,0) &= h_0 \hspace{5.3mm} \textup{on} \hspace{6.3mm} \mathbb{R}^n,
\end{aligned}
\right.
\end{equation}
where $h_0 \in L^\infty(\mathbb{R})$, $T > 0$, $s \in (0,1)$ and $\Phi_\varepsilon(x) \in C^\infty(\mathbb{R})$ is a function that we define as
\begin{equation}
\label{60}
\Phi_\varepsilon(x) :=\left\{
\begin{aligned}
 &x -L\hspace{20.7mm}\textup{if} \hspace{7mm} L \leq x, \\
& \psi_\varepsilon(x-L) \nc \hspace{21mm}\textup{if} \hspace{7mm} L- \varepsilon\leq x < L,\\
\varepsilon(x&-(L-\varepsilon)) - 2\varepsilon \hspace{6mm}\textup{if} \hspace{7mm} x < L - \varepsilon,
\end{aligned}
\right.
\end{equation}
with $ L > 0$ a constant and $\psi_\varepsilon(x) \in C^\infty(\mathbb{R})$  such that $\Phi_\varepsilon(x)$ is $C^\infty(\mathbb{R})$ and non-degenerate, in particular \
$$
\psi_\varepsilon(0)=0, \quad  \psi'_\varepsilon(0)=1, \quad  \psi_\varepsilon(-\varepsilon)=-2\varepsilon, \quad \psi'_\varepsilon(-\varepsilon)=\varepsilon,
$$
and $\varepsilon \leq \psi'_\varepsilon(x) \leq 1$. Note that  $1\ge \Phi_\varepsilon'(x) \geq \varepsilon > 0$ for all
$x \in \mathbb{R}$, i.e., $\Phi_\varepsilon(x)$ is a strictly increasing function. We want to study some basic properties of the solutions of \eqref{59} like uniqueness, existence of solutions, regularity and selfsimilarity.

%%%%%%%%%%%%%%%%%%%%%%%%%%%%%%%%%%%%%%%%%%%%%%%%%%%%%%%%%%%%%%%%%%%%%%%%
\begin{figure}[ht!]
 \centering
 \includegraphics[width=10cm]{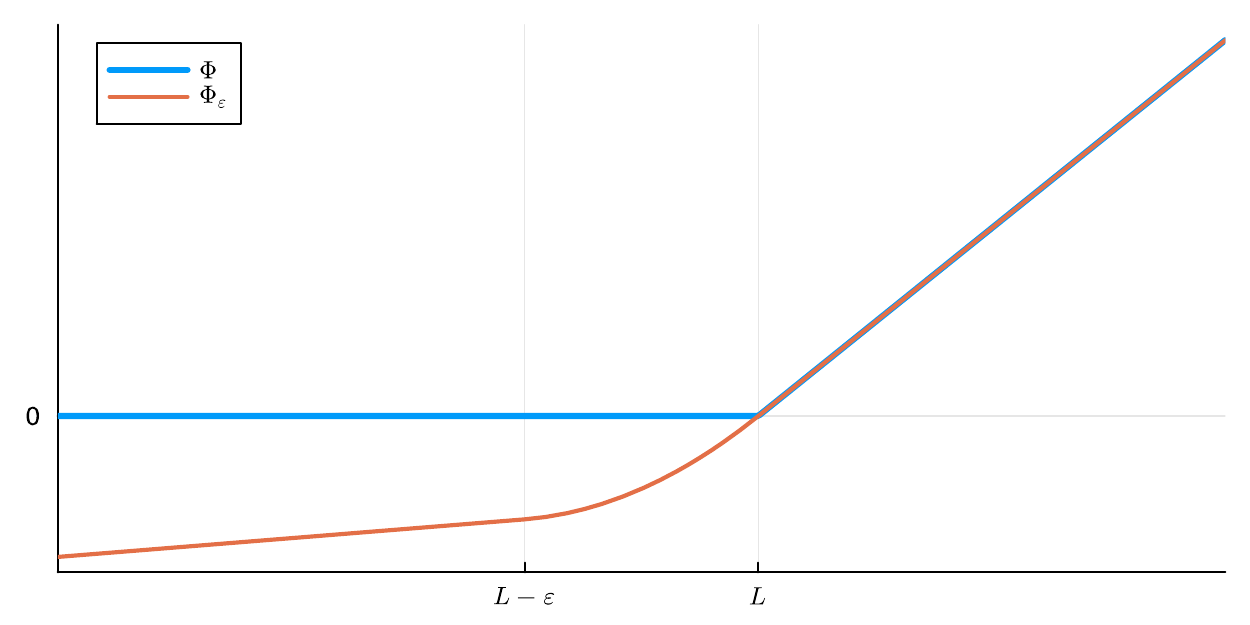}
 \caption{The Stefan nonlinearity $\Phi$ (blue) and its smooth approximation $\Phi_\varepsilon$ (red)}
 \label{fig:1}
\end{figure}
%%%%%%%%%%%%%%%%%%%%%%%%%%%%%%%%%%%%%%%%%%%%%%%%%%%%%%%%%%%%%%%%%%%%%%%%

\begin{rem} \rm
We take $\varepsilon \in (0,1)$ small, and  denote by $h_\varepsilon$ the solution of \eqref{59}. Note that \eqref{59} is a family of problems.
We call $u_\varepsilon= \Phi_\varepsilon(h_\varepsilon)$.\end{rem}

\subsection{Existence and uniqueness of  solution}
We will work in the framework of very weak solutions. The definition of weak solution is rather standard and  has been studied in \cite{DEV1,3,4} for a more general family of solutions including \eqref{59}.

\begin{defn}[Bounded very weak solutions]
\label{def 2.2}
    Let $h_0 \in L^\infty(\mathbb{R}^n)$ and $\varepsilon \in (0,1)$. We say that $h_\varepsilon \in L^\infty(\mathbb{R}^n \times (0,T))$ is a very weak solution of \eqref{59} if  for all $ \psi\in C_c^\infty(\mathbb{R}^n \times [0,T))$
    \begin{equation}
        \label{61}
        \int_0^T\int_{\mathbb{R}^n}
        \left( h_\varepsilon\partial_t \psi - \Phi_\varepsilon(h_\varepsilon) \,(-\Delta)^{s}\psi \right) \hspace{1mm}dxdt + \int_{\mathbb{R}^n}h_0(x)\psi(x,0)\hspace{1mm}dx = 0.
    \end{equation}
\end{defn}

\begin{rem} \rm
    An equivalent alternative for \eqref{61} is
    \[
    \partial_t h_\varepsilon + (-\Delta)^{s}\Phi_\varepsilon(h_\varepsilon) = 0 \quad \textup{in} \quad \mathcal{D}'(\mathbb{R}^n \times (0,T)),
    \]
    and
    \[
    \textup{ess}\lim_{t \to 0^+} \int_{\mathbb{R}^n}h_\varepsilon(x,t)\psi(x,t)\hspace{1mm}dx = \int_{\mathbb{R}^n}h_0(x)\psi(x,0)\hspace{1mm}dx \quad \forall  \psi\in C_c^\infty(\mathbb{R}^n \times [0,T)).
    \]
\end{rem}
Now we prove the existence and uniqueness of solution of \eqref{59}.

 \begin{thm}[Existence and uniqueness]
 \label{the 2.4}
     Let $s \in (0,1)$ and $h_0 \in L^\infty(\mathbb{R}^n\times (0,T))$. Then for each $\varepsilon\in (0,1)$  there exists a unique bounded very weak solution of \eqref{59}.
 \end{thm}
 \begin{proof}
     Let $\varepsilon \in (0,1)$. By \eqref{60} we know that $\Phi_\varepsilon(x)$ is a globally Lipschitz and strictly increasing function. Then by Theorem \ref{teo 1}, we deduce that there exists a unique  bounded very weak solution of \eqref{59}.
 \end{proof}

 A consequence of the Theorem \ref{the 2.4} and \cite[Theorem A.1]{DEV1} is the next well-posedness result of bounded very weak solutions of \eqref{59}.

\begin{thm}[Well-posedness of bounded solutions]
\label{teo 2.5}
    Let $s\in (0,1)$ and  $h_0 \in L^\infty(\mathbb{R}^n)$. Then for each $\varepsilon \in (0,1)$ there exists a unique very weak solution $h_\varepsilon \in L^\infty(\mathbb{R}^n \times (0,T))$ of \eqref{59}. Moreover,  given a fixed $\varepsilon \in (0,1)$ and two very weak solutions of \eqref{59} $h_\varepsilon,\hat{h}_\varepsilon \in L^\infty(\mathbb{R}^n\times (0,T))$ with initial data $h_0, \hat{h}_0 \in L^\infty(\mathbb{R}^n)$, they have the following properties:
\begin{enumerate}
\item[{$a)$}] $\norm{h_\varepsilon(\cdot,t)}_{L^\infty(\mathbb{R}^n)}$ $\leq \norm{h_0}_{L^\infty(\mathbb{R}^n)}$ $\quad$ for a.e $\quad$ $t \in (0,T)$.
      \item[{$b)$}] If $h_0 \leq \hat{h}_0$ a.e in $\mathbb{R}^n$, then $h_\varepsilon \leq \hat{h}_\varepsilon$ a.e in $\mathbb{R}^n \times (0,T)$.
\item[{$c)$}] If $(h_0 - \hat{h}_0)^+ \in L^1(\mathbb{R}^n)$, then
\[
 \int_{\mathbb{R}^n} (h_\varepsilon(x,t) - \hat{h}_\varepsilon(x,t))^+ \hspace{1mm}dx \leq \int_{\mathbb{R}^n} (h_0(x,t) - \hat{h}_0(x,t))^+ \hspace{1mm}dx \quad \textup{for a.e} \quad t\in (0,T).
 \]
         \item[{$d)$}] If $h_0 \in L^1(\mathbb{R}^n)$, then
         \[
         \int_{\mathbb{R}^n} h_\varepsilon(x,t) \hspace{1mm}dx = \int_{\mathbb{R}^n} h_0(x) \hspace{1mm}dx \quad \textup{for a.e} \hspace{2mm} t\in (0,T).
         \]
         \item[{$e)$}] If $\norm{h_0(\cdot + \xi) - h_0}_{L^1(\mathbb{R}^n)} \rightarrow 0$ as $|\xi|\rightarrow 0^+$, then $h_\varepsilon \in C([0,T] : L^1_{loc}(\mathbb{R}^n))$. Moreover for every $t,s\in [0,T]$ and every compact set $K \subset \mathbb{R}^n$
         \[
         \norm{h_\varepsilon(\cdot,t) - h_\varepsilon(\cdot,s)}_{L^1(K)}\leq \varLambda(|t-s|),
         \]
         where $\varLambda$ is a modulus of continuity depending on $K$ and $\norm{h_0(\cdot + \xi) - h_0}_{L^1(\mathbb{R}^n)}$.
\end{enumerate}
\end{thm}

\begin{proof}
    Thanks to Theorem \ref{the 2.4} we have uniqueness and existence of solution of \eqref{59}. The properties $a)-e)$ are a consequence of \cite[Theorem A.1]{DEV1}.
\end{proof}

%%%%%%%%%%%%%%%%%%%%%%%%%%%%%%%%%%%%%%%%%%%%%%%%%%%%%%%%%%%%%%%%
\subsection{Preliminary selfsimilar results}

Now we study the existence and uniqueness of selfsimilar solutions of the problem \eqref{59}. Firstly, we prove the next scaling invariance:
\begin{lem}[Scaling invariance]
\label{lemma 2.6}
    Let $s \in (0,1)$, $\varepsilon\in (0,1)$ and $h_\varepsilon \in L^\infty(\mathbb{R}^n \times (0,T))$ be a very weak solution of \eqref{59} with initial data $h_0 \in L^\infty(\mathbb{R}^n)$. Then for all $a > 0$ the function $h_a(x,t) := h(ax,a^{2s}t)$ is a very weak solution of \eqref{59} with initial data $h_{0,a}(x) := h(ax)$.
\end{lem}
\begin{proof}
    The proof is the same as in the \cite[Lemma 3.5]{DEV1} changing $\Phi(x)$ into $\Phi_\varepsilon(x)$.
\end{proof}
Once we know the selfsimilarity of the equation, choosing an appropriate initial data will provide us with a selfsimilar solution and a selfsimilar profile.

\begin{lem}[Existence and uniqueness of a selfsimilar solution]
\label{lemma 2.7}
Let $s \in (0,1)$, $\varepsilon\in(0,1)$ and $h_\varepsilon \in L^\infty(\mathbb{R}^n \times (0,T))$ the unique very weak solution of \eqref{59} with initial data $h_0\in L^\infty(\mathbb{R}^n)$. If $h_0(x) = h_0(ax)$ for all $a > 0$ and a.e $x \in \mathbb{R}^n$, then
\[
h_\varepsilon(x,t) = h_\varepsilon(ax,a^{2s}t) \quad \textup{for a.e} \quad (x,t) \in \mathbb{R}^n \times (0,T),
\]
and all $a > 0$. In particular, $h(x,t) = h(xt^{-1/(2s)},1)$ for a.e $(x,t) \in \mathbb{R}^n \times (0,T)$.
\end{lem}
\begin{proof}
    Recall that given $\varepsilon \in (0,1)$ and an initial data $h_0\in L^\infty(\mathbb{R}^n)$, by the Theorem \ref{the 2.4}, we know that there exists a unique solution of $h_\varepsilon$.

    For $\varepsilon\in(0,1)$ and $a >0$, by Lemma \ref{lemma 2.6} we have that $h_\varepsilon(ax,a^{2s}x)$ is a very weak solution of \eqref{59} with initial data $h_0(ax)$. By the scaling property $h_0(ax) = h_0(x)$ we thus get that $h_\varepsilon(ax,a^{2s}t)$ is a solution with initial data $h_0(x)$ for all $a >0$. By Theorem \ref{the 2.4} we deduce that $h_\varepsilon(x,t) = h_\varepsilon(ax,a^{2s}t)$ for all $a >0$. In particular, by choosing $a = t^{-1/(2s)}$ we get the identity.
\end{proof}

We can express the selfsimilar solution in terms of a profile satisfying a stationary equation.

\begin{lem}[Equation of the profile]
    \label{lemma 2.8}
    Under the assumptions of Lemma \ref{lemma 2.7}, consider the function $H_\varepsilon(\xi):= h_\varepsilon(\xi,1)$ for a.e $\xi \in \mathbb{R}^n$. Then, we have that $H_\varepsilon$ satisfies
    \begin{equation}
     \label{62}
    -\frac{1}{2s}\xi \cdot \nabla H_\varepsilon(\xi) + (-\Delta)^s \Phi_\varepsilon(H_\varepsilon)(\xi) = 0 \qquad \textup{in} \qquad \mathcal{D}'(\mathbb{R}^n).
    \end{equation}
\end{lem}
\begin{proof}
    The proof is the same as \cite[Lemma 3.7]{DEV1}.
\end{proof}

\begin{rem} \rm
   Note that the relation $u_\varepsilon(x,t) := \Phi_\varepsilon(h(x,t))$ allows us to  define a selfsimilar profile for $u_\varepsilon$ too. More precisely,
    \[
    u_\varepsilon(x,t) = \Phi_\varepsilon(h_\varepsilon(x,t)) = \Phi_\varepsilon(H_\varepsilon(\xi))=: U_\varepsilon(\xi),
    \]
    with $\xi= xt^{-1/(2s)}$.
 Then, we can rewrite our problem as
\begin{equation}
    \label{63}
     -\frac{1}{2s}\xi \cdot \nabla H_\varepsilon(\xi) + (-\Delta)^s U_\varepsilon(\xi) = 0 \qquad \textup{in} \qquad \mathcal{D}'(\mathbb{R}^n).
\end{equation}
\end{rem}

\subsection{Preliminary regularity results}
Thanks to the good properties of $\Phi_\varepsilon(x)$ we can prove some useful regularity results. We will  see that all the solutions of \eqref{59} are $C^\infty(\mathbb{R}^n)\cap C^{0,1}(\mathbb{R}^n)$. Moreover, we will prove that the selfsimilar profile of the temperature $U_\varepsilon$ is $C^{0,1}(\mathbb{R}^n)$.
\begin{lem}
\label{lemma 2.6b}
    Let $s\in (0,1)$, $\varepsilon \in (0,1)$, $h_0\in L^\infty(\mathbb{R}^n)$ and $h_\varepsilon$ the unique solution of \eqref{59}. Then $h_\varepsilon \in C^\infty(\mathbb{R}^n \times (0,T))$.
\end{lem}
\begin{proof}
    By the definition of $\Phi_\varepsilon(x)$ we have that the equation \eqref{59} is uniformly parabolic, because $\Phi_\varepsilon'(x) > 0$ for all $x \in \mathbb{R}$. Then, by Theorem \ref{teo 2} we obtain that $h_\varepsilon \in C^\infty(\mathbb{R}^n \times (0,T))$.
\end{proof}

A consequence of  Lemma \ref{lemma 2.6b} is that the selfsimilar profile  $U_\varepsilon(\xi)$ is uniformly Lipschitz in $\mathbb{R}^n$.

\begin{lem}
\label{lemma 2.11}
    Let $s\in (0,1)$, $\varepsilon \in (0,1)$ and $h_0 \in L^\infty(\mathbb{R}^n)$. Suppose that the initial data satisfies $h_0(x) = h_0(ax)$ for all $a > 0$. Then, the  unique selfsimilar solution $H_\varepsilon$ of \eqref{59} and its temperature profile $U_\varepsilon $ are globally Lipschitz in $\mathbb{R}^n$.
\end{lem}
\begin{proof}
    Recall that thanks to Lemma \ref{lemma 2.7} we know that there exists a unique selfsimilar solution $H_\varepsilon(xt^{-1/(2s)}) = h_\varepsilon(x,t)$ of \eqref{59} with initial data $h_0(x)$. Firstly we prove that $H_\varepsilon \in C^{0,1}(\mathbb{R}) $. Let $\xi = xt^{-1/(2s)}$ ,
 by Corollary \ref{co 3} we have
    \begin{equation}
    \label{64}
    \left|\partial_t h_\varepsilon(x,t)\right| = |\partial_t H_\varepsilon(xt^{-\frac{1}{2s}})| = \frac{1}{2s}|t^{-1}\xi  \cdot\nabla H_\varepsilon(\xi)| \leq C_\varepsilon,
    \end{equation}
    with $C_\varepsilon > 0$ a constant that only depends on $\varepsilon$. Taking in \eqref{64} $t=1$, we obtain
    \[
    |x \nabla H_\varepsilon(x)| \leq 2s C_\varepsilon.
    \]
    Then, if $x \in \mathbb{R}\setminus B_\delta(0)$ for $\delta > 0$ fixed, we deduce
\[
|\nabla H_\varepsilon(x)| \leq 2s C_\varepsilon \frac{1}{|x|} \leq 2s C_\varepsilon \frac{1}{\delta} ,
\]
i.e, the gradient is bounded in $ \mathbb{R}^n\setminus B_\delta(0)$. On the other hand, since by Lemma \ref{lemma 2.6b} $H_\varepsilon \in C^\infty(\mathbb{R}^n)$ in particular we know that $|\nabla H_\varepsilon(x)| \leq C_\varepsilon$ for all $x\in \overline{B}_\delta(0)$. Then we conclude that $|\nabla H_\varepsilon(x)| \leq C_\varepsilon$, i.e, $H_\varepsilon \in C^{0,1}(\mathbb{R}^n)$.

Finally, we prove that $U_\varepsilon \in C^{0,1}(\mathbb{R}^n)$. Let $x,y\in \mathbb{R}^n$,  by the definition of $U_\varepsilon$ and $\Phi_\varepsilon(x)$ we have
    \begin{align*}
     |U_\varepsilon(x) - U_\varepsilon(y)| &= |\Phi_\varepsilon(H_\varepsilon(x)) - \Phi_\varepsilon(H_\varepsilon(y))|\leq [\Phi_\varepsilon]_{C^{0,1}(\mathbb{R})}|H_\varepsilon(x) - H_\varepsilon(y)|\\
     & \leq [\Phi_\varepsilon]_{C^{0,1}(\mathbb{R})}[H_\varepsilon]_{C^{0,1}(\mathbb{R}^n)}|x-y|
     = [H_\varepsilon]_{C^{0,1}(\mathbb{R}^n)}|x-y|,
    \end{align*}
i.e.,  $U_\varepsilon$ is globally Lipschitz in $\mathbb{R}^n$ with the same constant as the function $H_\varepsilon$. Note that we are using that the function $\Phi_\varepsilon(x)$ and $H_\varepsilon$ are globally Lipschitz and by the definition of $\Phi_\varepsilon(x)$ that $[\Phi_\varepsilon]_{C^{0,1}(\mathbb{R})} = 1$.
\end{proof}

%%%%%%%%%%%%%%%%%%%%%%%%%%%%%%%%%%%%%%%%%%%%%%%%%%%%%%%%%%%%%%%%%%%%
\subsection{Properties of the selfsimilar solution}

In this section we want to study some properties of the selfsimilar solutions of the one-dimensional regularized Fractional Stefan Problem.  The main result that we want to prove is the next Theorem:

\begin{thm}[Properties of the selfsimilar solution in $\mathbb{R}$]
\label{teo 2.12}
Let $s \in (0,1)$, $\varepsilon \in (0,1)$, $n=1$ and $P_1, L > 0$. Suppose that $L > \varepsilon$ and consider the initial data
\begin{equation*}
h_0(x):=\left\{
\begin{aligned}
 &L + P_1 \hspace{10mm}\textup{if} \hspace{7mm} x \leq 0, \\
&0 \hspace{19.8mm}\textup{if} \hspace{7mm}x > 0,
\end{aligned}
\right.
\end{equation*}
and let $h_\varepsilon(x,t) \in L^\infty(\mathbb{R}^n \times (0,T))$ be the corresponding very weak solution of \eqref{59}. Then
\begin{itemize}
    \item[{$a)$}] \textup{(Profile)}
    $h_\varepsilon$ and $u_\varepsilon$ are selfsimilar with formulas
    \[
    h_\varepsilon(x,t) = H_\varepsilon(xt^{-1/(2s)}), u_\varepsilon(x,t) = U_\varepsilon(xt^{-1/(2s)}) \quad\textup{for all} \quad (x,t) \in \mathbb{R} \times (0,T),
    \]
    where the selfsimilar profiles $H_\varepsilon$ and $U_\varepsilon = \Phi_\varepsilon(H_\varepsilon)$ satisfy the 1-D nonlocal equation:
    \begin{equation}
    \label{65}
     -\frac{1}{2s}\xi \cdot \nabla H_\varepsilon(\xi) + (-\Delta)^s U_\varepsilon(\xi) = 0 \qquad \textup{in} \qquad \mathcal{D}'(\mathbb{R}^n).
    \end{equation}
    \item[{$b)$}] \textup{(Boundedness and limits)} $0 \leq H_\varepsilon(\xi)\leq L+P_1$ for all $\xi \in \mathbb{R}$ and
    \[
    \lim_{\xi \to -\infty} H_\varepsilon(\xi) = L + P_1 \quad \textup{and} \quad  \lim_{\xi \to +\infty} H_\varepsilon(\xi) = 0.
    \]
    \item[{$c)$}] \textup{(Regularity)} $H_\varepsilon \in C^\infty(\mathbb{R}) \cap C^{0,1}(\mathbb{R})$. Moreover, $U_\varepsilon \in C^{0,1}(\mathbb{R})$.
\end{itemize}

\end{thm}

We study each property of the Theorem \ref{teo 2.12} separately. Firstly, we start by showing Theorem \ref{teo 2.12}-$a)$.
\begin{lem}
    Under the assumptions of Theorem \ref{teo 2.12}, we have  that $h_\varepsilon(x,t) = H_\varepsilon(xt^{-1/(2s)})$ for a.e $(x,t) \in \mathbb{R}\times (0,T)$ and $H_\varepsilon$ satisfies \eqref{65}.
\end{lem}
\begin{proof}
 Recall that by Theorem \ref{the 2.4} we have existence and uniqueness of solution. Note that $h_0(ax) = h_0(x)$ for all $x \in \mathbb{R}$ and $a > 0$. Then the result follows from Lemma \ref{lemma 2.7} and Lemma \ref{lemma 2.8}.
\end{proof}

Now we prove Theorem \ref{teo 2.12}-$b)$.
\begin{lem}
    Under the assumptions of Theorem \ref{teo 2.12}, we have that $0 \leq H_\varepsilon(\xi)\leq L+P_1$ for all $\xi \in \mathbb{R}$ and
    \[
    \lim_{\xi \to -\infty} H_\varepsilon(\xi) = L + P_1 \quad \textup{and} \quad  \lim_{\xi \to +\infty} H_\varepsilon(\xi) = 0.
    \]
\end{lem}
\begin{proof}
  Since by hypothesis $P_2 > \varepsilon$ then $L+P_1$ and $0$ are stationary solutions of \eqref{59}.  Clearly $0 \leq h_0 \leq L+P_1$. Thus by Theorem \ref{teo 2.5}-$b)$, we have that $0\leq h_\varepsilon \leq L+P_1$ which by definition implies $0 \leq H_\varepsilon(\xi)\leq L+P_1$ for a.e $\xi \in \mathbb{R}$. For proving the limits we need to divide the proof in two steps:

  \noindent\underline{\textit{Step 1}}: Firstly, we will prove that the solution $H_\varepsilon$ is nonincreasing. Consider the initial data $\hat{h}_0(x) := h_0(x+a)$ for some $a > 0$ and the corresponding solution $\hat{h}_\varepsilon$. Note that $h_0 \geq \hat{h}_0(x)$, and thus by comparison, Theorem \ref{teo 2.5}, $h_\varepsilon\geq \hat{h}_\varepsilon$. Note that \eqref{59} is translational invariant, and so $h_\varepsilon(x+a,t)= \hat{h}_\varepsilon(x,t)$ for a.e $(x,t)\in \mathbb{R} \times (0,T)$. We have then that $h(x,t) \geq h(x+a,t)$ for a.e $(x,t)\in \mathbb{R} \times (0,T)$, which concludes that $H_\varepsilon$ is nonincreasing since $a > 0$ was arbitrary and the relation $H_\varepsilon(\xi) = h_\varepsilon(\xi,1)$ holds.\\

    \noindent\underline{\textit{Step 2}}: Now we prove the limits. By the Theorem \ref{teo 2.12}-$b)$ we know that the function $H_\varepsilon$ is bounded from above and from below. By the Step 1 we know that the $H_\varepsilon$ is monotone. This two facts imply the existence of the limits:
        \[
        \underline{H}_\varepsilon := \lim_{\xi \to -\infty} H_\varepsilon(\xi) \quad \textup{and} \quad \overline{H}_\varepsilon:=\lim_{\xi \to -\infty} H_\varepsilon(\xi).
        \]
        In particular,
        \[
        \textup{ess}\lim_{t \to 0^+} H_\varepsilon(xt^{-\frac{1}{2s}}) =\left\{.
    \begin{aligned}
    &\underline{H}_\varepsilon \hspace{10mm}\textup{if} \hspace{7mm} x < 0, \\
    &\overline{H}_\varepsilon \hspace{10mm}\textup{if} \hspace{7mm}x > 0.
    \end{aligned}
    \right.    \]
    Given any $\psi \in C_c^\infty(\mathbb{R} \times [0,T))$, we have by Definition \ref{def 2.2} that
    \begin{equation}
     \label{66}
    \textup{ess}\lim_{t \to 0^+} \int_{\mathbb{R}} H_\varepsilon(xt^{-\frac{1}{2s}})\psi(x,t) \hspace{1mm}dx = \textup{ess}\lim_{t \to 0^+}\int_{\mathbb{R}} h_\varepsilon(x,t)\psi(x,t) \hspace{1mm}dx = \int_{\mathbb{R}}h_0(x)\psi(x,0)\hspace{1mm}dx.
    \end{equation}
    Take now $\psi \in C_c^\infty(\mathbb{R} \times [0,T))$. The Lebesgue dominated convergence Theorem yields
    \begin{equation}
      \label{67}
     \textup{ess}\lim_{t \to 0^+} \int_{\mathbb{R}} H_\varepsilon(xt^{-\frac{1}{2s}})\psi(x,t) \hspace{1mm}dx = \int_{\mathbb{R}} \textup{ess}\lim_{t \to 0^+} (H_\varepsilon(xt^{-\frac{1}{2s}})\psi(x,t))\hspace{1mm}dx = \int_{\mathbb{R}_+}\overline{H}_\varepsilon\psi(x,0)\hspace{1mm}dx,
    \end{equation}
    and thus, we have by \eqref{66} and \eqref{67} the following identity for all $\phi \in C^\infty_c(\mathbb{R}_+)$
    \[
    \int_{\mathbb{R}_+}\overline{H}_\varepsilon \phi(x) \hspace{1mm}dx = \int_{\mathbb{R}}h_0(x) \phi(x) \hspace{1mm}dx.
    \]
    This implies that $\overline{H}_\varepsilon=h_0(x)$ for a.e $x \in \mathbb{R}_+$, i.e $\overline{H}_\varepsilon = 0$. The same argument taking $\psi\in C_c^\infty(\mathbb{R}_- \times[0,T))$ shows that $\underline{H}_\varepsilon = h_0(x)$ for a.e $x\in \mathbb{R}_-$, i.e $\underline{H}_\varepsilon = L + P_1$.
\end{proof}

Finally we prove Theorem \ref{teo 2.12}-$c)$ as a consequence of Lemma \ref{lemma 2.6b} and Lemma \ref{lemma 2.11}.

\begin{lem}
    Under the assumptions of Theorem \ref{teo 2.12}, we have that $H_\varepsilon \in C^\infty(\mathbb{R}) \cap C^{0,1}(\mathbb{R})$. Moreover, $U_\varepsilon \in C^{0,1}(\mathbb{R})$.
\end{lem}

%%%%%%%%%%%%%%%%%%%%%%%%%%%%%%%%%%%%%%%%%%%%%%%%%%%%%%%%%%%%%%%%%%
\section{Convergence to the selfsimilar solution}

The purpose of this section is to show that there exists a sequence of smooth solutions of the regularized fractional Stefan problem that converges to the unique solution of Problem \eqref{stefan}-\eqref{IC}. For proving this result will be crucial to use the Fr\'echet-Kolmogorov Theorem that characterizes the relatively compact subsets of the $L^p$ spaces. The main result of the section reads as follows:

\begin{prop}
\label{prop 4.1}
    Let $s \in (0,1)$, $\varepsilon \in (0,1)$, $\{h_{\varepsilon}(x,t)\}_\varepsilon$ a sequence of solutions of \eqref{59}-\eqref{IC} and $h(x,t)$ the unique solution of \eqref{stefan}-\eqref{IC}. Then there exists a subsequence $\{h_{\varepsilon_j}(x,t)\}_j$ such that
    \[
    h_{\varepsilon_{j}}(x,t) \to h(x,t) \qquad \textup{in} \qquad L^1_{loc}(\mathbb{R}\times (0,T))
    \]
and the convergence is also locally a.e.\end{prop}

\begin{proof}
    We divide the proof in two steps. Firstly, we prove that the family of solutions $\{h_\varepsilon(x,t)\}_\varepsilon$ is relatively compact and converges on compact sets of $\mathbb{R}\times (0,T)$ to some function $g \in L^1(\mathbb{R}\times (0,T))$. In the second step, we will show that $g = h$ for a.e. in $\mathbb{R}\times (0,T)$.\\

    \noindent\underline{\textit{Step 1}}: Let $K \subset \mathbb{R}\times (0,T)$ be a compact set. We define
    \[
    \Tilde{h}_\varepsilon(x,t) := h_\varepsilon(x,t) \mathcal{X}_K(x,t).
    \]
    Since $\{\tilde
{h}_\varepsilon(x,t)\}_\varepsilon \subset L^1(\mathbb{R}^2)$, by the Fr\'echet-Kolmogorov Theorem it is enough to prove that $\{\tilde
{h}_\varepsilon(x,t)\}_\varepsilon$ is equitight and equicontinuous to show that is relatively compact in $L^1(\mathbb{R}^2)$.  Clearly since for all $\varepsilon \in (0,1)$ the function $\tilde
{h}_\varepsilon(x,t)$ has compact support then is equitight. Thus, it is enough to prove equicontinuity.

Let $k>0$ be an increment\nc. We use the double integral sign  for clarity. By the triangular inequality we have
\begin{align}
    \label{4.2}
   &\iint_{\mathbb{R}^2}\nc |\Tilde{h}_\varepsilon(x-k\nc,t-k\nc) - \Tilde{h}_\varepsilon(x,t)| \hspace{1mm} dx\hspace{0.5mm}dt \leq \iint_{\mathbb{R}^2} |\Tilde{h}_\varepsilon(x-k,t-k) - \Tilde{h}_\varepsilon(x,t-k)| \hspace{1mm} dx\hspace{0.5mm}dt\\
    & \hspace{20mm}+ \iint_{\mathbb{R}^2} |\Tilde{h}_\varepsilon(x,t-k)-\Tilde{h}_\varepsilon(x,t)| \hspace{1mm} dx\hspace{0.5mm}dt \nonumber = I + II.
\end{align}

We study each integral separately. Splitting integral $I$ we obtain
\begin{align*}
    \iint_{\mathbb{R}^2} |\Tilde{h}_\varepsilon(x-k,t-k) - \Tilde{h}_\varepsilon(x,t-k)| \hspace{1mm} dx\hspace{0.5mm}dt &\leq I_i + I_{ii}
\end{align*}
with
\begin{align}
\label{4.3}
    & I_i := \iint_{\mathbb{R}^2} |h_\varepsilon(x-k,t-k) - h_\varepsilon(x, t-k)|\hspace{1mm}\mathcal{X}_K(x-k,t-k) \hspace{1mm} dx\hspace{0.5mm}dt,\\
\label{4.4}
&I_{ii} := \iint_{\mathbb{R}^2} |h_\varepsilon(x-k,t-k)|\hspace{1mm}|\mathcal{X}_K(x-k,t-k) - \mathcal{X}_K(x,t-k)|\hspace{1mm} dx\hspace{0.5mm}dt.
\end{align}
Notice that since $(h_0(x-k) - h_0(x))^+ \in L^1(\mathbb{R})$ then after doing the change $\Tilde{t} = t-k$ we have
\begin{align}
\label{4.5}
    I_i &\leq \iint_{\mathbb{R}\times (0,T)} |h_\varepsilon(x-k,t-k) - h_\varepsilon(x, t-k)| \hspace{1mm} dx\hspace{0.5mm}dt \\
    & \leq \iint_{\mathbb{R}\times (-k,T-k)} |h_\varepsilon(x-k,\Tilde{t}) - h_\varepsilon(x, \Tilde{t})| \hspace{1mm} dx\hspace{0.5mm}d\Tilde{t} \nonumber\\
    & \leq T \int_{\mathbb{R}} |h_0(x-k) - h_0(x)| \hspace{1mm}dx,\nonumber
\end{align}
where we have used the Theorem \ref{teo 2.5} in the last inequality. On the other hand, by the $L^\infty$ bound of Theorem \ref{teo 2.5} and after a change of variable we can bound $I_{ii}$ by
\begin{align}
\label{4.6}
    I_{ii} \leq \norm{h_0}_{L^\infty(\mathbb{R})}\iint_{\mathbb{R}^2} |\mathcal{X}_K(x-k,t-k) - \mathcal{X}_K(x,t-h)|\hspace{1mm} dx\hspace{0.5mm}dt.
\end{align}
Thus, thanks to \eqref{4.5} and \eqref{4.6} we deduce
\begin{equation}
\label{bound 1}
I\leq C\left(\norm{h_0(\cdot -k) - h_0}_{L^1(\mathbb{R})} + \norm{h_0}_{L^\infty(\mathbb{R})}\norm{\mathcal{X}_K(\cdot -k, \cdot -k) - \mathcal{X}_K(\cdot, \cdot -k)}_{L^1(\mathbb{R}^2)}\right),
\end{equation}
with $C > 0$ a constant that only depends on $T$.

For integral $II$ we have
\begin{align*}
    \iint_{\mathbb{R}^2} |\Tilde{h}_\varepsilon(x,t-k) - \Tilde{h}_\varepsilon(x,t)| \hspace{1mm} dx\hspace{0.5mm}dt &\leq II_i + II_{ii}
\end{align*}
with
\begin{align}
\label{4.7}
    & II_i := \iint_{\mathbb{R}^2} |h_\varepsilon(x,t-k) - h_\varepsilon(x, t-k)|\hspace{1mm}\mathcal{X}_K(x,t-k) \hspace{1mm} dx\hspace{0.5mm}dt,\\
\label{4.8}
&II_{ii} := \iint_{\mathbb{R}^2} |h_\varepsilon(x,t)|\hspace{1mm}|\mathcal{X}_K(x,t-k) - \mathcal{X}_K(x,t)|\hspace{1mm} dx\hspace{0.5mm}dt.
\end{align}
Arguing as before we can bound the integrals \eqref{4.7} and \eqref{4.8} by
\begin{align*}
    &II_i \leq T \iint_{\mathbb{R}} |h_0(x-k) - h_0(x)| \hspace{1mm}dx,\\
    & II_{ii} \leq \norm{h_0}_{L^\infty(\mathbb{R})}\int_{\mathbb{R}^2} |\mathcal{X}_K(x,t-k) - \mathcal{X}_K(x,t)|\hspace{1mm} dx\hspace{0.5mm}dt.
\end{align*}
Thus,
\begin{equation}
\label{bound 2}
II \leq C\left(\norm{h_0(\cdot -k) - h_0}_{L^1(\mathbb{R})} + \norm{h_0}_{L^\infty(\mathbb{R})}\norm{\mathcal{X}_K(\cdot, \cdot -k) - \mathcal{X}_K(\cdot, \cdot )}_{L^1(\mathbb{R}^2)}\right)
\end{equation}

Finally, thanks to \eqref{bound 1} and \eqref{bound 2} we can bound \eqref{4.2} by
\begin{align*}
    &\norm{\tilde{h}_\varepsilon(\cdot - h, \cdot -k) - \Tilde{h}_\varepsilon}_{L^1(\mathbb{R})} \leq C(\norm{h_0(\cdot -k) - h_0}_{L^1(\mathbb{R})} \\
    &+ \norm{h_0}_{L^\infty(\mathbb{R})}\norm{\mathcal{X}_K(\cdot -k, \cdot -k) - \mathcal{X}_K(\cdot, \cdot -k)}_{L^1(\mathbb{R}^2)}\\
    & \norm{h_0}_{L^\infty(\mathbb{R})}\norm{\mathcal{X}_K(\cdot, \cdot -k) - \mathcal{X}_K(\cdot, \cdot )}_{L^1(\mathbb{R}^2)}).
\end{align*}
By the continuity of the translations in $L^1$ we conclude that the sequence $\{\Tilde{h}_\varepsilon(x,t)\}_\varepsilon$ is equicontinuous. Thus, by the Fr\'echet-Kolmogorov Theorem we conclude that the set $\{\Tilde{h}_\varepsilon(x,t)\}_\varepsilon$ relatively compact in $L^1(\mathbb{R}^2)$, i.e., there exists a convergent subsequence of $\{\Tilde{h}_\varepsilon(x,t)\}_\varepsilon$ in $L^1(\mathbb{R}^2)$. In particular, by the definition of $\Tilde{h}_\varepsilon(x,t)$ we conclude that there exists a convergent subsequence of $\{h_\varepsilon(x,t)\}_\varepsilon$ in $L^1(K)$. Since the compact $K \subset \mathbb{R}\times (0,T)$  is arbitrary, by a covering argument, the convergence of the subsequence holds in $L^1_{loc}(\mathbb{R}\times (0,T))$.

\medskip

\noindent\underline{\textit{Step 2}}: In the first step we have shown that there exists a subsequence $\{h_{\varepsilon_j}(x,t)\}_j$ of $\{h_\varepsilon(x,t)\}_\varepsilon$ such that
\[
\{h_{\varepsilon_j}(x,t)\}_j \rightarrow g \qquad \textup{in} \qquad L^1_{loc}(\mathbb{R}\times (0,T)),
\]
for some $g \in L^1(\mathbb{R}\times (0,T))$, when $\varepsilon_j \to 0$. Now, we want to prove that the function $g $ is the unique solution of the problem  \eqref{stefan}-\eqref{IC}, i.e., $g = h$.
By the uniqueness theorem  of \eqref{stefan}-\eqref{IC} it is enough to show that $\forall \psi \in C^\infty_c(\mathbb{R}\times [0,T))$
 \begin{equation}
 \label{4.11}
     \int_0^T\int_{\mathbb{R}}(g \,\partial_t \psi - \Phi(g)(-\Delta)^{s}\psi)\hspace{1mm}dx\hspace{0.5mm}dt + \int_{\mathbb{R}}h_0(x)\psi(x,0)\hspace{1mm}dx = 0.
 \end{equation}
Since $h_{\varepsilon_j}$ is a solution of \eqref{59}-\eqref{IC} we have that $\forall \psi \in C^\infty_c(\mathbb{R}\times [0,T))$
 \begin{equation}
 \label{4.12}
     \int_0^T\int_{\mathbb{R}}(h_{\varepsilon_j} \partial_t \psi - \Phi_{\varepsilon_j}(h_{\varepsilon_j})(-\Delta)^{s}\psi)\hspace{1mm}dx\hspace{0.5mm}dt + \int_{\mathbb{R}}h_0(x)\psi(x,0)\hspace{1mm}dx = 0.
 \end{equation}
 Thus, if we prove
 \begin{equation}
 \label{4.13}
     \lim_{\varepsilon_j \to 0}\left(\int_0^T\int_{\mathbb{R}} (h_{\varepsilon_j} - g)\partial_t \psi - (\Phi_{\varepsilon_j}(h_{\varepsilon_j}) - \Phi(g))(-\Delta)^s \psi\hspace{1mm}dx\hspace{0.5mm}dt\right) = 0,
 \end{equation}
 passing to the limit in \eqref{4.12} we obtain \eqref{4.11}. Then we only have to study the limit \eqref{4.13}.

Since $\psi \in C^\infty_c(\mathbb{R}\times [0,T))$ has compact support, we can estimate
\begin{align}\label{4.14}
    & \left|\int_0^T\int_{\mathbb{R}} (h_{\varepsilon_j} - g)\partial_t \psi - (\Phi_{\varepsilon_j}(h_{\varepsilon_j}) - \Phi(g))(-\Delta)^s \psi\hspace{1mm}dx\hspace{0.5mm}dt\right| \nonumber  \\
    & \leq \norm{\partial_t \psi}_{L^\infty(\overline{Q}_1)} \int_{\overline{Q}_1}|h_{\varepsilon_j} - g|\hspace{1mm}dx\hspace{0.5mm}dt
    + \int_0^T\int_{\mathbb{R}} |\Phi_{\varepsilon_j}(h_{\varepsilon_j}) - \Phi(g)| |(-\Delta)^s \psi|\hspace{1mm}dx\hspace{0.5mm}dt,
\end{align}
for some compact cylinder $\overline{Q} \subset \mathbb{R} \times (0,T)$. Let $\overline{B}_R \times [0,t_0]\subset \mathbb{R} \times (0,T)$ another compact cylinder, with $t_0 \in (0,T)$ and $R > 0$ a constant that will be determined later. Then we have from \eqref{4.14}
\begin{align*}
    & \left|\int_0^T\int_{\mathbb{R}} (h_{\varepsilon_j} - g)\partial_t \psi - (\Phi_{\varepsilon_j}(h_{\varepsilon_j}) - \Phi(g))(-\Delta)^s \psi\hspace{1mm}dx\hspace{0.5mm}dt\right| \leq \norm{\partial_t \psi}_{L^\infty(\overline{Q})} \int_{\overline{Q}}|h_{\varepsilon_j} - g|\hspace{1mm}dx\hspace{0.5mm}dt\nonumber \\
    & \hspace{10mm}+ \int_0^{t_0}\int_{\overline{B}_R} |\Phi_{\varepsilon_j}(h_{\varepsilon_j}) - \Phi(g)| |(-\Delta)^s \psi|\hspace{1mm}dx\hspace{0.5mm}dt + \int_{t_0}^{T}\int_{\mathbb{R}\setminus\overline{B}_R} |\Phi_{\varepsilon_j}(h_{\varepsilon_j}) - \Phi(g)| |(-\Delta)^s \psi|\hspace{1mm}dx\hspace{0.5mm}dt\\
    & \leq \norm{\partial_t \psi}_{L^\infty(\overline{Q})} \int_{\overline{Q}}|h_{\varepsilon_j} - g|\hspace{1mm}dx\hspace{0.5mm}dt + \norm{(-\Delta)^s\psi}_{L^\infty(\overline{B}_R \times [0,t_0])}\int_0^{t_0}\int_{\overline{B}_R} |\Phi_{\varepsilon_j}(h_{\varepsilon_j}) - \Phi(g)|\hspace{1mm}dx\hspace{0.5mm}dt\\
    & \hspace{20mm}+\norm{\Phi_{\varepsilon_j}(h_{\varepsilon_j}) - \Phi(g)}_{L^\infty(\mathbb{R}\times (0,T))}\int_{t_0}^{T}\int_{\mathbb{R}\setminus\overline{B}_R}|(-\Delta)^s \psi| \hspace{1mm}dx\hspace{0.5mm}dt = I + II + III,\\
\end{align*}
with
\begin{align*}
    & I = \norm{\partial_t \psi}_{L^\infty(\overline{Q})} \int_{\overline{Q}}|h_{\varepsilon_j} - g|\hspace{1mm}dx\hspace{0.5mm}dt,\\
    &II = \norm{(-\Delta)^s\psi}_{L^\infty(\overline{B}_R \times [0,t_0])}\int_0^{t_0}\int_{\overline{B}_R} |\Phi_{\varepsilon_j}(h_{\varepsilon_j}) - \Phi(g)|\hspace{1mm}dx\hspace{0.5mm}dt,\\
    & III = \norm{\Phi_{\varepsilon_j}(h_{\varepsilon_j}) - \Phi(g)}_{L^\infty(\mathbb{R}\times (0,T))}\int_{t_0}^{T}\int_{\mathbb{R}\setminus\overline{B}_R}|(-\Delta)^s \psi| \hspace{1mm}dx\hspace{0.5mm}dt.
\end{align*}

Firstly, we study the integral $I$. Given $\Tilde{\varepsilon} > 0$ fix, thanks to the Step 1, since $h_{\varepsilon_j} \rightarrow g$ in $L^1_{loc}(\mathbb{R}\times (0,T))$, there exists a $\delta_0\in (0,1)$ such that $\forall \varepsilon_j \leq \delta_0$
\begin{equation}
    \label{4.15}
    \norm{\partial_t \psi}_{L^\infty(\overline{Q})} \int_{\overline{Q}}|h_{\varepsilon_j} - g|\hspace{1mm}dx\hspace{0.5mm}dt \leq \frac{\Tilde{\varepsilon}}{3}.
\end{equation}

On the other hand, we can bound the integral $II$ by
\begin{align}
\label{4.16}
    \norm{(-\Delta)^s\psi}&_{L^\infty(\overline{B}_R \times [0,t_0])}\int_0^{t_0}\int_{\overline{B}_R} |\Phi_{\varepsilon_j}(h_{\varepsilon_j}) - \Phi(g)|\hspace{1mm}dx\hspace{0.5mm}dt \leq \\
    &\leq \norm{(-\Delta)^s\psi}_{L^\infty(\overline{B}_R \times [0,t_0])} \left(\int_0^{t_0}\int_{\overline{B}_R} |\Phi_{\varepsilon_j}(h_{\varepsilon_j}) - \Phi_{\varepsilon_j}(g)|\hspace{1mm}dx\hspace{0.5mm}dt + \int_0^{t_0}\int_{\overline{B}_R} |\Phi_{\varepsilon_j}(g) - \Phi(g)|\hspace{1mm}dx\hspace{0.5mm}dt\right).\nonumber\\
    & \leq \norm{(-\Delta)^s\psi}_{L^\infty(\overline{B}_R \times [0,t_0])} \left(\int_0^{t_0}\int_{\overline{B}_R} |h_{\varepsilon_j} - g|\hspace{1mm}dx\hspace{0.5mm}dt + \int_0^{t_0}\int_{\overline{B}_R} |\Phi_{\varepsilon_j}(g) - \Phi(g)|\hspace{1mm}dx\hspace{0.5mm}dt\right)\nonumber,
\end{align}
where we have used that $\Phi_{\varepsilon_j}\in C^{0,1}(\mathbb{R})$ and $[\Phi_{\varepsilon_j}]_{C^{0,1}(\mathbb{R})} = 1$ for all $\varepsilon_j \in (0,1)$.
Then since $\Phi_{\varepsilon_j}(x) \rightarrow \Phi(x)$ and $\overline{B}_R \times [0,t_0]$ is a compact, arguing like as in the integral $I$ we have that there exist a $\delta_1 > 0$ such that $\forall \varepsilon_j \leq \delta_1$ we can bound \eqref{4.16} by
\begin{equation}
    \label{4.17}
    \norm{(-\Delta)^s\psi}_{L^\infty(\overline{B}_R \times [0,t_0])}\int_0^{t_0}\int_{\overline{B}_R} |\Phi_{\varepsilon_j}(h_{\varepsilon_j}) - \Phi(g)|\hspace{1mm}dx\hspace{0.5mm}dt \leq \frac{\Tilde{\varepsilon}}{3}.
\end{equation}

For the integral $III$, since $|(-\Delta)^s \psi (x,t)| \leq \frac{C}{(1 + |x|)^{1+2s}}$, we have that for $R>0$  big we can bound $III$ by
\begin{equation}
\label{4.18}
    \norm{\Phi_{\varepsilon_j}(h_{\varepsilon_j}) - \Phi(g)}_{L^\infty(\mathbb{R}\times (0,T))}\int_{t_0}^{T}\int_{\mathbb{R}\setminus\overline{B}_R}|(-\Delta)^s \psi| \hspace{1mm}dx\hspace{0.5mm}dt \leq \frac{\Tilde{\varepsilon}}{3}.
\end{equation}

Finally, thanks to the bounds \eqref{4.15}, \eqref{4.16} and \eqref{4.17} we have that for all $\Tilde{\varepsilon} > 0$ there exists $\Tilde{\delta} := \textup{min}\{\delta_0, \delta_1\}$ and $R > 0$ such that $\forall \varepsilon_j \leq \Tilde{\delta} $ then
\[
\left|\int_0^T\int_{\mathbb{R}} (h_{\varepsilon_j} - g)\partial_t \psi - (\Phi_{\varepsilon_j}(h_{\varepsilon_j}) - \Phi(g))(-\Delta)^s \psi\hspace{1mm}dx\hspace{0.5mm}dt\right| \leq \tilde{\varepsilon}.
\]
Thus, the limit \eqref{4.13} holds and we conclude the proof.
\end{proof}

\begin{rem} \rm
    Notice that by  Proposition \ref{prop 4.1} since $h_{\varepsilon_j} \rightarrow h$ in $L^1_{loc}(\mathbb{R}\times (0,T))$ then $h_{\varepsilon_j} \rightarrow h$ a.e. in $\mathbb{R} \times (0,T)$. In particular, since $h_{\varepsilon_j}$ and $h$ are selfsimilar and continuous solutions  then $H_{\varepsilon_j} \rightarrow H$ a.e. in $\mathbb{R}$.
\end{rem}

\begin{rem} \rm
    From this point on we will write $\{H_\varepsilon\}_\varepsilon$ instead of $\{H_{\varepsilon_j}\}_j$ to indicate a subsequence of selfsimilar solutions of \eqref{59}-\eqref{IC} that converges to $H$, the unique solution of \eqref{stefan}-\eqref{IC}.
\end{rem}

The next Lemma deals with the location of the intersection between the level $h=L$ and the sup of the positivity set of the functions  $\{H_\varepsilon\}$ of the approximate sequence.

\begin{lem}
\label{lema 4.19}
    Let $s\in (0,1)$, $\varepsilon \in (0,1)$, $L > 0$ and $\{H_\varepsilon \}_{\varepsilon}$ and sequence of solutions of \eqref{59}-\eqref{IC} such that $H_\varepsilon \rightarrow H$ a.e. in $\mathbb{R}$, with $H$ the unique solution of the problem \eqref{stefan}-\eqref{IC}. Let $\xi_0 > 0$  the free boundary of $H$. If we define $
    \xi_\varepsilon := \textup{sup} \{x\in \mathbb{R}\hspace{1mm}:\hspace{1mm} H_\varepsilon(x) > L\}$, then
   $$
     \liminf_{\varepsilon\to 0} \xi_\varepsilon\ge \xi_0.
    $$
\end{lem}
\begin{proof}
    Let $h > 0$ small and $\Tilde{\xi} := \xi_0 -k$ such that $\Tilde{\xi} \in (0, \xi_0)$ and $A = H(\Tilde{\xi}) > 0$. Since by Proposition \ref{prop 4.1} $H_\varepsilon(\Tilde{\xi}) \rightarrow H(\Tilde{\xi})$, then there exists $\varepsilon_0$ small such that $H_\varepsilon(\Tilde{\xi}) > \frac{A}{2}$ for all $\varepsilon \leq \varepsilon_0$. Thanks to the monotonicity of  $H_\varepsilon$ we deduce that $H_\varepsilon(\xi) > \frac{A}{2}$ for all $\varepsilon \leq \varepsilon_0(h)$ and $\xi \leq \Tilde{\xi}$. Thus, $\xi_\varepsilon > \Tilde{\xi}$. Since we can do the argument for all $\Tilde{\xi} \in (0,\xi_0)$ the  result follows.  Note that a similar result for the approximation from the right-hand side may be false because of possible flat tails\nc.
\end{proof}

%%%%%%%%%%%%%%%%%%%%%%%%%%%%%%%%%%%%%%%%%%%%%%%%%%%%%%%%%%%%%%%%%%%%%%%%%%%%%%%%%%%
\section{$C^{1,\alpha}$ estimates of the selfsimilar solutions} \label{section 4}

The purpose of this section is to prove  that in the subcritical case, $0<s < \frac{1}{2}$, the selfsimilar solutions of the regularized problem \eqref{59}, with the same initial data as in the Theorem \ref{teo 2.12}, not only are smooth in $\mathbb{R}$ but they also  satisfy a property of \textit{improvement of regularity} that allows to pass from bounds from $C^{0,1}(\mathbb{R})$ to bounds in $C^{1,\alpha}(\mathbb{R})$ with global constants that do not depend on $\varepsilon$. This is a crucial step of the paper in order to apply this property to the limit solution (see next section). The main result of the section reads as follows:

\begin{thm}
\label{pro princ 4}
    Let $s \in (0,\frac{1}{2})$ and $\varepsilon > 0$ small. Under the assumptions of Theorem \ref{teo 2.12} then
    \begin{equation}
    \label{sec 4 ine}
        \norm{H_\varepsilon}_{C^{1,\alpha}(\mathbb{R})}\leq C_s\left(\norm{H_\varepsilon}_{L^\infty(\mathbb{R})} + \norm{H_\varepsilon}_{C^{0,1}(\mathbb{R})}\right),
    \end{equation}
    for some $\alpha\in (0,1)$ and $C_s > 0$
    a constant that only depends on $s$, not on $\varepsilon$.

\end{thm}

 Let us take an auxiliary $A > 0$ such that $A < \xi_0$, with $\xi_0$ the free boundary point of $H$, the unique solution of the original problem \eqref{stefan}-\eqref{IC}. The proof of estimate \eqref{sec 4 ine} is a consequence of the combination of two partial estimates. The first  one covers the region far to the left
\begin{equation}
    \label{new 2}
[H_\varepsilon]_{C^{1,\alpha}(-\infty, A)} \nc \leq C_d\,\norm{H_\varepsilon}_{L^\infty(\mathbb{R})},
\end{equation}
where a fractional linear heat equation holds in the region $(-\infty, A)$. Result  \eqref{new 2}  is a consequence of Theorem \ref{davila} following \cite{ChD} and will be proved in Subsection \ref{subsec 5.2}.  Here $\alpha = \alpha_1 >0$.

The other estimate covers the region around the free boundary and beyond:
\begin{equation}
    \label{new 3}
    \norm{ H_\varepsilon}_{C^{1,\alpha}(A/2, \infty)}\leq C_s\left(\norm{H_\varepsilon}_{L^\infty(\mathbb{R})} + \norm{H_\varepsilon}_{C^{0,1}(\mathbb{R})}\right),
\end{equation}
 for some $\alpha>0$, and is proved in  Proposition \ref{prop ine 2} below. Here we may take $\alpha_2=(1-2s)/2>0$.
We point out that the constants appearing in these two inequalities  only depend on $s$, but not on $\varepsilon$.

\begin{rem} \rm Note that the domains of both inequalities overlap, so that we get a joint $C^{1,\alpha}$ estimate in all of $\mathbb{R}$ with $\alpha=\min\{\alpha_1, \alpha_2\}$. The reader may wonder why two overlapping estimates and not just one. The reason is that estimate \eqref{new 2} does not apply in regions where the free boundary $\xi_0$ is included since we need a neighborhood where the linear equation holds. On the other hand,  estimate \eqref{new 3} is more specific but does not apply in any domain that includes $\xi=0$\nc.
\end{rem}

\nc
%%%%%%%%%%%%%%%%%%%%%%%%%%%%%%%%%%%%%%%%%%%%%%%%%%%%%%%%%%%%%%%%%%%%%%%%%%%%%%%%%%
\subsection{Proof of inequality \eqref{new 2} to the left of $A$}\label{subsec 5.2}
We sketch the proof of the $C^{1,\alpha}$ regularity of $H$ and $H_\varepsilon$ in any region to the left of the free boundary, more precisely in the region $(-\infty, A)$ with $0< A < \xi_0$. We want to  use Theorem \ref{davila}. The latter applies in unitary domains while the solution $h$ and  $h_\varepsilon$ of the equivalent parabolic problem is defined in a region somewhat to the left of the parabolic free boundary. We proceed as follows.

(i) Let us prove the result for $H$. Notice that by Theorem \ref{teo 4} the free boundary of the selfsimilar solution $H(\xi)$ is $\xi_0>0$  so that the free boundary of the parabolic solution $h(\cdot, t)$  is  $X(t)=\xi_0\, t^{\frac{1}{2s}}$, which ensures that the parabolic free boundary  advances as the time passes. Then, we can choose a time $t_0 >1$ large enough so  that, after a convenient space-time translation, the set $Q=B_4(0)\times (-1,0]$ mentioned in Theorem \ref{davila} is fully contained in the region $x<X(t)$ where $h$ solves a linear fractional equation, so the Theorem applies.

The details are  as follows: we want to make sure that for any point $x_1\le A\, t_0^{\frac{1}{2s}}$, the translation of $Q$ of the form $\hat Q_4= B_4(x_1)\times (t_0-1,t_0]$ is always  contained in the region
 $$
 R=\{(x,t): t_0-1<t\le t_0, -\infty<x< \xi_0\, t^{\frac{1}{2s}} \}.
 $$
 The needed condition is $x_1+4\le \xi_0\, (t_0-1)^{\frac{1}{2s}}$.   Then Theorem \ref{davila} can be applied to the function
 $h_1(x,t)= h(x-x_1,t+t_0  )$ in the set $\widehat Q_{1/2}= B_{1/2}(x_1)\times (t_0-(1/2),t_0]$.  Now observe that
 $$
 h(x,t_0)= H(x \, \, t^{-\frac{1}{2s}}),
 $$
 which means that formula \eqref{form.Hchdv} with $t=s=t_0$ implies the H\"older regularity of $H$ with a constant that depends on $\xi_0$ and $t_0$. If $|x_2-x_1| \le 1$ then
\[
|\partial_{x}h(x_2, t_0) - \partial_{x}h(x_1, t_0)|\leq C_d\,\norm{H}_{L^\infty(\mathbb{R})}\, |x_2 -x_1|^\alpha.
\]
Here, $\alpha$ is a universal constant that depends on $s$. The estimate for $h(\cdot,t_0)$ immediately translates into a similar estimate for $H$ due to the selfsimilar formula.\nc

(ii) Let us now prove the result for $H_\varepsilon$.
    Notice that the approximations $H_\varepsilon$ satisfy the linear fractional heat equation in a set where the solution is positive which by Lemma \ref{lema 4.19} is $(-\infty, \xi_\varepsilon)$ where the end-points  $\xi_\varepsilon$ accumulate as $ \varepsilon\to 0$ at points to the right  of $\xi_0$. This means that the previous proof applies to $H_\varepsilon$
    with $\varepsilon > 0$ enough. In that case estimate \eqref{new 2}  holds. \nc

%%%%%%%%%%%%%%%%%%%%%%%%%%%%%%%%%%%%%%%%%%%%%%%%%%%%%%%%%%%%%%%%%%%%%%%%%%%%%%%%%
\subsection{Proof of the inequality \eqref{new 3}}

We use a different argument to obtain $C^{1,\alpha}$ regularity of $H$ is a region that covers the free boundary. As a preliminary to the proof of inequality \eqref{new 3} we prove a technical lemma that reflects the influence of the factor $\xi$ that accompanies $H'(\xi)$\nc.

\begin{lem}
\label{aco}
     Let $s \in (0,\frac{1}{2})$, $\varepsilon \in (0,1)$ and $A \in (0, \xi_0)$ with $\xi_0 > 0$ the free boundary point of the selfsimilar solution of the problem \eqref{stefan}-\eqref{IC}. Under the assumptions of Theorem \ref{teo 2.12} then
     \begin{equation}
       \label{bound}
      \norm{\xi H'_\varepsilon}_{L^\infty(A/2, \infty)} \leq C_s\left(\norm{U_\varepsilon}_{L^\infty(\mathbb{R})} + \norm{U_\varepsilon}_{C^{0,1}(\mathbb{R})}\right),
     \end{equation}
     where $C_s > 0$ is a constant that only depends on $s$ and not on $\varepsilon$.
\end{lem}

\begin{proof}
    By Theorem \ref{teo 2.12}-$a)$ we know that
    \[
     -\frac{1}{2s}\xi H'_\varepsilon(\xi) + (-\Delta)^s U_\varepsilon(\xi) = 0 \qquad \textup{in} \ \mathbb{R}.
    \]
    Taking $\delta > 0$ small and $\xi \in (A/2, \infty)$  we divide the proof in two cases depending on whether the small neighborhood of $\xi$, $B_\delta(\xi) $, is contained in $(A/2, \infty)$ or otherwise $B_\delta(\xi) \not\subset (A/2, \infty)$.

     (i) Firstly suppose that $B_\delta(\xi) \subset (A/2, \infty)$. Since by Lemma \ref{lemma 2.11} $U_\varepsilon \in C^{0,1}((A/2, \infty))$ then we have
    \begin{align*}
    \left|\xi H'_{\color{blue}\varepsilon}(\xi)\right| &= 2s |(-\Delta)^s U_\varepsilon(\xi)| = 2s\left|\int_{\mathbb{R}} \frac{U_\varepsilon(\xi) - U_\varepsilon(\eta)}{|\xi- \eta|^{1+2s}} \hspace{1mm}d\eta \right| \leq 2s \int_{\mathbb{R}} \frac{|U_\varepsilon(\xi) - U_\varepsilon(\eta)|}{|\xi- \eta|^{1+2s}} \hspace{1mm}d\eta\\
    &= 2s\left(\int_{B_{\delta}(\xi)} \frac{|U_\varepsilon(\xi) - U_\varepsilon(\eta)|}{|\xi- \eta|^{1+2s}} \hspace{1mm}d\eta + \int_{\mathbb{R} \setminus B_{\delta}(\xi)} \frac{|U_\varepsilon(\xi) - U_\varepsilon(\eta)|}{|\xi- \eta|^{1+2s}} \hspace{1mm}d\eta\right)\\
    & =2s\left([U_\varepsilon]_{C^{0,1}((\xi_0/2, \infty))}\int_{B_{\delta}(\xi)} \frac{1}{|\xi- \eta|^{2s}} \hspace{1mm}d\eta + \norm{U_\varepsilon}_{L^\infty(\mathbb{R})}\int_{\mathbb{R} \setminus B_{\delta}(\xi)} \frac{1}{|\xi- \eta|^{1+2s}} \hspace{1mm}d\eta\right)\\
    & \leq C_s\left(\norm{U_\varepsilon}_{L^\infty(\mathbb{R})} + [U_\varepsilon]_{C^{0,1}(\mathbb{R})}\right).
     \end{align*}

     (ii) Now suppose that $\xi -\delta < A/2$. Since $H_\varepsilon$ is a bounded selfsimilar solution we have by Theorem \ref{davila} the estimate
     \begin{equation*}
\norm{U_\varepsilon}_{C^{1,\alpha}(B_\delta(\xi))}\leq C_d\norm{U_\varepsilon}_{L^\infty(\mathbb{R})},
     \end{equation*}
     with $C_d > 0$ a constant that only depends on $s$. In particular we have
     \begin{equation}
         \label{dav 1}[U_\varepsilon]_{C^{0,1}(B_\delta(\xi))} \leq C_d\norm{U_\varepsilon}_{L^\infty(\mathbb{R})}.
     \end{equation}
     Thus,
     \begin{align*}
    \left|\xi H'(\xi)\right| &\leq 2s\left(\int_{B_{\delta}(\xi)} \frac{|U_\varepsilon(\xi) - U_\varepsilon(\eta)|}{|\xi- \eta|^{1+2s}} \hspace{1mm}d\eta + \int_{\mathbb{R} \setminus B_{\delta}(\xi)} \frac{|U_\varepsilon(\xi) - U_\varepsilon(\eta)|}{|\xi- \eta|^{1+2s}} \hspace{1mm}d\eta\right)\\
    & =2s\left([U_\varepsilon]_{C^{0,1}(B_\delta(\xi))}\int_{B_{\delta}(\xi)} \frac{1}{|\xi- \eta|^{2s}} \hspace{1mm}d\eta + \norm{U_\varepsilon}_{L^\infty(\mathbb{R})}\int_{\mathbb{R} \setminus B_{\delta}(\xi)} \frac{1}{|\xi- \eta|^{1+2s}} \hspace{1mm}d\eta\right)\\
    & \leq C_s\left(\norm{U_\varepsilon}_{L^\infty(\mathbb{R})} + [U_\varepsilon]_{C^{0,1}(B_\delta(\xi))}\right) \leq  C'_s\norm{U_\varepsilon}_{L^\infty(\mathbb{R})},
     \end{align*}
     where we have used  estimate \eqref{dav 1} in the last inequality. This concludes the proof.
\end{proof}

The next lemma improves the regularity of $\xi H_\varepsilon'(\xi)$ into H\"older regularity. This is the key point of the section.

 \begin{lem}\label{lem5.2} Under the same  assumptions, if $U_\varepsilon\in C^{0,1}(\mathbb{R})$ we prove that $\xi\,H_\ve' (\xi)$ is H\"older continuous with exponent at least $\alpha=1-2s$. We have for $\xi_1, \xi_2\ge A/2$
\begin{equation}
\label{newHolder}
        |\xi_2H_\varepsilon'(\xi_2) - \xi_1H_\varepsilon'(\xi_1)| \leq C [U_\varepsilon]_{C^{0,1}(\mathbb{R})} |\xi_2 - \xi_1|^{1-2s},
\end{equation}
where $C>0$ does not depend on $\ve$.
\end{lem}

\begin{proof}  Let $\xi_2, \xi_1 \in (A/2, \infty)$ and   $k= |\xi_2 - \xi_1|$.
Without loss of regularity we can concentrate on the case when $k < 1$. We use the stationary equation that satisfies $H_\varepsilon$, see \eqref{65}. It will be convenient to use an equivalent version of the formula for the 1D fractional Laplacian
   \begin{equation}
    \label{fl2}
    (-\Delta )^s u(x) = \,\textup{p.v.}\int_{\mathbb{R}} \frac{u(x) - u(x-y)}{|y|^{1+2s}} \hspace{1mm}dy\,,
    \end{equation}
where we have put $y=\xi-\eta$. This allows to simplify the calculations using the symmetry  of the formulas.
Applying this variant in the equation for $H_\varepsilon$ at $\xi_1$ and $\xi_2$ we get the desired expression
    \begin{equation}
    \label{70}
    \xi_2H_\ve'(\xi_2) - \xi_1 H_\ve'(\xi_1) =
    2s \int_{\mathbb{R} }\frac{U_\ve(\xi_2) - U_\ve(\xi_2-y) - U_\ve(\xi_1) + U_\ve(\xi_1-y)}{|y|^{1+2s}} \hspace{1mm}dy.
    \end{equation}

   \noindent     $\bullet$  We split the integral in the right-hand side into three domains of integration, the first being the left tail $y\in (-\infty,-d)$, the second
   $(-d, d)$ contains small increments around the points under scrutiny, and finally the right tail $(d,\infty)$. We take $d > k$  that will be determined later. We study each integral separately.
    \begin{align}
    \label{71}
        |I | &=\left|\int_{-\infty}^{- d}
        \frac{ U_\ve(\xi_2) - U_\ve(\xi_1)}{|y|^{1+2s}} \hspace{1mm}dy +
        \int_{-\infty}^{ - d}
        \frac{U_\ve(\xi_1-y) - U_\ve(\xi_2-y)}{|y|^{1+2s}} \hspace{1mm}dy \right|\nonumber\\
        & \leq 2[U_\varepsilon]_{C^{0,1}(\mathbb{R})} k \int_{-\infty}^{- d}
        \frac{dy}{|y|^{1+2s}} = [U_\varepsilon]_{C^{0,1}(\mathbb{R})} \frac{1}{s}\frac{k}{d^{2s}}
        \end{align}
The same argument and result  applies to the integral $III$ on the right-hand  tail $(d,\infty)$.

\noindent    $\bullet$ Study of integral $II$ that includes values of $U_\ve$ from $\xi_1-d$ to $\xi_2+d$. We have
\begin{equation}
    \label{76}
        II =        \int_{-d}^{d}
        \frac{ U_\ve(\xi_2) - U_\ve(\xi_2-y)}{|y|^{1+2s}} \hspace{1mm}dy +\int_{-d}^{d} \frac{U_\ve(\xi_1-y) - U_\ve(\xi_1)}{|y|^{1+2s}} \hspace{1mm}dy.
\end{equation}
Using the fact that $U_\ve \in C^{0,1}(\mathbb{R})$ we get $|U_\ve(\xi_i) - U_\ve(\xi_i-y)|\le |[U_\ve]_{C^{0,1}(\mathbb{R})} y |$, hence
    \begin{equation}
    \label{76b}
       | II |           \le     2 \int_{-d}^{d}\frac{[U_\ve]_{C^{0,1}(\mathbb{R})} |y|}{|y|^{1+2s}} \hspace{1mm}dy =
            2 [U_\ve]_{C^{0,1}(\mathbb{R})} \int_{-d}^{d} \frac{dy}{|y|^{2s}}= \frac{4}{1-2s}[U_\ve]_{C^{0,1}(\mathbb{R})} d^{1-2s}.
\end{equation}

 By combining \eqref{71} and \eqref{76b} with \eqref{70}, we arrive at the following result
 \begin{equation}
    \label{77}
    |\xi_2H_\ve'(\xi_2) - \xi_1 H_\ve'(\xi_1)| \le   [U_\varepsilon]_{C^{0,1}(\mathbb{R})} (\frac{4k}{d^{2s}} + \frac{8s}{1-2s}  d^{1-2s}).
    \end{equation}
\noindent    $\bullet$  We still have to find an optimal choice  $d$ as a function of $k$.  If in \eqref{77} we choose the value of $d = ck $ with a constant $c\ge 1$, we obtain
    \begin{align}
    \label{86}
        |\xi_2H_\varepsilon'(\xi_2) - \xi_1H_\varepsilon'(\xi_1)| \leq C [U_\varepsilon]_{C^{0,1}(\mathbb{R})} k^{1-2s}, \qquad
        C= \frac{4}{c^{2s}} + \frac{4sc^{1-2s}}{1-2s},
    \end{align}
which is the desired estimate, recalling that $k=|\xi_2-\xi_1|$\nc.
    \end{proof}

\nc  We may now complete the  result of the section:

\begin{prop}
\label{prop ine 2}
     Let $s \in (0,\frac{1}{2})$, $\varepsilon \in (0,1)$ and $A \in (0, \xi_0)$ with $\xi_0 > 0$ the free boundary point of the selfsimilar solution of  problem \eqref{stefan}-\eqref{IC} under the assumptions of Theorem \ref{teo 2.12}. Then,
     \[
      \norm{ H_\varepsilon}_{C^{1,\alpha}(A/2, \infty)}\leq C_s\left(\norm{H_\varepsilon}_{L^\infty(\mathbb{R})} + \norm{H_\varepsilon}_{C^{0,1}(\mathbb{R})}\right),
     \]
     where $C_s > 0$ is a constant that depends on $s$ and not on $\varepsilon$. Here the H\"older exponent we find is $\alpha_2=1-2s$.\nc
\end{prop}

\begin{proof}
  Notice that since $A > 0$ then $1/\xi \in C^\infty(A/2,\infty) \cap C^{0,1}(A/2,\infty)$. Hence, we can get the bound
     \[
[H'_\varepsilon]_{C^\alpha(A/2, \infty)}\leq \norm{1/\xi}_{L^\infty(A/2, \infty)}[\xi H'_\varepsilon]_{C^{\alpha}(A/2, \infty)} + \norm{\xi H'_\varepsilon}_{L^\infty(A/2, \infty)}[1/\xi]_{C^{\alpha}(A/2, \infty)}.
\]
In particular, using the inequality \eqref{bound} of Lemma \ref{aco} we have
\begin{equation}
\label{ine ine}
[H'_\varepsilon]_{C^\alpha(A/2, \infty)}\leq C_1 [\xi H'_\varepsilon]_{C^{\alpha}(A/2, \infty)} + C_2 \left(\norm{U_\varepsilon}_{L^\infty(\mathbb{R})} + \norm{U_\varepsilon}_{C^{0,1}(\mathbb{R})}\right),
\end{equation}
where $C_1, C_2 > 0$ are constants that depend on $s$ and the choice of $A$.
Using Lemma \ref{lem5.2}  we show that
\begin{equation}
\label{ine to prove}
 [\xi H'_\varepsilon]_{C^{\alpha}(A/2, \infty)}\leq C_s\left(\norm{U_\varepsilon}_{L^\infty(\mathbb{R})} + \norm{U_\varepsilon}_{C^{0,1}(\mathbb{R})}\right),
\end{equation}
with $C_s > 0$ a constant that only depends on $s$. Since \nc $\norm{U_\varepsilon}_{L^\infty(\mathbb{R})}\leq \norm{H_\varepsilon}_{L^\infty(\mathbb{R})}$, $[U_\varepsilon]_{C^{0,1}(\mathbb{R})} \le [H_\varepsilon]_{C^{0,1}(\mathbb{R})}$, we can bound \eqref{ine ine} by
\[
[H'_\varepsilon]_{C^\alpha(A/2, \infty)}\leq C_s\left(\norm{H_\varepsilon}_{L^\infty(\mathbb{R})} + \norm{H_\varepsilon}_{C^{0,1}(\mathbb{R})}\right).
\]
This ends the proof.
\nc \end{proof}

%%%%%%%%%%%%%%%%%%%%%%%%%%%%%%%%%%%%%%%%%%%%%%%%%%%%%%%%%%%%%%%%%%%%%%%%%%%%%%%%%%%%%%%%%%%%%%

\section{Subcritical Case. $C^{1,\alpha}$ regularity}\label{section 5}

The aim of this section is to complete the proof of Theorem \ref{principal}.  The idea we are following is to define a sequence of selfsimilar solutions of the regularized problem \eqref{59}-\eqref{IC} and prove that it is uniformly bounded and equicontinuous. Then, by the Arzelà-Ascoli Theorem we will extract a subsequence that converges uniformly to the unique selfsimilar solution of the problem \eqref{stefan}-\eqref{IC}. Finally, passing to the limit in \eqref{sec 4 ine} we want prove the $C^{1,\alpha}$ regularity using the a priori estimates of the previous section.

\begin{proof}[Proof of Theorem \ref{principal}]
    Recall that thanks to Theorem \ref{teo 4} there exists a unique selfsimilar solution $H$ of \eqref{stefan}-\eqref{IC}. Let $\{\varepsilon_n\}_{n\in \mathbb{N}}$ a sequence that $\varepsilon_n \rightarrow 0$ when $n \rightarrow \infty$, and $\{P_{\varepsilon_n}\}_{n\in \mathbb{N}}$  a sequence of regularized problems with initial data equal to \eqref{IC}. By Lemma \ref{lemma 2.7} we know that for each element of the sequence $\{\varepsilon_n\}_{n\in \mathbb{N}}$ the problem $P_{\varepsilon_n}$ has a unique solution that we denote by $H_{\varepsilon_n}$. Then, from the sequence $\{P_{\varepsilon_n}\}_{n\in \mathbb{N}}$ we can construct a sequence of selfsimilar solutions $\{H_{\varepsilon_n}\}_{n\in \mathbb{N}}$  of the regularized problem. We  show that there exists a subsequence of $\{H_{\varepsilon_n}\}_{n\in \mathbb{N}}$ that converges uniformly to $H$. This fact will be enough to prove $C^{1,\alpha}$ regularity and the estimate \eqref{prin prin}. We divide the remaining proof in two steps: \\

   \noindent\underline{\textit{Step 1}}: Firstly, we prove the existence of a subsequence of   $\{H_{\varepsilon_n}\}_{n\in \mathbb{N}}$ such that converges uniformly on compact subsets of $\mathbb{R}$ to $H$. By Theorem \ref{teo 2.12} and the choice  of the initial data \eqref{IC} we have that
   \begin{equation}
   \label{uniform}
\norm{H_{\varepsilon_n}}_{L^\infty(\mathbb{R})} = L + P_1 \qquad \textup{for all} \quad n\in \mathbb{N},
   \end{equation}
   i.e., the sequence is uniformly bounded.

   Now we want to show the equicontinuity. Thanks to the interpolation inequality, \cite[Lemma 6.35]{7}, we know that for all $\eta > 0$ and $n \in \mathbb{N}$
    \begin{equation}
    \label{87}
        \norm{H_{\varepsilon_n}}_{C^{0,1}(\mathbb{R})} \leq C_{\eta}\norm{H_{\varepsilon_n}}_{L^\infty(\mathbb{R})} + \eta\norm{H_{\varepsilon_n}}_{C^{1,\alpha}(\mathbb{R})},
    \end{equation}
 with $\alpha \in (0,1)$ and $C_\eta > 0$ a constant that only depends on $\eta$ and $\alpha$\nc. By the previous Theorem \ref{pro princ 4} there exists some $\alpha >0$ such that we also have that for all $n \in \mathbb{N}$
    \begin{equation}
    \label{88}
    \norm{H_{\varepsilon_n}}_{C^{1,\alpha}(\mathbb{R})}\leq C_s\left(\norm{H_{\varepsilon_n}}_{L^\infty(\mathbb{R})} + \norm{H_{\varepsilon_n}}_{C^{0,1}(\mathbb{R})}\right),
    \end{equation}
    with $C_s > 0$ a constant that only depends on $s$.
    Using \eqref{88} in \eqref{87} and taking $\eta = \frac{1}{2(1+ C_s)}$ we have
    \begin{align*}
    \norm{H_{\varepsilon_n}}_{C^{0,1}(\mathbb{R})} &\leq C_{\eta}\norm{H_{\varepsilon_n}}_{L^\infty(\mathbb{R})} + \frac{1}{2}\left(\norm{H_{\varepsilon_n}}_{L^\infty(\mathbb{R})} + \norm{H_{\varepsilon_n}}_{C^{0,1}(\mathbb{R})}\right)\\
    & \leq \left(C_\eta + \frac{1}{2}\right)\norm{H_{\varepsilon_n}}_{L^\infty(\mathbb{R})} + \frac{1}{2}\norm{H_{\varepsilon_n}}_{C^{0,1}(\mathbb{R})}.
    \end{align*}
    Thus,
    \begin{equation}
     \label{89}
    \norm{H_{\varepsilon_n}}_{C^{0,1}(\mathbb{R})} \leq 2\left(C_\eta + \frac{1}{2}\right)\norm{H_{\varepsilon_n}}_{L^\infty(\mathbb{R})}.
    \end{equation}
    By \eqref{uniform} we conclude that for all $n \in \mathbb{N}$ we have
    \begin{equation}
     \label{90}
    \norm{H_{\varepsilon_n}}_{C^{0,1}(\mathbb{R})} \leq 2\left(C_\eta + \frac{1}{2}\right)(L + P_1).
    \end{equation}
    In \eqref{90} we have  proved that the family $\{H_{\varepsilon}\}$ is Lipschitz continuous uniformly in space $\xi$ and also w.r.t. $\varepsilon$, \nc hence it is equicontinuous. Thanks to the Arzelà-Ascoli Theorem there exists a subsequence $\{H_{\varepsilon_j}\}_{j\in \mathbb{N}}$ that converges uniformly in compact sets of $\mathbb{R}$ to some function $g \in C(\mathbb{R})$. By the uniqueness of the limit and the Proposition \ref{prop 4.1} we conclude that $g = H$, the unique solution of the problem \eqref{stefan}-\eqref{IC}.
    Notice that taking the limit of the subsequence $\{H_{\varepsilon_j}\}_{j\in \mathbb{N}}$ in \eqref{90} we obtain that $H \in C^{0,1}(\mathbb{R})$.

\medskip

    \noindent\underline{\textit{Step 2}}: Finally, we prove estimate \eqref{prin prin}. Recall that by the \textit{Step 1} there exists a subsequence $\{H_{\varepsilon_{n_{j}}}\}_{n\in \mathbb{N}}$ of $\{H_{\varepsilon_n}\}_{n\in \mathbb{N}}$ that converges uniformly on compact sets of $\mathbb{R}$ to $H$. We define
    \[
    \Tilde{H}_{\varepsilon_{n_j}} := \frac{H_{\varepsilon_{n_j}}}{\norm{H_{\varepsilon_{n_j}}}_{C^{0,1}(\mathbb{R})} + \norm{H_{\varepsilon_{n_j}}}_{L^\infty(\mathbb{R})}}.
    \]
    By \eqref{88} we have
    \begin{equation}
        \label{new}
        \norm{\Tilde{H}_{\varepsilon_{n_j}}}_{C^{1,\alpha}(\mathbb{R})}\leq C_s,
    \end{equation}
    for all $j \in \mathbb{N}$. Since by \eqref{90}, $\norm{H_\varepsilon}_{C^{0,1}(\mathbb{R})} \longrightarrow \norm{H}_{C^{0,1}(\mathbb{R})}$, and by \eqref{uniform}, $ \norm{H_\varepsilon}_{L^\infty(\mathbb{R})} \longrightarrow \norm{H}_{L^\infty(\mathbb{R})}$, then thanks to \cite[H.8]{11} taking the limit in \eqref{new}  we conclude that
    \[\norm{\Tilde{H}}_{C^{1,\alpha}(\mathbb{R})}\leq C_s,
    \]
    with
    \[
    \Tilde{H} := \frac{H}{\norm{H}_{C^{0,1}(\mathbb{R})} + \norm{H}_{L^\infty(\mathbb{R})}}.
    \]
    Hence,
    \begin{equation}
    \label{fifi}
     \norm{H}_{C^{1,\alpha}(\mathbb{R})}\leq C_s\left(\norm{H}_{L^\infty(\mathbb{R})} + \norm{H}_{C^{0,1}(\mathbb{R})}\right).
\end{equation}
Since  $H \in C^{0,1}(\mathbb{R})$ we conclude from \eqref{fifi} that $H\in C^{1,\alpha}(\mathbb{R})$. This concludes the proof.  \end{proof}

\begin{rem} \nc
    Note that if we take the subsequence $\{H_{\varepsilon_n}\}_{n\in \mathbb{N}}$ of \textit{Step 1}, and pass to the limit in \eqref{90} we obtain that $H \in C^{0,1}(\mathbb{R})$. In addition, since $H = U + L$ in $(-\infty, \xi_0)$, then $U \in C^{0,1}(\mathbb{R})$. This is an important fact because  the solution of the problem

\begin{equation*}
\left\{
\begin{aligned}
\partial_t u + (-\Delta)^s u &= 0 \hspace{7mm}\textup{in} \hspace{6mm}  Q=\{(x,t): \ 0<t<T,-\infty <x <\xi_0 \,t^{1/2s} \},\nc\\
u &= 0 \hspace{7mm}\textup{for} \hspace{7mm}  \xi_0  \, t^{1/2s}<x<\infty, \ 0<t<T,\nc\\
u( \cdot,0) &= P_1 \hspace{5.1mm} \textup{for} \hspace{6mm}  -\infty< x < 0,\nc
\end{aligned}
\right.
\end{equation*}
that we are discussing satisfies the regularity \ $u(\cdot,t) \in C^{0,1}(\mathbb{R})$ for all $t > 0$. The effect of the free boundary allows us to prove more regularity than in \rm \cite{14}\sl, where for fixed domains the optimal regularity is $C^s$.
    \end{rem}

\begin{rem}
    We prove global $C^{1,\alpha}$ regularity of the enthalpy $H$. This is a big difference with the classical Stefan problem where the enthalpy has a  vertical jump  at the free boundary point\nc. Moreover, since $H = U+L$  we have proved that the temperature is smooth up to the boundary, i.e., $U \in C^{1,\alpha}(-\infty, \xi_0]$.
\end{rem}

  Finally,  we show that since $H$ and $U$ are regular enough, the selfsimilar solution satisfies equation \eqref{65} in the classical sense.

  \begin{prop}
  \label{propo 5.7}
      Under the assumptions of Theorem \ref{principal} we have that $\partial_t h$ and $(-\Delta)^s u$ exist and
      \[
      \partial_t h(x,t) + (-\Delta)^s u(x,t) = 0 \qquad \textup{for all} \qquad (x,t) \in \mathbb{R}\times (0,T).
      \]
      In the same way, $H'$ and $(-\Delta)^s U$ exist and
      \[
      -\frac{1}{2s}\xi H'(\xi) + (-\Delta)^{s} U(\xi) = 0 \qquad \textup{for all} \qquad  \xi \in \mathbb{R}.
      \]
  \end{prop}
  \begin{proof}
      Since $H$ is a bounded very weak solution of \eqref{stefan}-\eqref{IC} then
      \begin{equation}
        \label{weak}
      \int_0^T\int_{\mathbb{R}}(h\partial_t \psi - \Phi(h)(-\Delta)^{s}\psi)\hspace{1mm}dx\hspace{0.5mm}dt  = 0 \qquad \textup{for all} \qquad \psi \in C_c^\infty(\mathbb{R}\times (0,T)).
      \end{equation}
      On the other hand, by the Theorem \ref{principal} we know that $U \in  C^{0,1}(\mathbb{R})$ and $H \in C^{1,\alpha}(\mathbb{R})$. In particular, $(-\Delta)^{s} U \in C^{1-2s}(\mathbb{R})$ and $H' \in C^{\alpha}(\mathbb{R})$. Hence, integrating by parts in \eqref{weak} we obtain
      \[
      \int_0^T\int_{\mathbb{R}}(\partial_t h \, \psi - (-\Delta)^{s}\Phi(h)\psi)\hspace{1mm}dx\hspace{0.5mm}dt  = 0 \qquad \textup{for all} \qquad \psi \in C_c^\infty(\mathbb{R}\times (0,T)).
      \]
      As a consequence we get that
      \[
      \partial_t h(x,t) + (-\Delta)^s u(x,t) = 0 \qquad \textup{for all} \qquad (x,t) \in \mathbb{R}\times (0,T).
      \]
      The results on $H$ and $U$ follow in a similar way.
  \end{proof}

%%%%%%%%%%%%%%%%%%%%%%%%%%%%%%%%%%%%%%%%%%%%%%%%%%%%%%%%%%%%%%%%%%%%%%%%%%%%%%%%%%%%%%%%%%%%%%%%%%%%%%%%%
\subsection{Improvement of regularity in the water region} \label{section 7}

 We point out that the $\alpha$ of the global H\"older regularity is a positive constant obtained by combining two different estimates, as done in the previous section. Since we already know that $H'$ is Lipschitz continuous in $\mathbb{R}$ we may use   Proposition \ref{prop ine 2} a posteriori  to improve the value of $\alpha$ into the an explicit value  on any closed sub-interval of $(-\infty,0)\cup (0,\xi_0)$. The $C^\infty$ regularity of $H$ in  the ice region
  $(\xi_0,\infty)$ has been proved before by easier methods, cf. {\rm \cite[Theorem 3.1]{DEV1}}\nc.

\begin{lem}\label{lem6.2} Under the current  assumptions, if $U\in C^{0,1}(\mathbb{R})$ and $U'\in C^{\alpha}$ locally in closed  subinterval $K$ of $(-\infty, 0)\cup (0,\xi_0)$ we conclude that $U'\in C^{\beta}(K_1)$ for all $K_1$ strictly contained in $K$. We may take $\beta=(1+\alpha-2s)/(1+\alpha).$
\end{lem}

\begin{proof}  We redo the proof of Lemma \ref{lem5.2} with minor changes. Let $K=[a,b]$ and $K'=[a_1,b_1]$, with inclusions of the form $a+d_0<a_1$, $a_1+d_0<b_1$, $b_1+ d_0<b$. Then as in the mentioned lemma we take
$\xi_2, \xi_1 \in K_1$ and   $k= |\xi_2 - \xi_1|\le d_0$. We add $k<1$. Following the argument of Lemma \ref{lem5.2} we arrive at formula \eqref{70} for the increment
$\xi_2H_\ve'(\xi_2) - \xi_1 H_\ve'(\xi_1)$. We split the estimate of this increment in three subintegrals, and we find nothing to change in the parts labeled $I$ and $III$, However integral $II$ needs special attention.  We recall that
\begin{equation}
    \label{II.1}
        II =        \int_{-d}^{d}
        \frac{ U(\xi_2) - U(\xi_2-y)}{|y|^{1+2s}} \hspace{1mm}dy +\int_{-d}^{d} \frac{U(\xi_1-y) - U(\xi_1)}{|y|^{1+2s}} \hspace{1mm}dy.
\end{equation}
Note that for $\xi_1,\xi_2\in K_1$ we need  $\xi_1\pm d,\xi_2\pm d\in K$.
Using the fact that $U'\in C^{\alpha}(K)$   we have the expression $ U(\xi_2) - U(\xi_2-y)= U'(\xi_2) y + O(y^{1+\alpha}) $, so that
\begin{equation}
\int_{-d}^{d}
        \frac{ U(\xi_2) - U(\xi_2-y)}{|y|^{1+2s}} \,dy= U'(\xi_2) \int_{-d}^{d} \frac{1}{|y|^{1+2s}} \,dy+ C_0
        \int_{-d}^{d} \frac{|y|^{1+\alpha}}{|y|^{1+2s}} \,dy
\end{equation}
where $C_0$ comes from the local H\"older regularity of $U'$. The first integral on the right-hand side cancels out, and from the second we get
 a contribution $C_1 d^{1+\alpha-2s}$.  The same happens with other integral involving $\xi_1$. Hence,
\begin{equation}
    \label{II.2}
        |II | \le C_2 d^{1+\alpha-2s}.
\end{equation}
We combine this inequality with the previous ones for $I$ and $III$ to get
$$
|\xi_2H_\ve'(\xi_2) - \xi_1 H_\ve'(\xi_1)|\le C_3 \frac{4k}{d^{2s}} + C_4\,d^{1+\alpha-2s}
$$
The optimization in the choice of $d$ leads to $k=d^{1+\alpha}$ and then
$$
|\xi_2H_\ve'(\xi_2) - \xi_1 H_\ve'(\xi_1)|\le C_5 \, k^{\beta}, \quad \beta=\frac{1+\alpha-2s}{1+\alpha}.
$$
Eliminating the factor $\xi$ is easy since we are away from $\xi=0$.
\end{proof}

\begin{rem} \rm  The argument does not work near the free boundary, since we  recall that $U'$ is not continuous since $U=(H-L)_+$.
\end{rem}
\medskip

\begin{rem} \rm
We can see that $\beta$ improves $\alpha$ but  only if $\alpha$ is not very large. Figure \ref{fig:xy}  shows the comparison between $y=\beta$ and $y=\alpha$ to show how we can improve by iteration. We will obtain a limit value $\alpha^*$ in the interval $1-2s <\alpha^*< 1-s$. See calculations in Appendix I.\nc
\end{rem}

\begin{figure}[ht!]
 \centering
 \includegraphics[width=10cm]{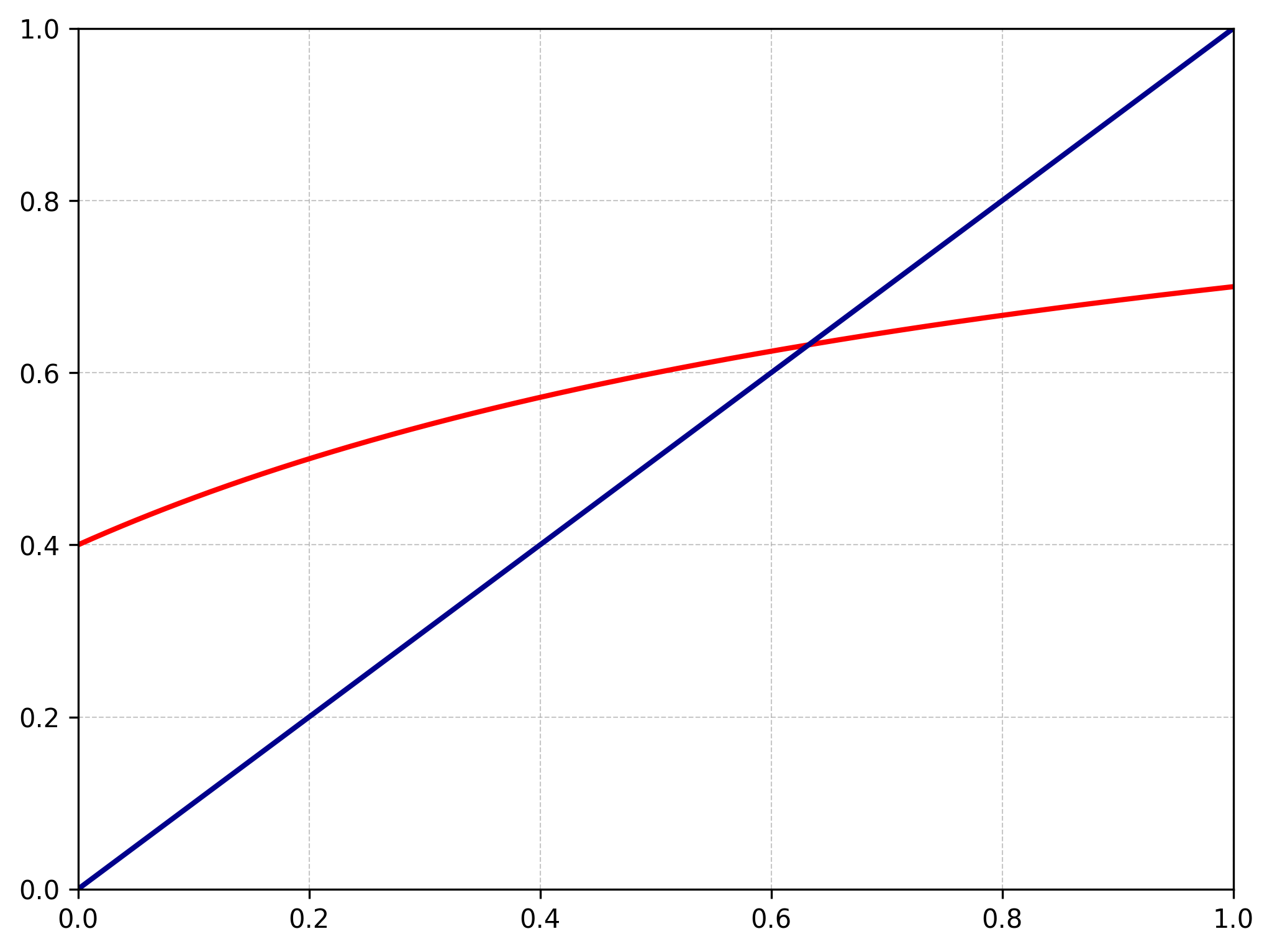}
 \caption{Dependence of $y=\beta$ on $x=\alpha$ in red (for $s=0.3$),  $y=x$ in blue.}
 \label{fig:xy}
\end{figure}

%%%%%%%%%%%%%%%%%%%%%%%%%%%%%%%%%%%%%%%%%%%%%%%%%%%%%%%%%%%%%%%%%%%%%%%%%%%%%%%%%%%%%%
\section{Subcritical FSP. Lateral regularity}\label{section 6}

In this section we study some lateral regularity properties of the Fractional Stefan Problem \eqref{stefan}-\eqref{IC}. We will show the different behaviour of the solution on both sides of the free boundary.   In the literature on free boundaries the situation where $|H'|$ is not zero at least on one side of the free boundary is usually called \textit{non-degenerate} free boundary. It usually allows to prove further regularity\nc.

\subsection{Regularity to the left of $\xi_0$}
 First, we consider the situation to the left of the free boundary, which is one half of formula \eqref{eq1.2}.

\begin{prop}
\label{prop 1.8}
    Let $s \in (0, \frac{1}{2})$. Then the selfsimilar profile satisfies
    \begin{equation}
        \label{41}
        -H'(\xi_0^-) = \frac{2s}{\xi_0}\int_{-\infty}^{\xi_0} \frac{U(\eta)}{|\xi_0 - \eta|^{1+2s}} \hspace{1mm}d\eta > 0.
    \end{equation}
\end{prop}

\begin{proof}  By Theorem \ref{principal} we know that there exists $H'(\xi_0^-)$ and
      \[
      |H'(\xi_0^-)| \leq \frac{2s [U]_{C^{0,1}(\mathbb{R})}}{\xi_0} \int_{\xi_0 - \varepsilon}^{\xi_0} \frac{1}{|\xi_0 - \eta|^{2s}} \hspace{1mm}d\eta  + \frac{2s [U]_{L^{\infty}(\mathbb{R})}}{\xi_0} \int_{ -\infty}^{\xi_0 - \varepsilon} \frac{1}{|\xi_0 - \eta|^{2s}} \hspace{1mm}d\eta < \infty.
      \]
Then for proving \eqref{41} it is enough to prove that for  $\xi_1 < \xi_0$
    \begin{equation}
    \label{42}
        \left|\xi_1H'(\xi_1) + 2s\int_{-\infty}^{\xi_0} \frac{U(\eta)}{|\xi_0 - \eta|^{1+2s}} \hspace{1mm}d\eta\right| \leq C|\xi_1 - \xi_0|^{\frac{1-2s}{2}} ,
    \end{equation}
    for $C > 0$ a constant that only depends on $s$, since taking the limit in \eqref{42} when $\xi_1 \rightarrow \xi_0^-$ we see that \eqref{41} holds.

    The proof of \eqref{42} follows the same ideas as in Theorem  \ref{pro princ 4}. Let $\xi_1\in (-\infty,\xi_0)$  and we denote by $h_1:= |\xi_1 - \xi_0|$, then
    \begin{align}
    \label{43}
        \left|\xi_1H'(\xi_1) + 2s\int_{-\infty}^{\xi_0} \frac{U(\eta)}{|\xi_0 - \eta|^{1+2s}}\hspace{1mm}d\eta\right| &\leq 2s\left|\int_{\mathbb{R} }\frac{U(\xi_1) - U(\eta)}{|\xi_1 - \eta|^{1+2s}} \hspace{1mm}d\eta + \int_{-\infty}^{\xi_0} \frac{U(\eta)}{|\xi_0 - \eta|^{1+2s}} \hspace{1mm}d\eta\right|\nonumber\\
        & \leq 2s (I + II + III),
    \end{align}
    where
    \begin{align*}
        & I := \left|\int_{-\infty}^{\xi_1 -d}\frac{U(\xi_1) - U(\eta)}{|\xi_1 - \eta|^{1+2s}} \hspace{1mm}d\eta + \int_{-\infty}^{\xi_1 - d} \frac{U(\eta)}{|\xi_0 - \eta|^{1+2s}} \hspace{1mm}d\eta\right|,\\
        &  II := \left|\int_{\xi_1 - d}^{\xi_1 +d}\frac{U(\xi_1) - U(\eta)}{|\xi_1 - \eta|^{1+2s}} \hspace{1mm}d\eta + \int_{\xi_1 - d}^{\xi_1 + d} \frac{U(\eta)}{|\xi_0 - \eta|^{1+2s}} \hspace{1mm}d\eta\right|,\\
        & III := \left|\int_{\xi_1 + d}^{\infty} \frac{U(\xi_1)}{|\xi_1-\eta|^{1+2s}}\hspace{1mm}d\eta\right|,
    \end{align*}
    and for $d > h$ that it will be determined later. Now we study each integral separately. Since $U \in C^{0,1}(\mathbb{R})$ we have
    \begin{align*}
        I &=\left|\int_{-\infty}^{\xi_1 - d} U(y)\left(\frac{1}{|\xi_0-\eta|^{1+2s}} - \frac{1}{|\xi_1-\eta|^{1+2s}}\right)\hspace{1mm}d\eta + \int_{-\infty}^{\xi_1 - d} \frac{U(\xi_1)}{|\xi_1-\eta|^{1+2s}} \hspace{1mm} d\eta \right|\\
        & \leq \left|\int_{-\infty}^{\xi_1 - d} U(\eta)\left(\frac{1}{|\xi_0-\eta|^{1+2s}} - \frac{1}{|\xi_1-\eta|^{1+2s}}\right)\hspace{1mm}d\eta\right| + \left|\int_{-\infty}^{\xi_1 - d} \frac{U(\xi_1)}{|\xi_1-\eta|^{1+2s}} \hspace{1mm} d\eta\right|\nonumber\\
        & \leq \norm{U}_{L^\infty(\mathbb{R})}\left|\int_{-\infty}^{\xi_1 - d} \frac{1}{|\xi_0-\eta|^{1+2s}} - \frac{1}{|\xi_1-\eta|^{1+2s}} \hspace{1mm}d\eta\right|\\
        &\hspace{45mm}+ [U]_{C^{0,1}(\mathbb{R})}h\int_{-\infty}^{\xi_1 - d} \frac{1}{|\xi_1-\eta|^{1+2s}} \hspace{1mm} d\eta\nonumber.
       \end{align*}
       Thus,
       \begin{equation}
       \label{44}
           I \leq \norm{U}_{L^\infty(\mathbb{R})}I_i + [U]_{C^{0,1}(\mathbb{R})}hI_{ii},
       \end{equation}
    with
    \begin{align*}
        I_i &:= \left|\int_{-\infty}^{\xi_1 - d} \frac{1}{|\xi_0-\eta|^{1+2s}} - \frac{1}{|\xi_1-\eta|^{1+2s}} \hspace{1mm}d\eta\right|,\\
        I_{ii} &:= \left|\int_{-\infty}^{\xi_1 - d} \frac{1}{|\xi_1-\eta|^{1+2s}} \hspace{1mm} d\eta\right|.
    \end{align*}
    Studying each integral separately we obtain
     \begin{align}
    \label{45}
        I_i &:= \left|\int_{-\infty}^{\xi_1 - d} \frac{1}{|\xi_0-\eta|^{1+2s}} - \frac{1}{|\xi_1-\eta|^{1+2s}} \hspace{1mm}d\eta\right| = \frac{1}{2s}\left|\frac{1}{(h+d)^{2s}} - \frac{1}{d^{2s}}\right| \\
        & = \frac{1}{2s}\left|\frac{d^{2s} - (h+d)^{2s}}{(d+h)^{2s}d^{2s}}\right| \leq \frac{hd^{2s-1}}{(d+h)^{2s}d^{2s}} \leq \frac{hd^{2s-1}}{d^{4s}} = \frac{h}{d^{1+2s}},\nonumber\\
        \label{46}
        I_{ii}&:= \left|\int_{-\infty}^{\xi_1 - d} \frac{1}{|\xi_1-\eta|^{1+2s}}\hspace{1mm}d\eta\right| = \frac{1}{2s}\frac{1}{d^{2s}}.
    \end{align}
    Substituting \eqref{45} y \eqref{46} in \eqref{44} we have
    \begin{equation}
        \label{47}
        I \leq \norm{U}_{L^\infty(\mathbb{R})} \frac{h}{d^{1+2s}} + \frac{1}{2s}[U]_{C^{0,1}(\mathbb{R})}\frac{h}{d^{2s}}.
    \end{equation}

    For the integral $II$, using the triangular inequality and the fact that  $U \in C^{0,1}(\mathbb{R})$ we obtain
    \begin{align}
    \label{48}
        II &\leq \int_{\xi_1-d}^{\xi_1 +d}\frac{|U(\xi_1) - U(\eta)|}{|\xi_1 - \eta|^{1+2s}} \hspace{1mm}d\eta + \int_{\xi_1-d}^{\xi_1 +d}\frac{| U(\eta)|}{|\xi_0 - \eta|^{1+2s}} \hspace{1mm}d\eta\\
        & \leq [U]_{C^{0,1}(\mathbb{R})}\left(\int_{\xi_1-d}^{\xi_1 +d}\frac{1}{|\xi_1 - \eta|^{2s}}\hspace{1mm}d\eta + \int_{\xi_1-d}^{\xi_1 +d}\frac{1}{|\xi_0 - \eta|^{2s}}\hspace{1mm}d\eta\right)\nonumber\\
        &= [U]_{C^{0,1}(\mathbb{R})}\left(II_i + II_{ii}\right),\nonumber
    \end{align}
    with
    \begin{align*}
        &II_i := \int_{\xi_1-d}^{\xi_1 +d}\frac{1}{|\xi_1 - \eta|^{2s}}\hspace{1mm}d\eta,\\
        &II_{ii} := \int_{\xi_1-d}^{\xi_1 +d}\frac{1}{|\xi_0 - \eta|^{2s}}\hspace{1mm}d\eta.
    \end{align*}
    Studying each integral we have
    \begin{align}
    \label{49}
        II_i &= 2\int_{\xi_1}^{\xi_1 +d} \frac{1}{|\xi_1 - \eta|^{2s}}\hspace{1mm}d\eta = \frac{2}{1-2s}d^{1-2s},\\
        \label{50}
        II_{ii} &= \int_{\xi_0}^{\xi_1 + d} \frac{1}{(\eta-\xi_0)^{2s}}\hspace{1mm}d\eta + \int_{\xi_0}^{\xi_0 + d + h} \frac{1}{(\eta-\xi_0)^{2s}}\hspace{1mm}d\eta = \frac{1}{1-2s}\left((d-k)^{1-2s} + (d+h)^{1-2s}\right).
    \end{align}
    Substituting \eqref{49} and \eqref{50} in \eqref{48} we obtain
    \begin{equation}
        \label{51}
        II \leq \frac{6}{1-2s}[U]_{C^{0,1}(\mathbb{R})}d^{1-2s}.
    \end{equation}

    Finally we study the integral III. Since $U\in C^{0,1}(\mathbb{R})$ we have
    \begin{align}
        \label{52}
        III &\leq \int_{\xi_1 + d}^{\xi_0 + d} \frac{|U(\xi_1)|}{|\xi_1-\eta|^{1+2s}}\hspace{1mm}d\eta +  \int_{\xi_0 + d}^{\infty} \frac{|U(\xi_1)|}{|\xi_1-\eta|^{1+2s}}\hspace{1mm}d\eta\\
        &\leq \int_{\xi_1 + d}^{\xi_0 + d} \frac{|U(\xi_1)|}{|\xi_1-\eta|^{1+2s}}\hspace{1mm}d\eta +  \int_{\xi_0 + d}^{\infty} \frac{|U(\xi_1)|}{|\xi_0-\eta|^{1+2s}}\hspace{1mm}d\eta\nonumber\\
        & \leq [U]_{C^{0,1}(\mathbb{R})}\int_{\xi_1 + d}^{\xi_0 + d} \frac{1}{|\xi_1-\eta|^{2s}}\hspace{1mm}d\eta +  [U]_{C^{0,1}(\mathbb{R})}\int_{\xi_0 + d}^{\infty} \frac{|\xi_1 - \xi_0|}{|\xi_1-\eta|^{1+2s}}\hspace{1mm}d\eta\nonumber\\
        &= [U]_{C^{0,1}(\mathbb{R})}III_i + [U]_{C^{0,1}(\mathbb{R})}h_1 \, III_{ii},\nonumber
     \end{align}
    with
    \begin{align*}
        III_i &:= \int_{\xi_1 + d}^{\xi_0 + d} \frac{1}{|\xi_1-\eta|^{2s}}\hspace{1mm}d\eta,\\
        III_{ii} &:= \int_{\xi_0 + d}^{\infty} \frac{1}{|\xi_1-\eta|^{1+2s}}\hspace{1mm}d\eta.
    \end{align*}
    Calculating each integral we obtain
    \begin{align}
        \label{53}
        III_i &= \int_{\xi_1 + d}^{\xi_0 + d} \frac{1}{|\xi_1-\eta|^{2s}}\hspace{1mm}d\eta = \frac{1}{1-2s}\left((d+h)^{1-2s} - d^{1-2s}\right) \leq \frac{(2^{1-2s} -1)}{1-2s}d^{1-2s},\\
        \label{54}
        III_{ii} &= \int_{\xi_0 + d}^{\infty} \frac{1}{|\xi_1-\eta|^{1+2s}}\hspace{1mm}d\eta = \frac{1}{2s}\frac{1}{(d+h)^{2s}} \leq \frac{1}{2s}\frac{1}{d^{2s}}.
    \end{align}
    Substituting \eqref{53} and \eqref{54} in \eqref{52} we obtain
    \begin{equation}
    \label{55}
        III \leq \frac{(2^{1-2s} -1)}{1-2s}[U]_{C^{0,1}(\mathbb{R})}d^{1-2s} + \frac{1}{2s}[U]_{C^{0,1}(\mathbb{R})}\frac{h_1}{d^{2s}} .
    \end{equation}
    Thanks to \eqref{47}, \eqref{51} and \eqref{55} we can bound \eqref{43} by
    \begin{align}
        \label{55b}
        \left|\xi_1H'(\xi_1) + 2s\int_{-\infty}^{\xi_0} \frac{U(\eta)}{|\xi_0 - \eta|^{1+2s}}\hspace{1mm} d\eta\right| &\leq \frac{20s}{1-2s} [U]_{C^{0,1}(\mathbb{R})} d^{1-2s} \\
        & + 2s\norm{U}_{L^\infty(\mathbb{R})}\frac{h_1}{d^{1+2s}}\ +2[U]_{C^{0,1}(\mathbb{R})}\frac{h_1}{d^{2s}}.\nonumber
    \end{align}
    If we take in \eqref{55b} $d = h_1^\gamma$ with $\gamma \in (0,1)$ we have
    \begin{align}
    \label{57}
        \left|\xi_1 H'(\xi_1) + 2s\int_{-\infty}^{\xi_0} \frac{U(\eta)}{|\xi_0 - \eta|^{1+2s}}\hspace{1mm} d\eta\right| &\leq C\left([U]_{C^{0,1}(\mathbb{R})}(h_1^{\gamma(1-2s)} + h_1^{1-2s\gamma})+ \norm{U}_{L^\infty(\mathbb{R})}h_1^{1 - (1+2s)\gamma}\right),
    \end{align}
    with $C > 0$ a constant that only depends on $s$. Arguing in the same way that in the proof of the Proposition \ref{prop 1.8} we have that the optimal value of $\gamma$ is $\gamma = \frac{1}{2}$. Thus, taking $d = h_1^{\frac{1}{2}}$ in the integrals that we solve above we obtain from \eqref{57}
    \begin{align*}
    \left|\xi_1 H'(\xi_1) + 2s\int_{-\infty}^{\xi_0} \frac{U(\eta)}{|\xi_0 - \eta|^{1+2s}}\hspace{1mm}d\eta\right| &\leq Ch_1^{\frac{1-2s}{2}}\left([U]_{C^{0,1}(\mathbb{R})} + \norm{U}_{L^\infty(\mathbb{R})}\right)\\
    &= \Tilde{C}|\xi_0 - \xi_1|^{\frac{1-2s}{2}},
    \end{align*}
    with $\Tilde{C} > 0$ a constant that not depends on $x_1$. Therefore we conclude the proof.
\end{proof}

%%%%%%%%%%%%%%%%%%%%%%%%%%%%%%%%%%%%%%%%%%%%%%%%%%%%%%%%%%%%%%%%%%%%%%%%%%%%%%%%%%%

\subsection{Regularity to the right of $\xi_0$}
Next,  we study the regularity of$H$ on the ice region to the right of $\xi_0$. The main result that we want to prove  is that $H \in C^{1, 1-2s}[\xi_0, \infty)$.

Before proving it,  we start by showing through another method that  the
right derivative of $H$ at the point $\xi_0$ satisfies the formula \eqref{41}.

\begin{prop}
\label{prop 1.1}
    Let $s \in (0,\frac{1}{2})$. Then
    \[
    H'(\xi_0^+) = \frac{2s}{\xi_0}\int_{-\infty}^{\xi_0} \frac{U(\eta)}{|\xi_0 - \eta|^{1+2s}} \hspace{1mm}d\eta.
    \]
\end{prop}
\begin{proof}
     By Proposition \ref{propo 5.7} we know that the function $H(\xi)$ satisfies
    \begin{equation}
    \label{1}
         -\frac{1}{2s}\xi H'(\xi) + (-\Delta)^s U(\xi) = 0 \qquad\textup{for all} \qquad \xi \in \mathbb{R}.
    \end{equation}
    Let $\{\xi_n\}_{n\in \mathbb{N}}$ a decreasing sequence such that $\xi_n \rightarrow
 \xi_0$. For $\eta \in (-\infty, \xi_0]$ we define  the sequence of functions
    \[
    f_n(\eta):= \frac{U(\eta)}{|\xi_n - \eta|^{1+2s}}\qquad \textup{for} \qquad n \in \mathbb{N}.
    \]
    Note that $\forall n \in \mathbb{N}$   and $\eta\in (-\infty, \xi_0]$, then $0 \leq f_n \leq f_{n+1}$. Thanks to \eqref{1} and the Monotone convergence Theorem,  we have
    \begin{align*}
       \lim_{n \to \infty} |\xi_n H'(\xi_n)| &= 2s\lim_{n \to \infty} |(-\Delta)^s U(\xi_n)| = 2s\lim_{n \to \infty} \int_{-\infty}^{\xi_0} f_n(\eta) \hspace{1mm}d\eta \\
       &= 2s\int_{-\infty}^{\xi_0} \lim_{n \to \infty} f_n(\eta) \hspace{1mm}d\eta
       = \int_{-\infty}^{\xi_0} \frac{U(\eta)}{|\xi_0 - \eta|^{1+2s}} \hspace{1mm}d\eta \nonumber.
    \end{align*}
    Thus,
    \begin{equation*}
    H'(\xi_0^+) := \frac{2s}{\xi_0}\int_{-\infty}^{\xi_0} \frac{-U(\eta)}{|\xi_0 - \eta|^{1+2s}} \hspace{1mm}d\eta.
    \end{equation*}
\end{proof}

 \noindent $\bullet$  {\bf Lack of regularity.} Finally, we obtain a limitation to the regularity of the selfsimilar solution $H$ to the problem \eqref{stefan}-\eqref{IC}  at $\xi_0$. The key idea is to prove that  second derivative of $H$ on the right of $\xi_0$ blows up. This completes part of c) of Theorem  \ref{principal}.

%%%%%%%%%%%%%%%%%%%%%%%%%%%%%%%%%%%%%%%%%%%%%%%%%%%%%%%%%%%
   \begin{prop}
    \label{lemma 1}
    Let $s \in (0,1/2)$. Then $H(x) \notin C^2(\mathbb{R})$ because $H''(\xi_0^+)=\infty$. More precisely, $H''(\xi)>0 $ in the interval $(\xi_0,\infty)$ and there is a constant $C_1>0$ such that
   \begin{equation}\label{est.H''}
    H''(\xi)\,|\xi-\xi_0|^{2s}\ge C_1 \quad \mbox{ for all $\xi>\xi_0$ with $|\xi-\xi_0|$ small.}
      \end{equation}
\end{prop}

\begin{proof}
    For proving that $H \notin C^2(\mathbb{R})$ it is enough to show that $H''(\xi_0^+) = \infty$. By Theorem \ref{principal} we know that
    \begin{equation}
      \label{derivative}
    -H'(\xi) = \frac{1}{\xi}\int_{-\infty}^{\xi_0}\frac{U(\eta)}{|\xi-\eta|^{1+2s}}\hspace{1mm}d\eta
    \end{equation}
    for all $\xi \in \mathbb{R}$. Let $\xi > \xi_0$, taking the derivative in \eqref{derivative} and using the Leibnitz Rule for integration we can pass the derivative inside the integral and we obtain
    \begin{equation}
    \label{7}
    H''(\xi) = \frac{2s}{\xi^2}\int_{-\infty}^{\xi_0} \frac{U(\eta)}{|\xi-\eta|^{1+2s}}\hspace{1mm}d\eta + \frac{2s(1+2s)}{\xi}\int_{-\infty}^{\xi_0} \frac{U(\eta)}{|\xi-\eta|^{2+2s}}\hspace{1mm}d\eta= I(\xi)+II(\xi).
    \end{equation}
    Now we want to see what happens when take the limit when $\xi \rightarrow \xi_0$ in \eqref{7}. Arguing as before and using the fact that $U$ is linear to the left of $\xi_0$, we easily see that the first integral $I(\xi)$ remains bounded as $\xi \rightarrow \xi_0$. Therefore, we only have to examine the divergence of the last integral, $II(\xi)$.  It is interesting to point out that direct inspection of the integrands shows that $I(\xi)$ is lower order than $II(\xi)$ when $\xi\sim \xi_0$.

    Proceeding with the proof, the integrand in $II(\xi)$ is positive and we can bound it from below the integral as follows, using $\xi_1<\xi_0<$ such that $\xi_0-\xi_1=\xi-\xi_0$:
     \begin{equation}
    II(\xi) \ge  \frac{2s(1+2s)}{\xi_0+h}    \int_{-\infty}^{\xi_1} \frac{U(\eta)}{|\xi-\eta|^{2+2s}}\hspace{1mm}d\eta.
    \end{equation}

    We know by Theorem  \ref{principal} that $|H'(\xi_0^-)| = |U'(\xi_0^-)| =C_0> 0$,  then by the continuity of $U '$ there exists  $a > 0$ and  such that $|U'(\xi)| \geq C_0/2 >0$ for all $\xi \in [\xi_0 - a, \xi_0]$. Then we have the linear lower bound to the right of $\xi_0$:
    \begin{equation}
        \label{desi}
       U(\eta) = U(\eta) - U(\xi_0) \geq \frac12 C_0 \,|\xi_0-\eta |,
    \end{equation}
    for all $\eta \in [\xi_0 - a, \xi_0]$.
Letting now $\xi-\xi_0=h>0$ be very small we have \ $\xi_1=\xi_0-h$ and   $|\xi - \eta |\le 2\,|\xi_0 - \eta|$  \, for $\eta\le \xi_1$. It follows that \nc
  \begin{equation}\label{est.H2}
    II(\xi) \ge  \frac{s(1+2s)C_0}{3\xi_0}    \int_{\xi_0-a}^{\xi_0-h} \frac{1}{|\xi-\eta|^{1+2s}}\hspace{1mm}d\eta.
    \end{equation}
    When we fix $ a$ and  let $h=\xi-\xi_0$ be close to zero we get
 \begin{equation}
 H''(\xi)\ge II(\xi)\ge C_1\, |\xi-\xi_0|^{-2s}
 \end{equation}
 with a constant that we can easily estimate. \end{proof}

\noindent {\bf Remark.} The proof of estimate \eqref{est.H2} can be repeated later for $s\ge 1/2$ on the condition of obtaining a linear lower bound for $U$ to the right of $\xi_0$. See sections \ref{eleven} and \ref{section 13}.

\medskip

\noindent Finally, we prove the optimality assertion mentioned in Theorem \ref{principal}.

\begin{cor}
    Let $s \in (0,1/2)$. Then, $H \in C^{1,\alpha}([\xi_0,\infty))$ with optimal exponent $\alpha = 1-2s$.
\end{cor}

\begin{proof}
    Since $H\in C^\infty(\xi_0, \infty)$, it is enough to show that the exponent $\alpha = 1-2s$ it is optimal near the free boundary. By Proposition \ref{prop ine 2} we know that
    \begin{equation}
     \label{above}
    |H'(\xi) - H'(\xi_0)| \leq C_s |\xi - \xi_0|^{1-2s} \quad \forall \xi > \xi_0.
    \end{equation}
    On the other hand, by Proposition \ref{lemma 1} we know that
    \begin{equation}
        \label{inferior}
        H''(\xi)|\xi - \xi_0|^{2s} \geq C_1 \quad \forall \xi > \xi_0.
    \end{equation}
    Integrating in \eqref{inferior} over $(\xi_0, \xi)$ we obtain
    \begin{equation}
    \label{below}
     |H'(\xi) - H'(\xi_0)| \geq C_2 |\xi - \xi_0|^{1-2s}
    \end{equation}
 for all $\xi>\xi_0$ with $|\xi-\xi_0|$ small.  The optimality of the exponent is a consequence of the inequalities \eqref{above} and \eqref{below}.
\end{proof}

\newpage

%%%%%%%%%%%%%%%%%%%%%%%%%%%%%%%%%%%%%%%%%%%%%%%%%%%%%%%%%%%%%%%%%%%%%%%%%%%%%%
\section{Subcritical tail behaviour at spatial infinity}
\label{sec.nine}

In this section we introduce a different topic. We want to study the tail behaviour of the unique solution of Problem \eqref{stefan}-\eqref{IC} when $s \in (0, 1/2)$. We want to estimate $H$ and $H'$ as  functions of $\xi$ when  $\xi$ is very large, either positive or negative.

\noindent {\bf Notation.}  In the expressions below $O(\cdot) $ and $o(\cdot)$ are the usual  Landau notations, here taken in the limits when either $\xi\to\infty$ or $\xi\to -\infty$. For  the relative limit behaviour of two nonnegative quantities $a(x)$ and $b(x)$, we use the notations $a(x)\precsim b(x)$ when there is a positive constant $C$ such that $a(x)\le C b(x)$, while the stricter notation  $a(x)\asymp b(x)$ means that there exist positive $C_1<C_2$ such that  $C_1 b(x) \le a(x)\le C_2 b(x)$.
If the limit ratio is a given constant $C>0$ we will let it know with the $  \approx$ sign. $O(\cdot) $ and $o(\cdot)$ are usual Landau notations.

\subsection{Tail behaviour at plus infinity}   The behaviour of $H(\xi)$ and $H'(\xi)$ has been stated  in first asymptotic approximation in \cite[Theorem 3.1(g)]{DEV1} and proved in Section 3.8, by estimating $H'(\xi)$ for $\xi>\xi_0$ by means of the selfsimilar
equation \eqref{eq.ssh}. The result is
        \begin{equation}
H(\xi) \asymp \frac{1}{|\xi|^{2s}}, \quad \mbox{and} \quad H'(\xi) \asymp \frac{1}{|\xi|^{1+2s}}
        \end{equation}
as $\xi\to\infty$. It works for all $0<s<1$. We will come back to these formulas later to get precise constants in the first order approximation and an estimate of the lower-order rest.

\subsection{Subcritical tail behaviour at minus infinity} Since the study at plus infinity is easier we turn our attention to the study as $\xi\to-\infty$ looking for similar asymptotics\nc. The main result of the section reads as follows:
\begin{thm}
\label{teo 2.1}
    Let $s \in (0, 1/2)$ and $H \in L^{\infty}(\mathbb{R})$ the unique selfsimilar solution of \eqref{stefan}-\eqref{IC}. Then, for $\xi \ll \xi_0$
        \begin{equation}
        \label{sub1}
        L+P_1 - H(\xi)= P_1-U(\xi) \approx \frac{P_1}{2s}\,|\xi|^{-2s},
        \end{equation}
        and
        \begin{equation}
        \label{sub2}
        |H'(\xi)| \approx P_1\,|\xi|^{-(1+2s)}.
        \end{equation}
\end{thm}

\medskip

 We divide the proof of the Theorem in some Lemmas. We start with \eqref{sub2}.

 \begin{lem}\label{lem9.1}
Under the assumptions of Theorem \ref{teo 2.1} we have $H'(\xi)=o(|\xi|^{-1})$, i.e.,
  \begin{equation}
    \label{8.2}
\lim_{\xi \to - \infty} \xi H'(\xi) = 0.
\end{equation}
\end{lem}

\begin{proof} The expression means that
\begin{equation}
    \label{8.3}
    |H'(\xi)| \leq \frac{\varepsilon(\xi)}{|\xi|}  \qquad \textup{for all} \qquad \xi \ll \xi_0,
\end{equation}
where $\varepsilon(\xi) > 0$ is a constant that depends on $\xi$ and satisfies $\varepsilon(\xi)\to 0 $ when $\xi \to - \infty$. The  bound obtained here is not enough for our purposes  but it represents a preliminary step.

The idea  for proving \eqref{8.2} is to split the integral in three pieces. Let $\xi \ll 0$, thanks to Theorem \ref{teo 4}-$a)$ we have
\begin{equation}
    \label{8.5}
   0\le  \xi H'(\xi) = 2s  \,  (-\Delta)^s U(\xi) = \int_{\mathbb{R}} \frac{U(\xi) - U(\eta)}{|\xi-\eta|^{1+2s}}\hspace{1mm}d\eta = 2s \, (I + II + III),
\end{equation}
where
\begin{align*}
    & I := \int_{-\infty}^{\xi -d}\frac{U(\xi) - U(\eta)}{|\xi - \eta|^{1+2s}} \hspace{1mm}d\eta,\\
        &  II := \int_{\xi-d}^{\xi +d}\frac{U(\xi) - U(\eta)}{|\xi - \eta|^{1+2s}} \hspace{1mm}d\eta,\\
        & III := \int_{\xi + d}^{+ \infty}\frac{U(\xi) - U(\eta)}{|\xi - \eta|^{1+2s}} \hspace{1mm}d\eta,
\end{align*}
with $d$ a value that will be determined later. By Theorem \ref{teo 4}-$d)$ the function $H$ is nonincreasing, hence we have $U(\eta) \geq U(\xi)$ $\forall \eta \in (-\infty, \xi]$, hence $I$ is negative. Looking at the sign of  the left-hand side $\xi H'(\xi)\ge 0$ we can drop $I$ and pass \eqref{8.5} into  the form
\begin{align}
    \label{8.25}
    |\xi H'(\xi)| & \leq 2s\, ( |II| + |III|).\nonumber
\end{align}

\noindent    $\bullet$ Let us examine integral $II$\nc. Since  $U \in C^{0,1}(\mathbb{R})$, then we have
\begin{equation}
    \label{8.8}
    |II| \leq \int_{\xi -d}^{\xi +d} \frac{|U(\xi) - U(\eta)|}{|\xi - \eta|^{1+2s}} \hspace{1mm}d\eta \leq [U]_{C^{0,1}(\mathbb{R})}\int_{\xi -d}^{\xi +d} \frac{1}{|\xi - \eta|^{2s}}\hspace{1mm}d\eta \leq \frac{2[U]_{C^{0,1}(\mathbb{R})}}{1-2s}d^{1-2s}.
\end{equation}

\noindent    $\bullet$ Let us now look at  integral $III$. Again, it will be convenient to split
in two pieces. First observe that  by Theorem \ref{teo 4}-$b)$, $\lim_{\xi \to -\infty} H(\xi) = P_1$,, so that given $\delta>0$ small there exists $ {\xi}^*$ negative  with $ |\xi|^*$ large enough such that
$$
P_1-U(\eta)< \delta \quad \mbox{for } \eta<\xi^*(\delta).
$$
Given $\delta>0$ we now take $\xi$ much to the left of $\xi^*(\delta)$ and then put $z=\lambda\, \xi$ with $\lambda \in (0,1/2) $ and take a $d>0$ such that $\xi+d<z$. We split the integral III as follows:
\begin{equation}
    \label{8.9}
    III = \int_{\xi + d}^{z}\frac{U(\xi) - U(\eta)}{|\xi - \eta|^{1+2s}} \hspace{1mm}d\eta + \int_{z}^{+ \infty}\frac{U(\xi) - U(\eta)}{|\xi - \eta|^{1+2s}} \hspace{1mm}d\eta= III_i + III_{ii}
\end{equation}
Using the above information we deduce
\begin{align}
\label{8.13}
    |III_i |&\leq \int_{\xi + d}^{z}\frac{|U(\xi) - U(\eta)|}{|\xi - \eta|^{1+2s}} \hspace{1mm}d\eta \leq \delta \int_{\xi + d}^{z}\frac{1}{|\xi - \eta|^{1+2s}} \leq \frac{\delta}{2s}\left(\frac{1}{d^{2s}} + \frac{1}{|\xi-z|^{2s}}\right)\\
    & \leq \frac{1}{2s}\left(\frac{\delta}{d^{2s}} + \frac{\delta}{(1-\lambda)^{2s}|\xi-\lambda \xi|^{2s}}\right).
\end{align}
Besides, we can bound $III_{ii}$ by
\begin{equation}
    \label{8.14}
    |III_{ii} |\leq  \norm{U}_{L^{\infty}(\mathbb{R})}\int_{z}^{+ \infty} \frac{1}{|\xi - \eta|^{1+2s}}\hspace{1mm} d\eta = \frac{P_1}{2s}\frac{1}{|\xi - z|^{2s}} = \frac{P_1}{2s(1-\lambda)^{2s}}\frac{1}{|\xi|^{2s}}.
\end{equation}
Substituting \eqref{8.13} and \eqref{8.14} into \eqref{8.9}, we obtain
\begin{equation}
    \label{8.15}
    |III | \leq \frac{1}{2s}\left( \frac{\delta}{d^{2s}} + \frac{\delta}{(1-\lambda)^{2s}|\xi|^{2s}}\right)+ \frac{P_1}{2s(1-\lambda)^{2s}}\frac{1}{|\xi|^{2s}}.
\end{equation}

\noindent    $\bullet$ Thanks to \eqref{8.8} and \eqref{8.15}  we can get a bound \eqref{8.5}. If moreover choose $d=\delta$ we have
\begin{equation}
    \label{8.16}
    |\xi H'(\xi)| \leq \left( \frac{ 4[U]_{C^{0,1}(\mathbb{R}) }}{s(1-2s)} + 1 \right)  d^{1-2s} +
    \frac{P_1}{(1-\lambda)^{2s}}\frac{1}{|\xi|^{2s}}.
%    & \leq C\left( d^{1-2s} + \frac{\delta}{d^{2s}}\right) + \frac{P_1+\delta}{|\xi|^{2s}}, \nonumber
\end{equation}
We can  now show that both terms on the right-hand side can be made to tend to 0 as $\xi\to-\infty $. The last one is obvious as long as $\lambda$ is not close to 0. Hence we get for all  $\xi\ll 1$
$$
\lim_{\xi\to -\infty}     |\xi H'(\xi)| \leq  C \delta^{1-2s}.
$$
But $\delta $ can be taken as small as desired, so that in the end
$$
\lim_{\xi\to-\infty}     |\xi H'(\xi)| =0
$$
as we wanted to prove.  \end{proof}

This estimate can now be used to refine the calculation.

 \begin{lem}
Under the same assumptions we get
     \begin{equation}
        \label{sub2c}
 |H'(\xi)| =O( 1/|\xi|^{1+2s}) \qquad \textup{as} \quad \xi\to -\infty.
  \end{equation}
 \end{lem}
\noindent \begin{proof}   The idea is to use the previous lemma to  improve the speed of this bound as $\xi \to\infty$. If we look at the proof done for Lemma \ref{lem9.1},  we need only change \nc the bound  obtained for integral $II$. Thus, let $\xi \ll \xi_0$, thanks to the mean value Theorem and \eqref{8.3} we have
\begin{equation}
    \label{8.17}
   | II | \leq \int_{\xi -d}^{\xi +d} \frac{|U(\xi) - U(\eta)|}{|\xi - \eta|^{1+2s}} \hspace{1mm}d\eta \leq \int_{\xi -d}^{\xi +d} \frac{|U'(x_o)| }{|\xi - \eta|^{2s}}\hspace{1mm}d\eta \leq \frac{2\varepsilon(\xi) d^{1-2s}}{(1-2s)|\xi|},
\end{equation}
where $x_o \in (\xi-d, \xi+d)$. Arguing in the same way as in the Step 1 in the integrals $III_i$ and $III_{ii}$ and applying  \eqref{8.17} we can bound \eqref{8.5} by
\begin{equation}
\label{8.19}
|\xi H'(\xi)|\leq C\left(\frac{d^{1-2s}}{|\xi|}+ \frac{\delta}{d^{2s}} + \frac{\delta+P_1}{|\xi|^{2s}}\right),
\end{equation}
where $C >0$ a constant that depends on $s$, $\lambda$, $[U]_{C^{0,1}(\mathbb{R})}$, $\norm{U}_{L^{\infty}(\mathbb{R})}$ and $\varepsilon(\xi)$. Taking this time $d = \delta|\xi|$ in \eqref{8.19} we obtain
\[
|\xi H'(\xi)| \leq \frac{C}{|\xi|^{2s}},
\]
where $C$  depends on $s$, $\delta$, $\lambda$, $[U]_{C^{0,1}(\mathbb{R})}$, $\norm{U}_{L^{\infty}(\mathbb{R})}$ and $\varepsilon(\xi)$. Then for all $\xi \ll \xi_0$ we conclude that
\[
|H'(\xi)| \leq \frac{C}{|\xi|^{1+2s}}.
\]
%%%%%%%%%%%%%%%%%%%%%%%%%%%%%%%%%%%%%%%%%%%%%%%%%%%%
\noindent{\bf Remark.}
The proof of this upper bound for $H'(\xi)$  implies after integration that the tail estimate \eqref{sub1} for $H$  holds  with $\precsim$ sign estimate.

\medskip

\noindent\underline{\textit{Step 3}}: In order to get a lower bound for $|H'(\xi)|$ we go back to formula \eqref{8.5} and write
$$
0\le  \xi H'(\xi) = 2s \, (I + II + III),
$$
Then we observe that the leading term as $\xi\to -\infty$ is $III$ which is positive. More precisely, is can be estimated  in first approximation on both sides by
$$
C_1\frac{P_1}{|\xi|^{2s}} \le  III\le C_2\frac{P_1}{|\xi|^{2s}}
$$
for all large $|\xi|$, $\xi<0$ as is  easily seen. The term $II$ has a lower value, whatever its sign, $II=O({|\xi|^{-1-2s}})$. Besides, the first integral is negative but is estimated as
\begin{equation}
     |I| = \int_{-\infty}^{\xi -d}\frac{U(\eta) - U(\xi)}{|\xi - \eta|^{1+2s}}  \hspace{1mm}d\eta \le
     \int_{-\infty}^{\xi -d}\frac{P_1 - U(\xi)}{|\xi - \eta|^{1+2s}}  \hspace{1mm}d\eta \le
      \frac{C}{|\xi-d |^{2s}} \int_{-\infty}^{\xi -d} \frac{d\eta}{|\xi - \eta|^{1+2s}}
\end{equation}
so that
$$
|I| \le \frac{C}{|\xi|^{2s}}\int_{-\infty}^{\xi -d}\frac{  d\eta}{|\xi - \eta|^{1+2s}}= \frac{\delta}{2s|\xi-d|^{2s}} \frac{1}{d^{2s}}.
$$
Since we take $d = \delta|\xi|$  the estimate for $I$ has a lower size than $III$. This proves the lower part of the estimate \eqref{sub2} for $H'(\xi)$\nc.
\end{proof}
%%%%%%%%%%%%%%%%%%% inserted 17 nov 2025

\medskip
%%%%%%%%%%%%%%%%%%%%%%%%%%%%%%%%%%%%%%%%%%%%%%%%%%%%
\noindent{\bf Remark.}
The proof of the upper bound for $H'(\xi)$ \eqref{sub2} implies after integration that \eqref{sub1} holds.
Thanks to a refined comparison principle we can prove a stricter formula \eqref{sub1} with exact size estimate, both power and constant.

\begin{lem}
\label{lemma 9.2}
Under the assumptions of Theorem \ref{teo 2.1} we have that
\[
L+P_1 - H(\xi) = P_1-U(\xi) \approx P_1/|\xi|^{2s} \qquad \textup{for} \quad \xi \ll \xi_0.
\]
\end{lem}
\begin{proof}
    The idea of the proof is to use the comparison principle for the solutions of the Fractional Stefan Problem to bound $L+P_1 - H(\xi)$. In order to give a rigorous proof, we split the proof in two steps.

    \noindent\underline{\textit{Step 1}}: First we prove a \textit{lower bound} of \eqref{sub1}. Let $\overline{h}(x,t) = \overline{H}(xt^{-\frac{1}{2s}})$ the unique selfsimilar solution of the problem
\begin{equation}
\label{pro1}
\left\{
\begin{aligned}
\partial_t \overline{h} + (-\Delta)^s \Phi (\overline{h}) &= 0 \hspace{7mm}\textup{in} \hspace{7mm} \mathbb{R} \times (0,T),\\
\overline{h}( \cdot,0) &= \overline{h}_0 \hspace{5.3mm} \textup{on} \hspace{6.2mm} \mathbb{R}.
\end{aligned}
\right.
\end{equation}
with
\begin{equation*}
\overline{h}_0(x):=\left\{
\begin{aligned}
 &L + P_1 \hspace{10mm}\textup{if} \hspace{7mm} x \leq 0, \\
&L  \hspace{19.7mm}\textup{if} \hspace{7mm}x > 0.
\end{aligned}
\right.
\end{equation*}
By Theorem \ref{teo 4} we know that $L \leq \overline{H}(\xi) \leq L+P_1$, then the problem \eqref{pro1} doesn't have a free boundary. In particular we have that $\Phi(\overline{h}) = \overline{h} - L$. Since $(-\Delta)^s \Phi(\overline{h}) = (-\Delta)^s \overline{h}$ we can rewrite \eqref{pro1} like
\begin{equation}
\label{pro2}
\left\{
\begin{aligned}
\partial_t \overline{h} + (-\Delta)^s  \overline{h} &= 0 \hspace{7mm}\textup{in} \hspace{7mm} \mathbb{R} \times (0,T),\\
\overline{h}( \cdot,0) &= \overline{h}_0 \hspace{5.3mm} \textup{on} \hspace{6.2mm} \mathbb{R}.
\end{aligned}
\right.
\end{equation}
Notice that the problem \eqref{pro2} is a fractional heat equation with a step initial data.

On the other hand, the solution of the problem
\begin{equation*}
\left\{
\begin{aligned}
\partial_t h_1 + (-\Delta)^s  h_1 &= 0 \hspace{15.8mm}\textup{in} \hspace{7mm} \mathbb{R} \times (0,T),\\
h_1( \cdot,0) &= P \delta_0(x) \hspace{5.3mm} \textup{on} \hspace{6.2mm} \mathbb{R},
\end{aligned}
\right.
\end{equation*}
is given by
\begin{equation*}
    h_1(x,t) = P\mathcal{B}(x,t),
\end{equation*}
where $\mathcal{B}(x,t)$ is the fundamental solution of the fractional heat equation. Since the distributional derivative of a step function is a Dirac delta, then the solution of the problem \eqref{pro1} is
\[
\overline{h}(x,t) = L+ P_1 - \int_{-\infty}^{x}h_1(y,t) \hspace{1mm}dy = L+P_1 - P_1\int_{-\infty}^{x} \mathcal{B}(y,t)\hspace{1mm}dy.
\]
Thus, since $h_0 \leq \overline{h}_0$ by the comparison principle  then $h \leq \overline{h}$, where $h$ is the unique selfsimilar solution of \eqref{stefan}-\eqref{IC}. In particular taking $t=1$ we deduce
\begin{equation}
    \label{upper bound}
    h(x,1) = H(x) \leq  L+P_1 - P_1\int_{-\infty}^{x} \mathcal{B}(y,1)\hspace{1mm}dy.
\end{equation}

\noindent\underline{\textit{Step 2}}: Now we compute an \textit{upper bound} of \eqref{sub1}. Let $\underline{h}$ the unique solution of the two-phase Fractional Stefan Problem
\begin{equation}
\label{pro3}
\left\{
\begin{aligned}
\partial_t \underline{h} + (-\Delta)^s \Phi_2 (\underline{h}) &= 0 \hspace{7mm}\textup{in} \hspace{7mm} \mathbb{R} \times (0,T),\\
\underline{h}( \cdot,0) &= \underline{h}_0 \hspace{5.3mm} \textup{on} \hspace{6.2mm} \mathbb{R}.
\end{aligned}
\right.
\end{equation}
with
\begin{equation}
\label{ini2}
\underline{h}_0(x):=\left\{
\begin{aligned}
 & L + P_1 \hspace{10mm}\textup{if} \hspace{7mm} x \leq 0, \\
&-P_1  \hspace{13mm}\textup{if} \hspace{7mm}x > 0,
\end{aligned}
\right.
\end{equation}
and $\Phi_2(\underline{h}) := \textup{max}\{\underline{h}- L\} + \textup{min}\{\underline{h},0\}$, with $L > 0$ the latent heat. By \cite[Theorem 5.4]{DEV1} we know that the temperature $\underline{u} = \Phi_2(\underline{h})$ satisfies the problem
\begin{equation}
\label{pro 4}
\left\{
\begin{aligned}
\partial_t \underline{u} + (-\Delta)^s  \underline{u} &= 0 \hspace{7mm}\textup{in} \hspace{7mm} \mathbb{R} \times (0,T),\\
\underline{u}( \cdot,0) &= \underline{u}_0 \hspace{5.3mm} \textup{on} \hspace{6.2mm} \mathbb{R}.
\end{aligned}
\right.
\end{equation}
with
\begin{equation}
\label{ini 3}
\underline{u}_0(x):=\left\{
\begin{aligned}
 & P_1 \hspace{10mm}\textup{if} \hspace{7mm} x \leq 0, \\
&-P_1  \hspace{5.3mm}\textup{if} \hspace{7mm}x > 0.
\end{aligned}
\right.
\end{equation}
In particular, arguing like in the \textit{Step 1}, the solution of the problem \eqref{pro 4}-\eqref{ini 3} is given by
\[
\underline{u}(x,t) = P_1 - 2P_1\int_{-\infty}^{x}\mathcal{B}(y,t)\hspace{1mm}dy.
\]

On the other hand, since the problem \eqref{stefan}-\eqref{IC}, then, by the comparison principle we have that $\underline{h}\leq h$. Notice that we are seeing the problem \eqref{stefan}-\eqref{IC} like a two-phase problem. Since for all $t \in (0,T)$ $\underline{u}(x,t) = \underline{h}(x,t) - L$ for $x < 0$ we conclude
\[
\underline{h}(x,t) = \underline{u}(x,t)  + L = L + P_1 - 2P_1\int_{-\infty}^{x}\mathcal{B}(y,t)\hspace{1mm}dy \leq h(x,t),
\]
for all $ t \in (0,T)$ and $x < 0$. Thus, taking $t = 1$ we obtain
\begin{equation}
\label{kaka}
\underline{h}(x,1) = \underline{H}(x) = L + P_1 - 2P_1\int_{-\infty}^{x}\mathcal{B}(y,1)\hspace{1mm}dy \leq H(x).
\end{equation}

\noindent\underline{\textit{Step 3}}: Finally, we obtain the bound \eqref{sub1}. By \eqref{upper bound} and \eqref{kaka} we have for $x \ll \xi_0$
\begin{equation}
\label{step3}
L + P_1 - 2P_1\int_{-\infty}^{x}\mathcal{B}(y,1)\hspace{1mm}dy \leq H(x) \leq L + P_1 - P_1\int_{-\infty}^{x}\mathcal{B}(y,1)\hspace{1mm}dy.
\end{equation}
Since $\frac{C}{|x|^{1+2s}} \leq \mathcal{B}(x,1) \leq \frac{\Tilde{C}}{|x|^{1+2s}} $ we obtain from \eqref{step3} that
\[
L+P_1 -\frac{2P_1\tilde{C}}{|x|^{2s}}\leq H(x) \leq L+P_1 -\frac{P_1C}{|x|^{2s}}.
\]
Thus,
\[
(L+P_1) - H(\xi) \approx P_1/|\xi|^{2s}.
\]
\end{proof}

%%%%%%%%%%%%%%%%%%%%%%%%%%%%%%%%%%%%%%%%%%%%%%%%%%%%%%%%%%%%%%%%%%%%%%%%%%%%%%%%%%%%%%%%%%%%
\subsection{\bf  Decay of the enthalpy at plus infinity and comparison}

\begin{thm}
Let $s \in (0, 1/2)$ and $\xi > \xi_0$. (i)  When  $\xi \to + \infty$ we have
 \begin{equation}
     \label{4.4b}
     H'(\xi) = -\frac{P_1}{|\xi|^{1+2s}} + \textup{o}\hspace{0.5mm}(|\xi|^{-(1+2s)}).
 \end{equation}

  \noindent (ii) Integrating we obtain in the same limit
\[
H(\xi) = \frac{P_1}{2s} \frac{1}{|\xi|^{2s}} + \textup{o}\hspace{0.5mm}(|\xi|^{-2s}).
\]
\end{thm}

This calculation was done or $H$ in \cite[Section 3.8]{DEV1} for $s>1/2$ and is based on the formula for $H'(\xi)$ if $\xi>\xi_0$
$$
 -\frac{\xi}{2s} H'(\xi) = \int_{-\infty}^{\xi_0} \frac{U(\eta)}{|\xi - \eta|^{1+2s}} \,d\eta.
$$
We will apply the formula here with slightly different consequences.
Since $U(\eta)\le P_1$ we immediately get
$$
 -{\xi} H'(\xi)\le {P_1} \, |\xi-\xi_0|^{-2s},
$$
and from that we get
$$
H(\xi)\le P_1 \int_{\xi}^{\infty} \frac{d\eta}{\xi(\xi-\xi_0)^{2s}}= \frac {P_1}{2s}|\xi|^{-2s}+O(|\xi|^{-2s-1})
$$
which is  the same behaviour we obtain on the other end, $\xi\to\infty$ (as a first-order error with respect to the limit).

In other to get the lower bound we realize that $U(\xi)\to P_1$ as $\xi\to-\infty$ so that for all $\varepsilon>0$ there is a point $\xi_1\ll 0$ such that
$U(\eta)\ge P_1-\varepsilon$ for $\eta\le \xi_1$. Moreover, we may use the approximation we have obtained before in this section:
 $$
 P_1-U(\eta)\approx C|\eta|^{-2s}
 $$
for all $\eta\le \xi_1<0$ to get
$$
 -\xi H'(\xi)\ge 2s \int_{-\infty}^{\xi_1} \frac{P_1-C\eta^{-2s}} {|\xi - \eta|^{1+2s}}\,d\eta = P_1 |\xi|^{-2s}+O(|\xi|^{-4s})+ O(|\xi|^{-1-2s}).
$$
Therefore,
$$
-H'(x)=P_1|\xi|^{-2s-1}+O(|\xi|^{-1-4s}), \qquad
H(x)=\frac {P_1}{2s}|\xi|^{-2s} + O(|\xi|^{-4s}).
$$
\nc

%%%%%%%%%%%%%%%%%%%%%%%%%%%%%%%%%%%%%%%%%%%%%%%%%%%%%%%%%%%%%%%%%%%%%%%%%%%%%%%%%%%%%%%%%%%%%
\subsection{Mass transfer for the subcritical regime}
\label{ten}

In this section we study the question of mass transfer when $s \in (0,1/2)$. In \cite{DEV1} the authors proved the following result  for the  supercritical case:

\begin{thm}
    If $s \in (1/2, 1)$, then
    \begin{equation}
    \label{4.2b}
    \int_{-\infty}^{0} ((L+P_1) - H(\xi))\hspace{1mm}d\xi = \int_{0}^{\infty} H(\xi) \hspace{1mm}d\xi < + \infty.
    \end{equation}
\end{thm}

Clearly by this theorem we have equal balance  of mass transfer for the supercritical case. We mean the mass of the solution that evolves from jump-like initial data propagates equally in both directions: from $-\infty$ to 0, and from 0 to $+\infty$.

In the subcritical case we cannot have an analogous theorem because
\[
(L+P_1) - H(\xi) \asymp 1/|\xi|^{2s} \quad \textup{for all}\quad \xi \ll \xi_0,
\]
so that the left-hand side of the mass equality in \eqref{4.2b} is infinite. Hence, we cannot speak about global mass transfer. Nonetheless, we can show that the solution has an asymptotic constant with different sign at $-\infty$ and $+\infty$. This allows us to tell that we have in some sense a balance of  infinite masses at infinity\nc.

%\newpage

%%%%%%%%%%%%%%%%%%%%%%%%%%%%%%%%%%%%%%%%%%%%%%%%%%%%%%%%%%%%%%%%%%%%%%%%%%%%%%%%%%%%%%%%%%%%%%%%%%%%%%%%%%%%%%
\section{Critical case. Regularity}\label{eleven}

Here we are concerned with the regularity of the solution of the problem \eqref{stefan}-\eqref{IC} when we examine the critical exponent $s = 1/2$. In Theorem \ref{principal} we have proved that the selfsimilar solutions are $C^{1,\alpha}$ in the whole subcritical exponent range $0<s<1/2$ with and $\alpha>0$ that depends on $s$. In this section we show that the  selfsimilar solution fails to be $C^1$ smooth at the critical value $s =1/2$. Indeed, $H'$ las a singularity of at least logarithmic size\nc.

\medskip

\subsection{Linear lower bound on the left side of $\xi_0$. The wedge}\label{ssec.supcritreg12}
What we want to prove first is that the temperature $u(x,t)$ approaches the free boundary point from the left in a non-degenerate way, i.e., for $t > 0$ there exists a $C(t) > 0$ such that
\begin{equation}
\label{criti criti}
u(x,t) \geq C(t)(\xi_o t- x)
\end{equation}
with $x$ near to $\xi_o t$. Equivalently, in terms of the selfsimilar profile we will prove that
\[
U(\eta) \geq C(\xi_0 - \eta) \qquad \forall  \eta\in (\xi_0/2, \xi_0).
\]
 We call this situation having a ``linear wedge from below''. Then the argument to prove the singularity of $H'(\xi)$ at $\xi_0$ goes as follows. Since  for $\xi>\xi_0$ we have
\[
-H'(\xi) = \frac{1}{\xi}\int_{-\infty}^{\xi_0}\frac{U(\eta)}{|\xi-\eta|^{2}}\hspace{1mm}d\eta,
\]
then by \eqref{criti criti} we have for every  $\xi_1=\xi_0+\varepsilon$,  $\varepsilon>0$ small:
\[
-H'(\xi_1) = \frac{1}{\xi_1}\int_{-\infty}^{\xi_0}\frac{U(\eta)}{|\xi_1-\eta|^{2}}\hspace{1mm}d\eta \geq \frac{C}{\xi_1}\int_{\xi_0/2 }^{\xi_0- \varepsilon} \frac{|\xi_0 - \eta|}{|\xi_1 - \eta|^2}\hspace{1mm}d\eta.
\]
But in the integration interval we have $|\xi_0 - \eta|\ge 1/2|\,\xi_1 - \eta|$ so  that
\[
-H'(\xi_1) \ge \frac{C}{2\xi_1}\int_{\xi_0/2 }^{\xi_0 - \varepsilon} \frac{1}{|\xi_1 - \eta|}\hspace{1mm}d\eta=
\frac{C}{2\xi_1}\left(-\log(2\varepsilon)+\log(\xi_0/2- \varepsilon)\right)
\]
For $\varepsilon=\xi_1-\xi_0>0$ much smaller than $\xi_0/2$ we get
\[
-H'(\xi_1) \ge \frac{C}{4\xi_1}\,|\log(1/2\varepsilon)|\sim \frac{C_1}{\xi_1}\,|\log(\xi_1-\xi_0)|.
\]

\nc  Thus, it is enough to prove the lower bound \eqref{criti criti} to show that in the critical case the solutions are not $C^1$. We picture this situation as having a small wedge below. Because of the self-similarity we can prove it at any particular $t > 0$. We will use a comparison theory based on the Caffarelli-Silvestre extension, cf. \cite{CaSi}.

\subsection{Caffarelli-Silvestre extensions}
Let $h(x,t)$ the solution of Problem \eqref{stefan}-\eqref{IC} with $s =1/2$ and let $u(x,t)$ be its temperature. We denote by  $W(x,y): \mathbb{R} \times \mathbb{R} \rightarrow \mathbb{R}$ the harmonic extension of  $u(x,0)$ towards $\mathbb{R}^{n+1}_+ := \{(x,y) \in \mathbb{R}^n \times \mathbb{R} \hspace{1mm}:\hspace{1mm} y > 0\}$. That is, $W(x,y)$ is the only solution to
\[
\left\{
\begin{aligned}
\Delta W(x,y) &= 0 \hspace{15mm}\textup{in} \hspace{7mm} \mathbb{R}^2_+ ,\\
W(x,0) &= u(\cdot,0) \hspace{6.6mm} \textup{on} \hspace{6.2mm} \mathbb{R}.
\end{aligned}
\right.
\]
With a more descriptive notation
\begin{equation}
\label{11.1}
W(x,y) = E[u(\cdot,0), \mathbb{R}^2_+](x,y),
\end{equation}
which means that W is the harmonic extension in $\{(x,y)\hspace{1mm}:\hspace{1mm} y >0\}$ that takes boundary value $W(x,0) = u(\cdot,t)$. We know that the standard fractional Laplacian for $s =1/2$ is defined as
\[
(-\Delta)^s u(x,0) = - \partial_y W(x,0).
\]
Then, since for $x < \xi_o t$, $u(x,t)$ satisfies a fractional heat equation we obtain that
\begin{equation}
\label{derivada}
\partial_y W(x,0) = \partial_t u(x,0) = \partial_t W(x,0),
\end{equation}
i.e., it is the same to take the derivative in $y$ or in $t$. This is an important property of the solutions in the critical case.

A simple inspection at the domains and the extension definition show that
\[
E[W(x,y_1), \mathbb{R}^2_+](x,y) = E[u(\cdot,0), \mathbb{R}^2_+](x,y+y_1).
\]
Moreover, if we define the family of extension functions $W(x,y,t) = E[\partial_t u, \mathbb{R}^2_+](x,y)$ we have
\[
W(x,y,t) = E[u(\cdot,0), \mathbb{R}^2_+](x,y+t).
\]
Here, $t > 0$ is a parameter that points at the initial level height where the initial data are taken.

\subsection{Proof of Theorem \ref{critical}}
The idea of the proof is to construct a subsolution of \eqref{stefan}-\eqref{IC} using the Caffarelli-Silvestre extension  \cite{CaSi} and the comparison principle.

\begin{proof}
    In order to prove a non-degenerate contact from the right we find a lower barrier or subsolution. We will consider a curvilinear domain in space-time strictly contained in the support of $u(x,t)$ which is
\[
D =\{(x,t) \hspace{0,5mm}:\hspace{0,5mm}  x < \xi_0, \hspace{0,5mm}t > 0\}.
\]
Let that domain be limited by a bottom and two side lines
\[
Q := \{(x,t) \hspace{0,5mm}:\hspace{0,5mm} t > \tau, \hspace{0,5mm}|x| < L(t) = a + b(t-\tau)\} \subset D,
\]
where $0 < \tau < 1$ and $a,b > 0$ such that $a + b(1-\tau) = \xi_0$.  That means that $Q$ and $D$ touch at the end time. In that domain we consider the solution $\tilde{u}$ obtained as the harmonic extension to $Q$ of a very smooth and small $\phi \geq 0$ as datum for $t = \tau$. After $t = 1$ we continue expanding $D$ and we put zero on the lateral boundaries. Moreover, $\tilde{u}$ is continued as zero for $|x| > L(t)$.

Let $w(x,y,t) = \tilde{u}(x,t+y)$. We call it the family of $Q$ extensions of $\tilde{u}$ with parameter $t$. We observe that with that definition, for every point in $Q$ we will have
\begin{equation}
\label{11.2}
\partial_t \tilde{u}(x_1,t_1) = \partial_y w(x_1, 0, y =t_1).
\end{equation}
This property will be very important to compare $W(x,y,t)$ and $w(x,y,t)$.

It is clear from continuity that $\tilde{u}(\cdot,t)$ is strictly less than $u(\cdot,t)$ in its support $Q(t)$ for times $t$ near the initial $\tau$. We want to compare $\tilde{u}$ and $u$, for it we argue by contradiction. We concentrate on the first point at with $\tilde{u}$ and $u$ touch at a point where they are not zero, and we discuss on the possibility that the coincidence takes place at time $t_1 < 1$. By assumption for $\tau < t < t_1$ we have
\[
\tilde{u}(x,t) < u(x,t) \qquad \forall x \in Q(t)
.\]
Hence, by the properties of the harmonic extensions, for such $t$ we have
\[
w(x,y,t) < W(x,y,t) \qquad \forall x \in Q(t) \quad \textup{and} \quad y > 0.
\]
In the limit we get for all $x \in Q(t)$, $y > 0$ that
\[
w(x,y,t_1) \leq W(x,y,t_1).
\]

Let the coincidence point at $t = t_1$ be $x_1$. Near that point we have classical solutions and it follows from first coincidence that
\[
\partial_t u(x_1,t_1) \leq \partial_t \tilde{u}(x_1,t_1).
\]
On the other hand, by \eqref{derivada} and the definition of $w(x,y,t)$ we have that
\[
\partial_y W(x_1,0,t_1) \leq \partial_y w(x_1,0,t_1).
\]
Let us now look at the function
\[
F(x,y) := W(x,y,t_1) - w(x,y,t_1) \geq 0.
\]
It is a nonnegative and locally harmonic function and it is equal to zero at $y = 0$ and $x=x_1$ by assumptions. The Hopf principle for harmonic functions at minimal points at the boundary implies that
\[
\partial_y F(x_1,0) > 0,
\]
i.e.,
\[
\partial_y W(x_1,0,t_1) > \partial_y w(x_1,0,t_1),
\]
that is a contradiction. This proves that there is strict comparison for all $t_1 < 1$. In the limit for $t = 1$ we have that
\[
\tilde{u}(x,1) \leq u(x,1) \qquad \forall x \in Q(1).
\]
This set includes the free boundary point as a border point. Thus, \eqref{11.1} is a consequence of the Hopf lemma for the harmonic function $\tilde{u}(x,t)$ at the lateral boundary of the domain $Q$. This concludes the proof of the Theorem.
\end{proof}

%%%%%%%%%%%%%%%%%%%%%%%%%%%%%%%%%%%%%%%%%%%%%%%%%%%%%%%%%%%%%%%%%%%%%%%%%%%%%%%%%%%%
\section{Critical case. Behaviour at $-\infty$}
\label{twuelve}
In this section we want to study the tail behaviour of the unique solution of the problem \eqref{stefan}-\eqref{IC} when $s = 1/2$. The main result of the section reads as follows:
\begin{thm}
    Let $s =1/2$ and $H$ the unique selfsimilar solution of \eqref{stefan}-\eqref{IC}. Then, for $\xi \ll \xi_0$
        \begin{equation}
        \label{sub11}
        (L+P_1) - H(\xi) \approx P_1/|\xi|,
        \end{equation}
        and
        \begin{equation}
        \label{sub22}
        |H'(\xi)| \approx P_1/|\xi|^{2}.
        \end{equation}
 Similar  behaviour at $+\infty$.
\end{thm}

\nc
\begin{proof}
    The proof of \eqref{sub11} is the same as in the Lemma \ref{lemma 9.2}. Then we only have to prove \eqref{sub22}.  We divide the proof in several steps:\\

    \noindent\underline{\textit{Step 1}}: Firstly, we want to prove that
    \begin{equation}
    \label{12.4}
\lim_{\xi \to - \infty} \xi H'(\xi) = 0.
\end{equation}
that can be written as $H'(\xi)=o(|\xi|^{-1})$ and means that
\begin{equation}
    \label{12.5}
    |H'(\xi)| \leq \frac{\varepsilon(\xi)}{|\xi|}  \qquad \textup{for all} \qquad \xi \ll \xi_0,
\end{equation}
where $\varepsilon(\xi) > 0$ is a constant that depends on $\xi$ and that satisfies $\varepsilon(\xi)\to 0 $ when $\xi \to - \infty$.

Let $\xi \ll \xi_0$. In order to  proving \eqref{12.4} the idea is to split the integral in three pieces. Thanks to Theorem \ref{teo 4}-$a)$ we have
\begin{equation}
    \label{12.6}
    |\xi H'(\xi)| = \left| (-\Delta)^s U(\xi)\right| = \left|\int_{\mathbb{R}} \frac{U(\xi) - U(\eta)}{|\xi-\eta|^{2}}\hspace{1mm}d\eta\right| \leq |I + II + III|,
\end{equation}
where
\begin{align*}
    & I := \int_{-\infty}^{\xi -d}\frac{U(\xi) - U(\eta)}{|\xi - \eta|^{2}} \hspace{1mm}d\eta,\\
        &  II := \int_{\xi-d}^{\xi +d}\frac{U(\xi) - U(\eta)}{|\xi - \eta|^{2}} \hspace{1mm}d\eta,\\
        & III := \int_{\xi + d}^{+ \infty}\frac{U(\xi) - U(\eta)}{|\xi - \eta|^{2}} \hspace{1mm}d\eta,
\end{align*}
with $d$ a value that will be determined later. Notice that since by Theorem \ref{teo 4}-$d)$ the function $H$ is nonincreasing we have $U(\eta) \geq U(\xi)$ $\forall \eta \in (-\infty, \xi]$. Then we can bound \eqref{12.6} by
\begin{align}
    \label{12.7}
    |\xi H'(\xi)| &= \left| (-\Delta)^s U(\xi)\right| = \left|\int_{\mathbb{R}} \frac{U(\xi) - U(\eta)}{|\xi-\eta|^{2}}\hspace{1mm}d\eta\right|\\
    &\leq |I + II + III| \leq  |II + III| \leq ( |II| + |III|).\nonumber
\end{align}
Now we study each integral separately.

For the integral $II$, by Theorem 5.2 of \cite{3} we know that $U \in C^{1,\alpha}(-\infty,\xi_0)$ for some $\alpha\in (0,1)$ uniformly away from $\xi_0$. On the other hand, since $\int_{-d}^{d} \frac{U'(\xi)\eta}{|\eta|^2}\hspace{1mm}d\eta = 0$ we have
\begin{align}
    \label{12.8}
    |II| & =\left|\int_{\xi -d}^{\xi +d} \frac{U(\xi) - U(\eta)}{|\xi - \eta|^{2}} \hspace{1mm}d\eta\right| = \left|\int_{-d}^{d} \frac{U(\xi + z) - U(\xi)}{|z|^{2}} \hspace{1mm}dz\right| = \left|\int_{-d}^{d} \frac{U(\xi + \eta) - U(\xi)}{|\eta|^{2}} \hspace{1mm}d\eta\right|\\
    &= \left|\int_{-d}^{d} \frac{U(\xi + \eta) - U(\xi) - U'(\xi)\eta}{|\eta|^{2}} \hspace{1mm}d\eta\right|\leq \int_{-d}^{d} \frac{|U(\xi + \eta) - U(\xi) - U'(\xi)\eta|}{|\eta|^{2}} \hspace{1mm}d\eta \nonumber\\
    &\leq C_d\int_{-d}^{d}\frac{|\eta|^{1+\alpha}}{|\eta|^2}\hspace{1mm}d\eta = \frac{2C_d}{\alpha}d^\alpha\nonumber
\end{align}
where we have used that if $\xi \ll \xi_0$ then
\[
|U(\xi + \eta) - U(\xi) - U'(\xi)\eta| \leq C_d|\eta|^{1+\alpha},
\]
where $C_d >0$ is the Lipschitz constant that appears in  Theorem 5.2 of \cite{3}. Note that this constant does not depend on the point $\xi$.

Now, we study the integral $III$. Splitting the integral in two pieces we have
\begin{align}
    \label{12.9}
    |III| &= \left|\int_{\xi + d}^{z}\frac{U(\xi) - U(\eta)}{|\xi - \eta|^{2}} \hspace{1mm}d\eta + \int_{z}^{+ \infty}\frac{U(\xi) - U(\eta)}{|\xi - \eta|^{2}} \hspace{1mm}d\eta \right|\\
    &\leq \left|\int_{\xi + d}^{z}\frac{U(\xi) - U(\eta)}{|\xi - \eta|^{2}} \hspace{1mm}d\eta\right| + \left|\int_{z}^{+ \infty}\frac{U(\xi) - U(\eta)}{|\xi - \eta|^{2}} \hspace{1mm}d\eta \right| = III_i + III_{ii}\nonumber,
\end{align}
where $z = \lambda \xi$ with $\lambda < 1$ a constant that will depend on d, and
\begin{align}
\label{12.10}
& III_i := \left|\int_{\xi + d}^{z}\frac{U(\xi) - U(\eta)}{|\xi - \eta|^{2}} \hspace{1mm}d\eta\right|,\\
\label{12.11}
& III_{ii} := \left|\int_{z}^{+ \infty}\frac{U(\xi) - U(\eta)}{|\xi - \eta|^{2}} \hspace{1mm}d\eta \right|.
\end{align}
Firstly, we study the integral $III_i$. Let $\Tilde{\xi}_1 \ll \xi_0$ and $\lambda \in (0,1)$ such that $\forall \xi \in (- \infty, \Tilde{\xi}_1]$
\begin{equation}
\label{12.12}
\xi \leq \xi + d \leq z = \lambda \xi,
\end{equation}
and
\begin{equation}
\label{12.13}
|U(\xi) - U(\eta)| \leq \delta \qquad  \textup{for all} \qquad \xi, \eta \in [-\infty, \lambda \Tilde{\xi}_1].
\end{equation}
Notice that by Theorem \ref{teo 4}-$b)$, $\lim_{\xi \to \infty} H(\xi) = P_1$, then there exists $\Tilde
{\xi}_1$ large enough such that \eqref{12.13} holds.
Using the above information we deduce
\begin{align}
\label{12.14}
    III_i &\leq \int_{\xi + d}^{z}\frac{|U(\xi) - U(\eta)|}{|\xi - \eta|^{2}} \hspace{1mm}d\eta \leq \delta \int_{\xi + d}^{z}\frac{1}{|\xi - \eta|^{2}} \leq \frac{\delta}{2s}\left(\frac{1}{d} + \frac{1}{|\xi-z|}\right)\\
    & \leq \delta\left(\frac{1}{d^{2s}} + \frac{1}{|\xi-\lambda \xi|^{2s}}\right) = \frac{\delta}{(1-\lambda)}\left(\frac{(1-\lambda)}{d} + \frac{1}{|\xi|}\right)\nonumber.
\end{align}
On the other hand, taking $\xi \ll \xi_0$, we can bound $III_{ii}$ by
\begin{equation}
    \label{12.15}
    III_{ii} \leq 2 \norm{U}_{L^{\infty}(\mathbb{R})}\int_{z}^{+ \infty} \frac{1}{|\xi - \eta|^{2}}\hspace{1mm} d\eta = 2 \norm{U}_{L^{\infty}(\mathbb{R})}\frac{1}{|\xi - z|} = \frac{2 \norm{U}_{L^{\infty}(\mathbb{R})}}{1-\lambda} \frac{1}{|\xi|}.
\end{equation}
Substituting \eqref{12.14} and \eqref{12.15} in \eqref{12.9}, we obtain
\begin{equation}
    \label{12.16}
    III \leq \frac{\delta}{1-\lambda}\left(\frac{1}{d} + \frac{1}{|\xi|}\right) + \frac{2 \norm{U}_{L^{\infty}(\mathbb{R})}}{1-\lambda}\frac{1}{|\xi|}.
\end{equation}

Thanks to \eqref{12.8} and \eqref{12.16} we can bound \eqref{12.7} by
\begin{align}
    \label{12.17}
    |\xi H'(\xi)| &\leq \frac{2C}{1-\alpha}d^\alpha + \frac{\delta}{1-\lambda}\left(\frac{1}{d} + \frac{1}{|\xi|}\right) + \frac{2 \norm{U}_{L^{\infty}(\mathbb{R})}}{1-\lambda}\frac{1}{|\xi|}\\
    & \leq C\left( d^{\alpha} + \frac{\delta}{d} + \frac{1}{|\xi|}\right),\nonumber
\end{align}
with $C > 0$ a constant that depends on $s$, $\lambda$, $C_s$ and $\norm{U}_{L^{\infty}(\mathbb{R})}$. The idea now is to optimize the two errors that we have in \eqref{12.17}. For it we define the function
\[
F(d) := \frac{\delta}{d} + d^{\alpha}.
\]
Taking the derivative we have
\[
F'(d) = -\frac{\delta}{d^2} + \alpha\frac{1}{d^{1-\alpha}}.
\]
Studying the singular points we obtain a minimum when
\[
d = \sqrt[1+\alpha]{\frac{\delta}{\alpha}}
\]
 Substituting this value of $d$ in \eqref{12.17} we obtain
\begin{equation}
    \label{3.17}
    |\xi H'(\xi)| \leq C\left(\delta^{\frac{\alpha}{1+\alpha}} + \frac{1}{|\xi|}\right),
\end{equation}
where we have redefined the constant $C$.

Finally to conclude the proof of \eqref{12.4} we argue as follows. For $\varepsilon > 0$ fix, we choose $\tilde{\xi}_1$ and $\lambda$ such that \eqref{12.12} and \eqref{12.13} holds for
\[
\delta = \sqrt[\frac{1+\alpha}{\alpha}]{\frac{\varepsilon}{2C}}, \quad \textup{and} \quad d = \sqrt[1+\alpha]{\frac{\delta}{\alpha}}.
\]
On the other hand, we choose $\tilde{\xi}_2 \ll \xi_0$ such that
\[
\frac{C}{|\xi|} \leq \frac{\varepsilon}{2} \qquad \textup{for all} \qquad \xi \in (-\infty, \Tilde{\xi}_2].
\]
Finally, taking  $\Tilde{\xi} := \textup{min}\{\Tilde{\xi}_1, \Tilde{\xi}_2\}$
we conclude that \nc
\[
|\xi H'(\xi)| \leq \varepsilon \qquad \textup{for all} \qquad \xi \in (-\infty, \Tilde{\xi}].
\]
This proves \eqref{12.4}.

\medskip

\noindent\underline{\textit{Step 2}}:  Note that by the Step 1 we have that
\begin{equation}
\label{12.19}
|H'(\xi)| \leq \frac{\varepsilon_\xi}{|\xi|} \qquad \textup{forall} \qquad \xi \ll \xi_0,
\end{equation}
where $\varepsilon_\xi > 0$ is a constant that depends on $\xi$ and that satisfies  $\varepsilon_\xi \to 0$ when $\xi \to -\infty$. The idea in this step is to use this information in the integral $II$ to improve the bound of $H'$. Let $r > 0$ a enough small constant  that will be determined later. We split the integral $II$ in two pieces:
\begin{equation}
    \label{12.20}
    |II| = \left|\int_{B_d(\xi) \setminus B_r(\xi)}\frac{U(\xi) - U(\eta)}{|\xi - \eta|^{2}} \hspace{1mm}d\eta + \int_{B_r(\xi)}\frac{U(\xi) - U(\eta)}{|\xi - \eta|^{2}} \hspace{1mm}d\eta\right| = |II_i + II_{ii}| \leq |II_i| + |II_{ii}|,
\end{equation}
where
\begin{align*}
    & II_i := \int_{B_d(\xi) \setminus B_r(\xi)}\frac{U(\xi) - U(\eta)}{|\xi - \eta|^{2}} \hspace{1mm}d\eta,\\
        &  II_{ii} := \int_{B_r(\xi)}\frac{U(\xi) - U(\eta)}{|\xi - \eta|^{2}} \hspace{1mm}d\eta.
\end{align*}
Now we study each integral separately.

Firstly, we study integral $|II_i|$. Thanks to the mean value Theorem and \eqref{12.19} we obtain
\begin{align}
    \label{12.21}
     |II_i| &= \left|\int_{B_d(\xi) \setminus B_r(\xi)}\frac{U(\xi) - U(\eta)}{|\xi - \eta|^{2}} \hspace{1mm}d\eta\right| \leq \int_{B_d(\xi) \setminus B_r(\xi)}\frac{|U(\xi) - U(\eta)|}{|\xi - \eta|^{2}} \hspace{1mm}d\eta \\
     & = \int_{B_d(\xi) \setminus B_r(\xi)}\frac{|U'(x_o)||\xi - \eta|}{|\xi - \eta|^{2}} \hspace{1mm}d\eta \leq \frac{\varepsilon_\xi}{|\xi|}\int_{B_d(\xi) \setminus B_r(\xi)} \frac{1}{|\xi - \eta|}\hspace{1mm}d\eta = \frac{2\varepsilon_\xi}{|\xi|}\ln{\frac{d}{r}}\nonumber.
\end{align}
where $x_o \in (\xi -d, \xi + d)$.
On the other hand, since $B_r(\xi)$ is a symmetric domain, we can argue like in the Step 1 to obtain
\begin{equation}
    \label{12.22}
    |II_{ii}| \leq \int_{-r}^{r} \frac{|U(\xi + \eta) - U(\xi) - U'(\xi)\eta|}{|\eta|^{2}} \hspace{1mm}d\eta = \frac{2C_d}{\alpha}r^\alpha.
\end{equation}
Finally, arguing the same as in the Step 1 in the integral $III$ and applying the bounds \eqref{12.21} and \eqref{12.22} we can bound \eqref{12.7} by
\begin{equation}
    \label{12.23}
    |\xi H'(\xi)| \leq C\left(\frac{1}{|\xi|}\ln{\left(\frac{d}{r}\right)} + r^\alpha +\frac{\delta}{d} + \frac{1}{|\xi|} \right),
\end{equation}
where $C > 0$ is a constant that depends on $\lambda$, $\varepsilon_\xi$, $\norm{U}_{L^\infty(\mathbb{R})}$, $\alpha$, $\delta$ and $C_d$. Taking $r = \frac{1}{|\xi|^{\frac{1}{\alpha}}}$ in \eqref{12.23} we obtain
\begin{equation}
    \label{12.24}
      |\xi H'(\xi)| \leq C\left(\frac{1}{|\xi|}\ln{(|\xi|^{\frac{1}{\alpha}}d)} + \frac{1}{|\xi|} + \frac{\delta}{d} + \frac{1}{|\xi|}\right) \leq  C\left(\frac{1}{|\xi|}\ln{(|\xi|^{\frac{1}{\alpha}}d)} + \frac{\delta}{d} + \frac{1}{|\xi|}\right),
\end{equation}
for $\xi \ll \xi_0$. Optimizing the error that we  have in \eqref{12.24} \nc we obtain a minimum when $d = \delta |\xi|$, then
\[
|\xi H'(\xi)| \leq C\left(\frac{\ln{(\delta |\xi|^{\frac{1 + \alpha}{\alpha}})}}{|\xi|} + \frac{2}{|\xi|}\right).
\]
Therefore, since
\[
\frac{1}{|\xi|} \leq \frac{\ln{(\delta |\xi|^{\frac{1+\alpha}{\alpha}})}}{|\xi|} \qquad \textup{for all} \qquad \xi \ll \xi_0,
\]
we conclude that
\begin{equation}
  \label{12.25}
|H'(\xi)|\leq C \frac{\ln{(\delta |\xi|^{\frac{1+\alpha}{\alpha}})}}{|\xi|^2} \qquad \textup{for all} \qquad \xi \ll \xi_0,
\end{equation}
for some $\alpha \in (0,1)$. \\

\noindent\underline{\textit{Step 3}}:
Finally in this step we want to use the bound \eqref{12.25} to prove the estimate \eqref{sub22}. We use the same idea that in the Step 2, splitting the integral $II$ in to integrals and applying the bound \eqref{12.25} to the integral $II_i$. Then we have
\begin{align}
\label{12.26}
    |II_i| &= \left|\int_{B_d(\xi) \setminus B_r(\xi)}\frac{U(\xi) - U(\eta)}{|\xi - \eta|^{2}} \hspace{1mm}d\eta\right| \leq \int_{B_d(\xi) \setminus B_r(\xi)}\frac{|U(\xi) - U(\eta)|}{|\xi - \eta|^{2}} \hspace{1mm}d\eta \\
     & = \int_{B_d(\xi) \setminus B_r(\xi)}\frac{|U'(x_o)||\xi - \eta|}{|\xi - \eta|^{2}} \hspace{1mm}d\eta \leq C \frac{\ln{(\delta |\xi|^{\frac{1+\alpha}{\alpha}})}}{|\xi|^2} \int_{B_d(\xi) \setminus B_r(\xi)} \frac{1}{|\xi - \eta|}\hspace{1mm}d\eta\nonumber\\
     &= C \frac{\ln{(\delta |\xi|^{\frac{1+\alpha}{\alpha}})}}{|\xi|^2} \ln{\frac{d}{r}}\nonumber.
\end{align}
where $x_o \in (\xi-d, \xi+d)$. Arguing the same as in the Step 2 with the integrals $II_{ii}$ and $III$ and thanks to \eqref{12.26}, we can bound \eqref{12.6} by
\begin{equation}
    \label{12.27}
    |\xi H'(\xi)| \leq C\left(\frac{\ln{(\delta |\xi|^{\frac{1+\alpha}{\alpha}})}}{|\xi|^2} \ln{\left(\frac{d}{r}\right)} + r^\alpha +\frac{\delta}{d} + \frac{1}{|\xi|} \right),
\end{equation}
where $C > 0$ is a constant that depends on $\lambda$, $\varepsilon_\xi$, $\norm{U}_{L^\infty(\mathbb{R})}$, $\alpha$, $\delta$ and $C_d$. Taking $r = \frac{1}{|\xi|^{\frac{1}{\alpha}}}$ in \eqref{12.27} we obtain
\begin{equation}
    \label{12.28}
      |\xi H'(\xi)| \leq C\left(\frac{\ln{(\delta |\xi|^{\frac{1+\alpha}{\alpha}})}}{|\xi|^2} \ln{(|\xi|^{\frac{1}{\alpha}}d)} + \frac{\delta}{d} + \frac{1}{|\xi|}\right),
\end{equation}
for $\xi \ll \xi_0$. Optimizing the error that we have in \eqref{12.28} we have a minimum when $d = \delta |\xi|$ then since
\[
|\xi H'(\xi)| \leq C\left(\frac{\ln^2{(\delta |\xi|^{\frac{1 + \alpha}{\alpha}})}}{|\xi|^2} + \frac{1}{|\xi|}\right).
\]
Therefore, since
\[
\frac{\ln^2{(\delta |\xi|^{\frac{1+\alpha}{\alpha}})}}{|\xi|^2} \leq \frac{1}{|\xi|}\qquad \textup{for all} \qquad \xi \ll \xi_0,
\]
we conclude that
\begin{equation*}
|H'(\xi)|\leq C \frac{1}{|\xi|^2} \qquad \textup{for all} \qquad \xi \ll \xi_0,
\end{equation*}
for some $\alpha \in (0,1)$. This conclude the proof.

\end{proof}

%\newpage

%%%%%%%%%%%%%%%%%%%%%%%%%%%%%%%%%%%%%%%%%%%%%%%%%%%%%%%%%%%%%%%%%%%%%%%%%%%%%%%%%%%%%%%%

\section{Supercritical case. Limited regularity}
\label{section 13}

In this section we study the regularity in the supercritical case, $s\in (1/2,1)$. We prove that the solution of \eqref{stefan}-\eqref{IC} is not $C^1$ proving that near the free boundary point $\xi_0$ the temperature is bounded below by a function that is linear near $\xi_0$.

\subsection{The wedge. Linear lower bound on the left side of $\xi_0$}

In this section we prove the existence of a lower linear barrier to the right of the free boundary.
\begin{prop}
    Let $s \in (\frac{1}{2}, 1)$ and $H$ the unique selfsimilar solution of \eqref{stefan}-\eqref{IC}. Then, there exists a linear function $V(\xi)\ge 0$ such that $0< V(\xi) \leq H(\xi)$ for all $\xi \le \xi_0$.
\end{prop}

\begin{proof}
    We recall that since $U = 0$ for all $\xi \geq \xi_0$ and $H = U+L$, when $\xi \leq \xi_0$ the temperature $U$ satisfies the elliptic equation
    \begin{equation}
    \label{tipo equa}
    -\xi U'(\xi) + 2s(-\Delta)^s U(x) = 0 \qquad \xi \leq \xi_0.
    \end{equation}
    The idea of the proof is to find a linear subsolution $V$ to this problem in a interval $J=(\xi_0-a, \xi_0)$, with some small $a > 0$  in the sense that
     \begin{equation}\label{V.eq}
    -\xi V'(\xi) + 2s(-\Delta)^s V(\xi) \leq 0 \qquad \forall \ \xi\in J,
     \end{equation}
     and the exterior condition
      \begin{equation}\label{V.ext}
    V(x)\le U(x) \qquad  \mbox{for } \ \xi\in J.
    \end{equation}

    \noindent (i) We propose the candidate function $V$
    \begin{equation*}
    V(\xi):=\left\{
    \begin{aligned}
     &C \hspace{0.5mm}\xi_0 \hspace{16.5mm}\textup{if} \hspace{7mm} \xi < 0, \\
    &C (\xi_0 -\xi)_+ \hspace{5mm}\textup{if} \hspace{7mm}\xi \geq 0,
    \end{aligned}
    \right.
    \end{equation*}
    We choose  $C > 0$ a  small enough constant such that the continuous function $U(\xi)$ is larger than $V(\xi)$ in the interval $(0, \xi_0- a]$, with $a > 0$ fixed but small (in particular $a< \xi_i/2$). In this way the exterior condition \eqref{V.ext} will be met. Note that $V>0$ for $\xi <\xi_0$ and $V=0$ for  $\xi\ge\xi_0$.

  \noindent (ii) In  order to check \eqref{V.eq} we study the integral
    \begin{equation}
    \label{lapla fracciona}
    (-\Delta)^s V(\xi) = \int_{\mathbb{R}} \frac{V(\xi)- V(\eta)}{|\xi-\eta|^{1+2s}} \hspace{1mm}d\eta = I + II + III+ IV,
    \end{equation}
    we estimate the four subintegrals for $\xi\in J.$
    \begin{align*}
        & I := \int_{- \infty}^{0} \frac{V(\xi)- V(\eta)}{|\xi-\eta|^{1+2s}} \hspace{1mm}d\eta,\\
        & II := \int_{0}^{\xi-k} \frac{V(\xi)- V(\eta)}{|\xi-\eta|^{1+2s}} \hspace{1mm}d\eta,\\
        & III :=  \int_{\xi-k}^{\xi+h} \frac{V(\xi)- V(\eta)}{|\xi-\eta|^{1+2s}} \hspace{1mm}d\eta,\\
        & IV := \int_{\xi+k}^{\infty} \frac{V(\xi)- V(\eta)}{|\xi-\eta|^{1+2s}} \hspace{1mm}d\eta.
    \end{align*}

    We take  $k=\xi_0-\xi<a  $. We study each integral separately. For $I$ we have
    \begin{align}
        \label{I}
        I &=  \int_{- \infty}^{0} \frac{V(\xi)- V(\eta)}{|\xi-\eta|^{1+2s}} \hspace{1mm}d\eta =  \int_{- \infty}^{0} \frac{V(\xi)- C\hspace{0.5mm}\xi_0}{|\xi-\eta|^{1+2s}} \hspace{1mm}d\eta =  -\int_{\xi}^{+ \infty} \frac{C\hspace{0.5mm}\xi_0- V(\xi)}{|z|^{1+2s}} \hspace{1mm}dz\\
        &= \frac{C\hspace{0.5mm}\xi_0 - V(\xi)}{2s} \left[z^{-2s}\right]_\xi^{+\infty} = - \left(\frac{C\hspace{0.5mm}\xi_0 - V(\xi)}{2s}\right)\frac{1}{\xi^{2s}}.\nonumber
    \end{align}
    This integral gives a negative bounded contribution for different $\xi$ in $J$.     Now, we study $II$
    \begin{align}
        \label{II}
        II &= \int_{0}^{\xi-k} \frac{V(\xi)- V(\eta)}{|\xi-\eta|^{1+2s}} \hspace{1mm}d\eta =
        \int_{0}^{\xi-k} \frac{C(\eta - \xi)}{|\eta-\xi|^{1+2s}}\hspace{1mm}d\eta =
         - C\int_{0}^{\xi-k} \frac{d\eta}{(\xi-\eta)^{2s}}\hspace{1mm}\\
        &= \frac{C}{(2s-1)}\left(\frac{1}{\xi^{2s-1}}- \frac{1}{k^{2s-1}}\right).\nonumber
    \end{align}
    This negative integral diverges when $k$ tends to 0, the leading term is proportional to $k^{-(2s-1)}$.
   For integral $III$  we obtain
    \begin{equation}
        \label{III}
        III = \int_{\xi-k}^{\xi+k} \frac{V(\xi)- V(\eta)}{|\xi-\eta|^{1+2s}} \hspace{1mm}d\eta = \int_{\xi-k}^{\xi+k} \frac{C(\eta-\xi)}{|\xi-\eta|^{1+2s}}\hspace{1mm}d\eta = 0,
    \end{equation}
    by cancellation of left side and right side.   Recalling that $\xi+k=\xi_o$ we get the last integral gives
    \begin{align}
        \label{IV}
        IV = \int_{\xi+k}^{\infty} \frac{V(\xi)- V(\eta)}{|\xi-\eta|^{1+2s}} \hspace{1mm}d\eta = \int_{\xi_0}^{\infty}\frac{C|\xi_0 - \xi|-0}{|\xi-\eta|^{1+2s}}\hspace{1mm}d\eta = \int_{k}^{\infty} \frac{Ck}{z^{1+2s}}\hspace{1mm}dz = \frac{C}{2s}\frac{1}{k^{2s-1}},
    \end{align}
    This positive integral is singular when $k\to 0$, i.e., $\xi\to\xi_o$ with leading term  $O(k^{-(2s-1)})$. Since we want $(-\Delta)^s V(\xi) $ to be negative in $J$ we compare the two leading terms and find that
    $$
    -\frac{C}{(2s-1)}\frac{1}{k^{2s-1}} + \frac{C}{2s}\frac{1}{k^{2s-1}} =-\frac{C}{2s(2s-1)}\frac{1}{k^{2s-1}},
    $$
Finally, substituting \eqref{I}, \eqref{II}, \eqref{III} and \eqref{IV} into \eqref{lapla fracciona} we obtain the inequality for the $V$ equation
    \begin{equation}
        \label{fracc final}
        (-\Delta)^s V(\xi) = -\left(\frac{C\hspace{0.5mm}\xi_0 - V(\xi)}{2s}\right)\frac{1}{\xi^{2s}} + \frac{C}{(2s-1)}\frac{1}{\xi^{2s-1}}  -\frac{C}{2s(2s-1)}\frac{1}{k^{2s-1}},
    \end{equation}
    as needed.

\noindent (iii)     Since \ $-\xi V'(\xi) = C \xi \le C\xi_0$   \ in $(\xi_0-a,\xi_0)$  we obtain by \eqref{fracc final} that
    \begin{equation}
    \label{final final}
    -\xi V'(\xi) + 2s (-\Delta)^s V(\xi) < C\,\left(\xi_0 + \frac{2s}{(2s-1)}\frac{1}{\xi^{2s-1}}- \frac{1}{(2s-1)}\frac{1}{k^{2s-1}}\right) < 0,
    \end{equation}
    for all $\xi \in J=(\xi_0-a,\xi)$. The negative sign depends only on the fact that $a$ is small enough. We conclude by the standard comparison principle for the nonlocal elliptic equation \eqref{tipo equa} that $V(\xi) \leq U(\xi)$ for all $\xi \in J$. Since the comparison outside of $J$ was chosen since the very beginning, we get  \    $\xi \in \mathbb{R}$. This concludes the proof of the linear lower bound (the wedge) on the left side of $\xi_0$.
    \end{proof}

    %%%%%%%%%%%%%%%%%%%%%%%%%%%%%%%%%%%%%%%%%%%%%%%%%%%%%%%%%%%%%%%%%%%%%%

\subsection{Bad behaviour  on the right side of $\xi_0$}\label{ssec.supcritreg}

 We use the same argument of Subsection \ref{ssec.supcritreg12} but now with $s>1/2$.  What have proved that the temperature $u(x,t)$ approaches the free boundary point from the left in a non-degenerate way, i.e., in terms of the selfsimilar profile we have
\begin{equation}
\label{super criti}
U(\eta) \geq C(\xi_0 - \eta) \qquad \forall  \eta\in (\xi_0/2, \xi_0).
\end{equation}
Since
\[
-H'(\xi_0^+) = \frac{1}{\xi}\int_{-\infty}^{\xi_0}\frac{U(\eta)}{|\xi_0-\eta|^{1+2s}} \hspace{1mm}d\eta,
\]
then by \eqref{super criti} we have for every  $\xi_1=\xi_0+\varepsilon$,  $\varepsilon>0$ small:
\[
-H'(\xi_1) = \frac{1}{\xi_1}\int_{-\infty}^{\xi_0}\frac{U(\eta)} {|\xi_1-\eta|^{1+2s}}\hspace{1mm}d\eta \geq \frac{C}{\xi_1}\int_{\xi_0/2 }^{\xi_0- \varepsilon} \frac{|\xi_0 - \eta|}{|\xi_1 - \eta|^{1+2s}}\hspace{1mm}d\eta.
\]
But in the integration interval we have $|\xi_0 - \eta|\ge 1/2|\,\xi_1 - \eta|$ so  that
than $\xi_0$ we have
\[
-H'(\xi_1) \ge \frac{C_1}{\xi_1}\int_{\xi_0/2 }^{\xi_0- \varepsilon} \frac{1}{|\xi_1 - \eta|^{2s}}\hspace{1mm}d\eta=
\frac{C_2}{\xi_1}\left(\frac{1}{\varepsilon^{2s-1}}-\frac{1}{(\xi_0/2- \varepsilon)^{2s-1}}\right)
\]
For $\varepsilon=\xi_1-\xi_0>0$ much smaller than $\xi_0/2$ we get
\[
-H'(\xi_1) \ge \frac{C_2}{\xi_1}\,(\xi_1-\xi_0)^{1-2s}.
\]
This ends the proof of Theorem \ref{supercritical}. \qed

\medskip

\noindent {\bf Remark.} We can also use the proof of estimate \eqref{est.H''} to get a lower estimate of $H''$  for $s\ge 1/2$ once we have established the  condition the linear lower bound for $U$ to the left of $\xi_0$.
%%%% Nota: \color{magenta} Anhadido en la version de 27 oct 25

\subsection{More on regularity}\label{ssec.supcritreg2}

We do not know the optimal regularity of the temperature on the left of the free boundary. As minimal behaviour we have the linear wedge subsolution. This would be a very good result if true as an upper bound, For moment we have the regularity proved in Lemma 3.15 of \cite{3}.

\begin{lem}
\label{lemma Cs}
Under our current assumptions there is a constant $C=C_s>0$ such that
\begin{equation}
U(\xi) \le C(\xi_0 - \xi)^s  \quad \mbox{for all $\xi \le  \xi_0.$}
\end{equation}
\end{lem}

So the regularity of the temperature on the water side is estimated somewhere between $C^s$ and $C^1$.
The obtained regularity on the water side of the free boundary can be translated to the ice side working as before:

\begin{prop} Let $s\in  ( 1/2,1) $ and assume that $U(\xi) \le C(\xi_0 - \xi)^{\alpha}$  for all $\xi \le  \xi_0.$  with some $s\le \alpha \le 1$. Then
$H \in C^{\gamma}([\xi_0, \infty))$ with $\gamma= 1+\alpha - 2s$. For $\alpha=s$ we get $\gamma=1-s$, for $\alpha=1$ we get $\gamma=2-2s$.
\end{prop}

The proof is obtained working as before.

%%%%%%%%%%%%%%%%%%%%%%%%%%%%%%%%%%%%%%%%%%%%%%%%%%%%%%%%%%%%%%%%%%%%%%%%%%%%%%%%%%%%%%%%%%%%%%%%%%%%%%%%

\section{Open problems}\label{sec.op}

We comment on some open problems which are relevant in view of  our work.

\begin{enumerate}

\item As shown in Theorem \ref{principal} and Proposition \ref{lemma 1}, we get $H\in C^{1,\alpha}(\mathbb{R})$ but $H \notin C^2(\mathbb{R})$ if $0<s<1/2$. A first natural question is: can the regularity of $H$ be improved?

 We study in Section \ref{section 6} the lateral regularity of $H$ at the free boundary point. We show that near the free boundary we have  $H \in C^{1, 1-2s}$ and this regularity  is optimal towards the right, i.e. at $\xi_0^+$. Is it also optimal towards the left?      What is the optimal local regularity in $(-\infty,\xi_0)$?  \nc

\item A very important open problem is the study of the lateral regularity of the selfsimilar solutions of \eqref{stefan}-\eqref{IC} when $s \geq \frac{1}{2}$ , i.e., in the critical and supercritical cases. Are the selfsimilar solutions $C^{1}(-\infty,\xi_0]$?
    Numerical evidence displayed in \cite{DEV1} seems to indicate that  for the critical exponent $s=1/2$ the temperature is not even Lipschitz continuous at $\xi_0$\nc.

\item We have centered our study on the fine behaviour of the bounded selfsimilar solutions. Could we prove similar results for more general classes of solutions?
\end{enumerate}

%%%%%%%%%%%%%%%%%%%%%%%%%%%%%%%%%%%%%%%%%%%%%%%%%%%%%%%%%%%

\section*{Appendix I}
  Let the assume that the iteration method allows to pass from H\"older exponent $\alpha\in (0,1)$ to the exponent $\beta(\alpha)$:
$$
\beta(\alpha)=\frac{1+\alpha-2s}{1+\alpha}.
$$
The question is how far we can go with this iteration. We prove next nice upper and lower bounds for the optimal iterated $\alpha$.

\noindent 1) First, we prove that $0<d\beta/d\alpha<2s$, with end values
$$
\beta(0)= 1-2s>0, \qquad \beta(1)= 1-s.
$$
Therefore, for small $\alpha\sim 0$ we get $\beta(\alpha)> \alpha$. However, for $\alpha\sim 1$ we get $\beta \sim 1-s$ which is less than $\alpha$, so the iteration is useless near that end.

\medskip

\noindent 2) We conclude that for $\alpha> 0$ we have $\beta> 1 -2s$.

\medskip

\noindent 3) On the other hand, we have seen that for $\alpha<1$ we have $\beta < 1-s$. \\
Moreover, for $\alpha\le 1-s$ we have $\beta < 1 -2s$.
Indeed,
$$
\beta= \frac{1+\alpha-2s}{1+\alpha}=\frac{2-3s}{2-s}<1 -s,
$$
Hence, the optimal value that we get by this method is a H\"older exponent $\alpha^*$ between $1-2s$ and $1-s$. \qed \nc

%%\newpage
%%%%%%%%%%%%%%%%%%%%%%%%%
%%{\color{magenta} Include comment on two References on related results in red in the biblio.}

\section*{Acknowledgments}
The research of M. Llorca is part of his Ph.~D. project at the University Aut\'onoma de Madrid. J.~L.~V\'azquez has been partially funded by grants PID2021-127105NB-I00 and CNS2024-154515 \nc from Agencia Estatal de Investigación, AEI, of the Spanish Ministry of Science, Innovation and Universities. We are grateful to F. del Teso and D. G\'omez-Castro for fruitful discussions. Figure 1 is taken from  \cite{DEV1} with the consent of the authors.

\medskip

%%%\newpage

\

%\newpage

\noindent {\sc Addresses:}

\noindent Marcos Llorca. Departamento de Matem\'{a}ticas, \\
Universidad Aut\'{o}noma de Madrid,\\ Campus de Cantoblanco, 28049 Madrid, Spain. \nc \\
e-mail address:~\texttt{marcosllorca@gmail.com}\\

\medskip

\noindent Juan Luis V\'azquez. Departamento de Matem\'{a}ticas, \\
Universidad Aut\'{o}noma de Madrid,\\ Campus de Cantoblanco, 28049 Madrid, Spain.  \\
e-mail address:~\texttt{juanluis.vazquez@uam.es}\\
%% webpage:  https://verso.mat.uam.es/\~juanluis.vazquez/

%%%
\end{document}